\pdfoutput=1
\RequirePackage{ifpdf}
\ifpdf % We are running pdfTeX in pdf mode
\documentclass[pdftex]{sigma}
\else
\documentclass{sigma}
\fi

\numberwithin{equation}{section}

\newtheorem{Theorem}{Theorem}[section]
\newtheorem{Corollary}[Theorem]{Corollary}
\newtheorem{Lemma}[Theorem]{Lemma}
\newtheorem{Proposition}[Theorem]{Proposition}
 { \theoremstyle{definition}
\newtheorem{Definition}[Theorem]{Definition}
\newtheorem{Example}[Theorem]{Example}
\newtheorem{Remark}[Theorem]{Remark} }

\begin{document}
\allowdisplaybreaks

\newcommand{\arXivNumber}{1712.03068}

\renewcommand{\PaperNumber}{096}

\FirstPageHeading

\ShortArticleName{The Variational Bi-Complex for Systems of Semi-Linear Hyperbolic PDEs}

\ArticleName{The Variational Bi-Complex for Systems\\ of Semi-Linear Hyperbolic PDEs in Three Variables}

\Author{Sara FROEHLICH}

\AuthorNameForHeading{S.~Froehlich}

\Address{Department of Mathematics and Statistics, McGill University,\\ 805 Sherbrooke Street West, Montreal, QC H3A 0B9 Canada}
\Email{\href{mailto:sara.froehlich@mail.mcgill.ca}{sara.froehlich@mail.mcgill.ca}}

\ArticleDates{Received December 11, 2017, in final form August 24, 2018; Published online September 09, 2018}

\Abstract{This paper extends, to a class of systems of semi-linear hyperbolic second order PDEs in three variables, the geometric study of a single nonlinear hyperbolic PDE in the plane as presented in [Anderson I.M., Kamran N., \textit{Duke Math.~J.} \textbf{87} (1997), 265--319]. The constrained variational bi-complex is introduced and used to define form-valued conservation laws. A method for generating conservation laws from solutions to the adjoint of the linearized system associated to a system of PDEs is given. Finally, Darboux integrability for a system of three equations is discussed and a method for generating infinitely many conservation laws for such systems is described.}


\Classification{35L65; 35A30; 58A15}

\section{Introduction}\label{section1}

This paper belongs to the field known broadly as the geometric study of partial differential equations, which seeks to understand differential equations through the study of properties which remain invariant under particular groups of transformations. The subject has its roots in the foundational works of Lie, Darboux, Cartan and others. It was Cartan who recast partial differential equations geometrically as exterior differential systems. In doing so, the solutions to partial differential equations were realized as integral manifolds of corresponding exterior diffe\-ren\-tial systems. More recently, the field of geometric PDEs has undergone a number of important developments. These include efforts to obtain explicit solutions and solution algorithms to specific classes of PDEs \cite{afv09,sz03}, the investigation of links between PDEs and the geometry of the submanifolds which their level sets define \cite{kt96,kt01,va10}, as well as the study and computation of invariants such as conservation laws \cite{ak97,cl97,wa04}. It is this last area, conservation laws, with which this paper is concerned.

The geometric approach to conservation laws has an extensive history and literature, which we will not attempt to describe in any detail, referring the reader instead to~\cite{ol00} for a comprehensive account of the subject. We will however provide a brief, non-exhaustive overview of some of the themes and contributions that are relevant to this paper. Our approach to conservation laws finds its origin in the study of the cohomology determined by the $\mathcal{C}$-spectral sequence introduced by Vinogradov in~\cite{vi841} and~\cite{vi842}. In addition to being used to characterize conservation laws in terms of cohomology, this construction has facilitated work in other important aspects of the study of differential equations including the inverse problem of the calculus of variations (see for example \cite{at92,dt80}) and Euler--Lagrange operators as studied by Tulczyjew in~\cite{tu77}. Tsujishita and Duzhin built upon the work of Vinogradov to study conservation laws of the BBM equation~\cite{dt84}, and Tsujishita expanded the applications of the $\mathcal{C}$-spectral sequence to include topics such as the study of characteristic classes and Gel'fand--Fuks cohomology~\cite{ts82} and conservation laws of the Klein--Gordon equation~\cite{ts79}.

In \cite{an89}, Anderson gives a comprehensive treatment of the variational bi-complex, which emer\-ged out of the works on the $\mathcal{C}$-spectral sequence mentioned above, and the horizontal cohomology of which will serve as the natural framework for our present study of conservation laws. It should be pointed out that the variational bi-complex also lends itself to the study of various other topics such as the equivariant version of the inverse problem of the calculus of variations, Riemannian structures, and the method of Darboux integrability for a~scalar second order PDE. A large body of work has been amassed in the study of conservation laws by utilizing the variational bi-complex. Of particular importance for the present discussion, we note that Anderson and Kamran performed an extensive study of the conservation laws of hyperbolic scalar second order PDEs in the plane in~\cite{ak97}.

In the contemporaneous work of Bryant and Griffiths \cite{bg95,bg03}, conservation laws were studied from the distinct, yet related, perspective of the characteristic cohomology of exterior differential systems. In particular, local invariants of an exterior differential system are shown to govern the properties of the system's characteristic cohomology. This approach carries the study of the variational bi-complex and the $\mathcal{C}$-spectral sequence into the realm of exterior differential systems, where the independent and dependent variables of a system of partial differential equations are treated equally.

To demonstrate the connection between this approach and that of Vinogradov, we refer the reader to Vinogradov's ``two line theorem'' \cite{vi841} which indicates that a system of PDEs of Cauchy--Kovalevskaya type in $n$ independent variables will have trivial horizontal cohomology of horizontal degree $\leq n-2$ in the associated variational bi-complex. This result can be recovered from a fundamental theorem of \cite{bg95} regarding characteristic cohomology (see \cite[Section~10.4]{ka02}). It can also be generalized by using Vinogradov's spectral sequence \cite{ts91} as well as characteristic cohomology techniques \cite{bg95}. The literature on characteristic cohomology is too extensive to discuss in detail, but we will note that the notion of a hyperbolic exterior differential system was introduced and studied in~\cite{bgh951} and~\cite{bgh952}. Additional examples of work in this field include that of Clelland \cite{cl97} and Wang \cite{wa04}, each of whom studied conservation laws using the approach of cha\-rac\-teristic cohomology, the former studying second order parabolic PDEs in one dependent and three independent variables, and the latter considering third order scalar evolution equations.

As mentioned previously, our study of conservation laws will take place in the setting of the variational bi-complex. In this framework, differential forms on the jet bundle of infinite order~$J^{\infty}(E)$ of a fibered manifold $\pi\colon E\rightarrow M$ are bi-graded and the (unconstrained) variational bi-complex is defined using the exterior derivative split into horizontal and vertical components, denoted by ${\rm d}_H$ and ${\rm d}_V$ respectively. The constrained variational bi-complex is associated to a given partial differential equation or system of equations, $\mathcal{R}$, by taking the pullback of the unconstrained variational bi-complex to the infinite prolongation of the equation manifold, $\mathcal{R}^{\infty}$. Then the study of equivalence classes of conservation laws of a given PDE (or system of PDEs) corresponds to the study of the horizontal cohomology of this constrained bi-complex, denoted by $H^{r,s}(\mathcal{R}^{\infty},{\rm d}_H)$ where $(r,s)$ is referred to as the bi-degree of the conservation law. In this context, a classical conservation law will have the form \begin{gather*}\omega=\sum M_{i_1\ldots i_r} {\rm d}x^{i_1}\wedge\cdots\wedge {\rm d}x^{i_r},\end{gather*} where ${\rm d}_H\omega=0$ and the $M_i$ are functions defined on $\mathcal{R}^{\infty}$. We call $\omega$ a trivial conservation law if there exists an $(r-1,0)$ form $\gamma$ on $\mathcal{R}^{\infty}$ such that ${\rm d}_H\gamma=\omega$. In other words, the classical conservation laws determine the cohomology classes of $H^{r,0}\big(\mathcal{R}^{\infty},{\rm d}_H\big)$ where two conservation laws lie in the same cohomology class if they differ by a trivial conservation law. The notion of a higher-dimensional, or contact form-valued, conservation law will then arise when~$\omega$ is a~representative of a~cohomology class for $s\geq1$. In this case the terms~$M_i$ would be type $(0,s)$ contact forms. Analogous to the classical case, $\omega\in H^{r,s}\big(\mathcal{R}^{\infty},{\rm d}_H\big)$ is said to be trivial if there exists a form~$\gamma$ of bi-degree $(r-1,s)$ such that ${\rm d}_H\gamma=\omega$. To illustrate these ideas, we provide the following simple example in two independent variables:

\begin{Example}Consider the Liouville equation
\begin{gather*}
u_{xy}={\rm e}^u.
\end{gather*}
An example of a classical conservation law for this equation is given by the cohomology class $\big[\big(u_{xx}-\frac{1}{2}u_x^2\big){\rm d}x\big]\in H^{1,0}\big(\mathcal{R}^{\infty},{\rm d}_H\big)$, while an example of a contact form valued conservation law would be $[(\theta_{xx}-u_x\theta_x)\wedge {\rm d}x]\in H^{1,1}\big(\mathcal{R}^{\infty},{\rm d}_H\big)$, where $\theta_I={\rm d}u_I-\sum\limits_{j=1}^pu_{I,j}{\rm d}x^j$.
\end{Example}

We will now restrict our attention to the particular class of PDEs with which this paper is concerned. Specifically, we undertake the study of involutive systems of three nonlinear, hyperbolic equations of the following form:
\begin{gather}\label{systemsbeingstudied}
F_{ij}\big(x^1,x^2,x^3,u,u_i,u_j,u_{ij}\big)=u_{ij}-f_{ij}\big(x^1,x^2,x^3,u,u_i,u_j\big)=0
\end{gather}
for $1\leq i < j\leq 3$. Note that the expression $f_{ij}$ depends on all three independent variables, but only depends on the derivatives of~$u$ with respect to $x^i$ and $x^j$. Systems of partial differential equations within this class appear in several interesting contexts. Linear systems of this type arise in the parametrization of Cartan submanifolds in Euclidean space as carried out by Kamran and Tenenblat~\cite{kt96}. This work built upon that of Chern, who in \cite{ch44, ch47} gave a generalization of the Laplace transformation to a class of $n$-dimensional submanifolds in projective space, which he termed \textit{Cartan submanifolds} as they had previously been studied by Cartan in~\cite{ca15}. The submanifolds in Chern's study admit a parametrization by a conjugate net, which in Euclidean space implies that the functions giving the parametrization satisfy an overdetermined system of second order PDEs, \begin{gather*}X_{ij}=\Gamma_{ij}^iX_i+\Gamma_{ij}^jX_j,\qquad 1\leq i\neq j\leq n.\end{gather*} These fit into the class of systems studied in~\cite{kt96}. Other examples of applications of systems of the form~(\ref{systemsbeingstudied}) include the study of semi-Hamiltonian systems of hydrodynamic type (see \cite{ka02,tsa91}, and \cite{dn84}), and (\ref{systemsbeingstudied}) are also a special case of the nonlinear Darboux--Manakov--Zakharov systems studied by Vassiliou in~\cite{va10}.

We proceed to give a synopsis of how this paper is organized. An overview of the essential background material needed to set the stage for the subsequent sections is carried out in Section~\ref{jetbundles} which includes the theory of jet bundles and the variational bi-complex. In Section~\ref{conservationlawssystems} we develop the Laplace adapted coframe which provides the framework for our study and describe a method for generating $(2,s)$ conservation laws for systems~(\ref{systemsbeingstudied}). Darboux integrability is discussed in Section~\ref{darbouxsection}, followed by concluding remarks and directions for future research in Section~\ref{section5}. The systems (\ref{systemsbeingstudied}) are described in the context of Cartan's structural classification~\cite{ca11} of involutive systems of three equations in one dependent and three independent variables in Appendix~\ref{structuralclassification} and the full structure equations and Lie bracket congruences for the Laplace adapted coframe are given in Appendix~\ref{appendixB}.

In Section~\ref{conservationlawssystems} the construction of the Laplace adapted coframe utilizes the methods presented in \cite{ak97}, where a version of the Laplace transformation that applies to the form-valued linearization of a PDE of the form
\begin{gather}\label{AKequations}
F(x,y,u,u_x,u_y,u_{xx},u_{xy},u_{yy})=0
\end{gather}
 is developed. Moving from a single equation (\ref{AKequations}) to a system of equations (\ref{systemsbeingstudied}) introduces a~set of nonlinear integrability conditions which must be satisfied in order for the system to be involutive. These conditions will be analogous to the integrability conditions described in~\cite{kt96}, with the partial derivatives having been replaced by total derivatives.

Once linearized, the system may be analyzed using a generalization of the well-known classical Laplace method, used for integrating a single, linear hyperbolic PDE. This method was adapted to the study of systems of linear hyperbolic PDEs in $n$ independent and one dependent variable in~\cite{kt96}. It was subsequently extended to the vector-valued case in~\cite{sz03}. In this paper, a~form-valued version of the generalized Laplace transform is used to investigate the systems~(\ref{systemsbeingstudied}). Following the detailed computation of the essential structure equations and Lie bracket congruences for the Laplace adapted coframe which are provided in Appendix~\ref{appendixB}, we are in a position to state and prove the first main result of Section~\ref{conservationlawssystems}, which is a structure theorem for the~$(2,s)$ conservation laws for systems~(\ref{systemsbeingstudied}).

In order to motivate the statement of this theorem, we will first describe an analogous result concerning classical conservation laws. The reader may refer to~\cite{ol00} for a detailed exposition on this topic, as all computations and proofs will be omitted here. Consider a system of differential equations
\begin{gather}\label{generalsystem}
\mathcal{R}\colon \ \Delta_{\nu}\big(x,u^{(k)}\big),\qquad 1\leq\nu\leq m,
\end{gather}
where $\big(x,u^{(k)}\big)=\big(x^i,u^{\alpha},\ldots,u^{\alpha}_I\big)$ with $1\leq i\leq n$, $1\leq\alpha\leq q$ and $|I|\leq k$ are local jet coordinates. A~conservation law for this system is an $n$-tuple $P=\big(P_1\big(x,u^{(k)}\big),\ldots,P_n\big(x,u^{(k)}\big)\big)$ whose total divergence vanishes identically for all solutions of~(\ref{generalsystem}). That is,
\begin{gather}\label{totaldivergence}
\operatorname{Div} P=\sum_{i=1}^n D_iP_i=0,
\end{gather}
where $D_i=\frac{\partial}{\partial x^i}+\sum\limits_{\alpha=1}^q\sum_Ju_{J,i}^{\alpha}\frac{\partial}{\partial u_J^{\alpha}}$.
A conservation law $P$ is said to be trivial if there exist $C^{\infty}$ functions $R_{ij}\big(x,u^{(k)}\big)$, $1\leq i,j\leq n$, where $R_{ij}=-R_{ji}$ such that, taking into account $\mathcal{R}$ and its prolongations,
\begin{gather*}P_i=\sum_{j=1}^nD_jR_{ij}.\end{gather*}
Then two conservation laws, $P$ and $\tilde{P}$, are considered equivalent if their difference $P-\tilde{P}$ is a~trivial conservation law. It can be shown using integration by parts that if~$\mathcal{R}$ is totally non-degenerate, i.e., the system of equations and all its prolongations are of maximal rank and locally solvable, then any conservation law~$P$ is equivalent to a conservation law $\tilde{P}$ whose total divergence can be written in the form
\begin{gather}\label{characteristicform}
\sum_{i=1}^n D_i\tilde{P}_i=\sum_{\nu=1}^m Q_{\nu}\Delta_{\nu},
\end{gather}
where the multipliers $Q_{\nu}=Q_{\nu}\big(x,u^{(k)}\big)$ are $C^{\infty}$ functions on the infinite jet bundle. The $m$-tuple $Q=(Q_1,\ldots,Q_m)$ is referred to as the \textit{characteristic} of the conservation law~$P$. A characteristic is trivial if it vanishes for all solutions of the system of equations, and two characteristics are said to be equivalent if they differ by a trivial characteristic. As stated in \cite{ol00}, two conservation laws $P$ and $\tilde{P}$ are equivalent if and only if their characteristics are equivalent. So it is natural to suspect that the study of characteristics and conservation laws go hand in hand. Applying the Euler--Lagrange operator, \begin{gather*}E_{\alpha}=\sum_{\vert I\vert\geq0}(-1)^{\vert I\vert}D_I\left(\frac{\partial}{\partial u^{\alpha}_I}\right),\end{gather*} to both sides of~(\ref{characteristicform}), one obtains the identity
\begin{gather}\label{kernelofadjoint}
\mathcal{L}_Q^*(\Delta)+\mathcal{L}_{\Delta}^*(Q)=0,
\end{gather}
where $\mathcal{L}^*_Q$ is the adjoint of the formal Fr\'{e}chet derivative of the differential operator $Q$, and likewise for $\mathcal{L}_{\Delta}^*$. Since the system of equations specifies $\Delta=0$, the identity (\ref{kernelofadjoint}) leads to the following conclusion:
\begin{Theorem}\label{classicalkernel}The characteristic $Q_{\nu}\big(x,u^{(k)}\big)$, $1\leq\nu\leq m$, of any conservation law \eqref{totaldivergence} for a totally non-degenerate system of differential equations \eqref{generalsystem} lies in the kernel of $\mathcal{L}_{\Delta}^*$, that is $\mathcal{L}_{\Delta}^*(Q)=0$.
\end{Theorem}

Then the first main result of this paper, Theorem~\ref{structuretheorem}, which we will state in full presently, can be seen as an analogous result for conservation laws of systems of the type (\ref{systemsbeingstudied}) for higher vertical degrees. Essentially, it tells us that any $(2,s)$ conservation law can be constructed from certain $(0,s-1)$ contact forms which satisfy an equation involving the adjoints of the linearized equations of the original system. Precisely, it says the following:

\begin{Theorem}\label{structuretheoremintro}Let $\mathcal{R}$ be a second order hyperbolic system of type \eqref{systemsbeingstudied}. Then for $s\geq1$ and $\omega\in\Omega^{2,s}\big(\mathcal{R}^{\infty}\big)$ a ${\rm d}_H$-closed form, there exist contact forms \begin{gather*}\rho_{ij}\in\Omega^{0,s-1}\big(\mathcal{R}^{\infty}\big)\qquad \textrm{and}\qquad \gamma\in\Omega^{1,s}\big(\mathcal{R}^{\infty}\big)\end{gather*} for $1\leq i < j\leq 3$, such that $\omega$ is given by \begin{gather*}\omega=\sum_{1\leq i < j\leq 3}\Psi_{ij}(\rho_{ij})+{\rm d}_H\gamma\end{gather*} and the $\rho_{ij}$ satisfy the equation
\begin{gather*}\sum_{1\leq i < j\leq 3}\mathcal{L}_{ij}^*(\rho_{ij})=0,\end{gather*}
where the $\Psi_{ij}$ are maps from the space of $(0,s-1)$ forms to the space of $(2,s)$ forms defined explicitly in Section~{\rm \ref{structuretheoremsec}}, and the operators~$\mathcal{L}_{ij}^*$ are the adjoints of the operators appearing in the linearized system.
\end{Theorem}

The second noteworthy result of Section~\ref{conservationlawssystems} concerns the cohomology of the constrained variational bi-complex which is defined in Section~\ref{constrainedbisec}. Before stating the theorem, we will take a~moment to introduce some terminology which will be helpful in the discussion of this result, and those to follow. Just as in the case of the classical Laplace method, there are generalized Laplace invariants which arise during the application of the generalized Laplace transform. The generalized Laplace invariants are relative invariants with respect to contact transformations, so their vanishing is a contact-invariant condition. When the generalized Laplace transformation is applied repeatedly to a particular system of equations, a sequence of generalized Laplace invariants is generated and this sequence may or may not terminate at some point. If, for example, the~$(i,j)$ Laplace invariants are zero after $p_{ij}$ applications of the $\mathcal{X}_{ij}$-Laplace transform, then~$p_{ij}$ is referred to as a \textit{Laplace index} of the system of equations, written $\operatorname{ind}(\mathcal{X}_{ij})=p_{ij}$. If the sequence of Laplace invariants never terminates, then we write $\operatorname{ind}(\mathcal{X}_{ij})=\infty$.

Then the crux of Theorem~\ref{sgeq3trivial} is that if each of these sequences of Laplace invariants fails to terminate, there will be no non-trivial horizontal cohomology of bi-degree $(2,s)$ for each $s\geq3$, as stated below.

\begin{Theorem}Let $\mathcal{R}$ be a second order hyperbolic system of type \eqref{systemsbeingstudied} and suppose that $\operatorname{ind}(\mathcal{X}_{ij})=\infty$ for $1\leq i< j\leq3$. Then, for all~$s\geq3$, all type $(2,s)$ conservation laws are trivial. That is,
\begin{gather*}
H^{2,s}\big(\mathcal{R}^{\infty},{\rm d}_H\big)=0.
\end{gather*}
\end{Theorem}

The proof of this result utilizes the fact that relative invariant contact forms can be constructed from nonzero solutions to the adjoint equation seen in Theorem~\ref{structuretheorem}. A contact form,~$\omega$, is said to be invariant relative to the characteristic vector field $X$ if $X(\omega)=\lambda\omega$ for some function $\lambda$ on $\mathcal{R}^{\infty}$, where $X(\omega)$ denotes the projected Lie derivative of $\omega$ with respect to~$X$. The existence of such contact forms then contradicts the hypothesis that none of the Laplace indices is finite.

The second set of results from this paper is contained in Section~\ref{darbouxsection} and addresses the topic of the Darboux integrability of systems of the form~(\ref{systemsbeingstudied}). Classically, a second order scalar hyperbolic partial differential equation (\ref{AKequations}) is Darboux integrable if there exist smooth, real-valued functions $I$, $\tilde{I}$, $J$, and $\tilde{J}$ such that ${\rm d}I\wedge {\rm d}\tilde{I}\neq0$, ${\rm d}J\wedge {\rm d}\tilde{J}\neq0$ and \begin{gather*}X(I)=X\big(\tilde{I}\big)=0\qquad \textrm{and}\qquad Y(J)=Y\big(\tilde{J}\big)=0,\end{gather*} where $X$ and $Y$ are the characteristic vector fields for the equation. It is well known, see for example~\cite{ka02}, that for any pair of monotone functions $f_1, f_2\in C^{\infty}(\mathbb{R};\mathbb{R})$, the system
\begin{gather}\label{darbouxintsystem}
F(x,y,u,u_x,u_y,u_{xx},u_{xy},u_{yy})=0,\qquad \tilde{I}=f_1(I),\qquad \tilde{J}=f_2(J)
\end{gather}
 is completely integrable in the sense of the Frobenius theorem. It is therefore evident that Darboux integrable equations may be solved via ordinary differential equation techniques.
\begin{Example}To illustrate the concepts above, we return to the Liouville equation \begin{gather*}u_{xy}={\rm e}^u.\end{gather*}
In this example, the characteristic vector fields are $X=D_x$ and $Y=D_y$ (the total derivatives with respect to $x$ and $y$, respectively). Then $I=y$ and $\tilde{I}=u_{yy}-\frac{1}{2}u_y^2$ are invariant functions with respect to $X$, and $J=x$ and $\tilde{J}=u_{xx}-\frac{1}{2}u_x^2$ are invariant functions with respect to $Y$ as required by the definition above. Letting $f_1=g_1'$ and $f_2=g_2'$, it is possible to integrate the system (\ref{darbouxintsystem}) to obtain \begin{gather*}u=\log\left|\frac{2g_1'(y)g_2'(x)}{(g_1(y)+g_2(x))^2}\right|.\end{gather*}
\end{Example}

The concept of Darboux integrability has been studied extensively, and extended to new settings. Two noteworthy examples are~\cite{ak97}, where the Darboux integrability of equations (\ref{AKequations}) is shown to imply the existence of infinitely many conservation laws of type $(1,s)$ for all $s\geq0$, and~\cite{afv09} where the definition of Darboux integrability is recast in a group-theoretic approach that applies to the general framework of exterior differential systems. The definition introduced in~\cite{afv09} equates Darboux integrability with the existence of what the authors refer to as a \textit{Darboux pair}. This terminology is explained in Section~\ref{darbforsystemofthree}, where we also prove the lemma quoted below, demonstrating that if certain characteristic invariant functions exist, systems of the form~(\ref{systemsbeingstudied}) will satisfy the definition of Darboux integrability given in~\cite{afv09}.

\begin{Lemma}\label{darbdef1} Let $u_{ij}=f_{ij}\big(x^1,x^2,x^3,u,u_i,u_j\big)$, $1\leq i< j<3$ be a system of three hyperbolic equations in one dependent and three independent variables with characteristic vector fields $X_1$, $X_2$, $X_3$. If there exist smooth, real-valued functions, $I$, $\tilde{I}$, $J$, $\tilde{J}$, and $K$, $\tilde{K}$, such that the following two conditions hold, then the system defines a Darboux pair and thus satisfies the notion of Darboux integrability defined in~{\rm \cite{afv09}}.
 \begin{enumerate}\itemsep=0pt
 \item[$1.$] $I$ and $\tilde{I}$ are invariant with respect to two of the characteristic vector fields, say $X_i$ and $X_j$, and $I$ and $\tilde{I}$ are functionally independent: ${\rm d}I\wedge {\rm d}\tilde{I}\neq0$.
 \item[$2.$] $J$, $\tilde{J}$, $K$, and $\tilde{K}$ are all invariant with respect to $X_l$, $l\neq i,j$, and are all functionally independent from each other: ${\rm d}J\wedge {\rm d}\tilde{J}\wedge {\rm d}K\wedge {\rm d}\tilde{K}\neq0$.
\end{enumerate}
\end{Lemma}

The next lemma describes a method of constructing a contact form invariant with respect to a pair of characteristic vector fields by using the characteristic invariant functions described in Lemma~\ref{darbdef1}.

\begin{Lemma}Let $I$ and $J$ be functions on $\mathcal{R}$ and $X_1$, $X_2$ and $X_3$ characteristic vector fields. If $I$ and $J$ are invariant with respect to both $X_1$ and $X_2$, such that $X_3(I)=I'$ and $X_3(J)=1$, then \begin{gather*}\omega={\rm d}_VI-I'{\rm d}_VJ\end{gather*} is an $X_1$ and $X_2$ invariant contact form.
\end{Lemma}

It is then possible to utilize the invariant contact form given in the previous lemma in order to draw a connection between Darboux integrability and the Laplace indices of the system, as the following corollary indicates.

\begin{Corollary}Let $\mathcal{R}$ be a system of equations of the form \eqref{systemsbeingstudied}. If $\mathcal{R}$ satisfies the conditions described in Lemma~{\rm \ref{darbdef1}}, then at least one of the Laplace indices $\operatorname{ind}(\mathcal{X}_{ij})$ must be finite.
\end{Corollary}

The connection between Darboux integrability and the termination of sequences of Laplace invariants has been established in many different contexts. In particular, \cite{ak97} and \cite{aj97} taken together show that a scalar second-order hyperbolic PDE in the plane is Darboux integrable if and only if the Laplace indices of each of the two generalized Laplace transforms associated to the equation are finite. In Section~\ref{darbouxsection}, we also describe an algorithm for generating infinitely many $(1,s)$ and $(2,s)$ conservation laws for a system that the conditions of Lemma~\ref{darbdef1}. Let us summarize here the procedure for generating $(1,s)$ conservation laws in particular.

First, a ${\rm d}$-closed $(s+1)$-form is constructed by taking the wedge product of the exterior derivatives of a sequence of characteristic invariant functions whose existence is guaranteed in the hypotheses of Lemma~\ref{darbdef1}. Let this $(s+1)$ form be written as $\alpha={\rm d}I_1\wedge {\rm d}I_2\wedge\cdots\wedge {\rm d}I_{s+1}$. It is then shown that the $(1,s)$ component of~$\alpha$ will be ${\rm d}_H$ closed. Additional characteristic invariant functions are found by applying an appropriate choice of characteristic vector field repeatedly to the invariant functions described in Lemma~\ref{darbdef1}, and in this way infinitely many conservation laws may be constructed. The method presented for the construction of type $(2,s)$ conservation laws is completely analogous. It is furthermore shown in Section~\ref{darbouxsection} that these conservation laws, of both types, can be written in such a way that they may be shown to be nontrivial, allowing us to conclude with the following theorem.
\begin{Theorem}If $\mathcal{R}$ is a system of equations satisfying the hypotheses of Lemma~{\rm \ref{darbdef1}}, then there exist infinitely many nontrivial type $(1,s)$ and type $(2,s)$ conservation laws for all $s\geq0$.
\end{Theorem}

The paper concludes with Section~\ref{section5} which provides a summary of our findings and suggestions for future research. The results presented in this paper are a part of the author's Ph.D.~Thesis at McGill University~\cite{froehlich}.

\section{The variational bi-complex}\label{jetbundles}

Presently we will establish the necessary definitions and notation pertaining to several key concepts, including infinite jet bundles, split exterior differentiation, and the variational bi-complex, as presented in \cite{an89}.

\subsection{Infinite jet bundles}

Begin with a fibered manifold \begin{gather*}\pi\colon \ E\rightarrow M,\end{gather*} with adapted coordinates $\big(x^i,u^{\alpha}\big)$, for $1\leq i \leq n$ and $1\leq \alpha \leq q$, over a connected base manifold~$M$ of dimension~$n$. In this paper we will be concerned with local properties, although the infinite jet bundle has important global properties as well (see~\cite{an89}). Given our focus, we will take $M$ to be an open connected subset of $\mathbb{R}^n$ and $\pi$ to be the trivial bundle. Let $J^k(E)=\cup_{x\in M}J^k_x(E)$ denote the bundle of $k$-jets of local sections of~$E$, with local coordinates on $J^k(E)$ consisting of $\big(x^i,u^{\alpha},u^{\alpha}_{i_1},u^{\alpha}_{i_1i_2},\ldots,u^{\alpha}_{i_1i_2\ldots i_k}\big)$ where $1\leq i_1\leq i_2\leq\dots\leq i_k\leq n$ and with the natural projection maps
\begin{gather*}\pi_M^k\colon \ J^k(E)\rightarrow M\qquad \textrm{and}\qquad \pi_E^k\colon \ J^k(E)\rightarrow E.\end{gather*}
In local coordinates a $p$-form on $J^{\infty}(E)$ will be a sum of the form \begin{gather*}\omega=\sum_{r+t=p}\sum_{\alpha}a_{\alpha}^{IJ}{\rm d}x^{i_1}\wedge\cdots\wedge {\rm d}x^{i_r}\wedge {\rm d}u^{\alpha_1}_{j_1}\wedge\cdots\wedge {\rm d}u^{\alpha_t}_{j_t},\end{gather*} where the coefficients $a_{\alpha}^{IJ}$ are $C^{\infty}$ real-valued functions defined on $J^{\infty}(E)$ and $I$ and $J$ are multi-indices, $I=i_1\cdots i_r$ and $J=j_1\cdots j_t$. The contact ideal on $J^{\infty}(E)$ generated by the contact 1-forms
\begin{gather*}%\label{contactform}
\theta^{\alpha}_I={\rm d}u^{\alpha}_I-\sum_{j=1}^n u^{\alpha}_{Ij}{\rm d}x^j
\end{gather*}
forms an ideal in $\Omega^*(J^{\infty}(E))$, which we will denote by $\mathcal{C}(J^{\infty}(E))$ and whose exterior derivatives are given by the structure equations
\begin{gather*}
{\rm d}\theta^{\alpha}_{I}=\sum_{j=1}^n{\rm d}x^j\wedge\theta^{\alpha}_{Ij}.
\end{gather*}
Two particular types of vector fields on $J^{\infty}(E)$, \textit{total} vector fields and \textit{vertical} vector fields, will play an important role in the variational bi-complex:
\begin{Definition}\label{verticalvfdef} A vector field $X$ is said to be $\pi_M^{\infty}$ vertical if $(\pi_M^{\infty})_*(X)=0$.
\end{Definition}
\begin{Definition}\label{totalvfdef} A vector field $X$ on $J^{\infty}(E)$ for which $X \,\lrcorner\, \omega=0$ for any contact form $\omega$ is referred to as a total vector field.
\end{Definition}
Note that in general total vector fields take the form
\begin{gather*}%\label{totalvf}
X=\sum_{j=1}^nA^jD_j.
\end{gather*}

\subsection{The bi-graded exterior derivative}

We may now introduce a bi-grading of the $p$-forms $\omega$ on $J^{\infty}(E)$ which will allow us to distinguish between independent and dependent variables when working with differential equations.
\begin{Definition}\label{totalverticaldef}A $p$-form $\omega$ is said to be of type $(r,s)$ if $r+s=p$ and $\omega(X_1,\ldots,X_p)=0$ whenever there are more than $r$ total vector fields or more than $s$ $\pi_M^{\infty}$ vertical vector fields among the vector fields $X_i$.
\end{Definition}

Using Definition~\ref{totalverticaldef}, the de Rham complex on $J^{\infty}(E)$ can be bi-graded as follows
\begin{gather*}%\label{directsum}
\Omega^n(J^{\infty}(E))=\bigoplus_{r+s=n}\Omega^{r,s}(J^{\infty}(E)),
\end{gather*}
 where in local coordinates a differential form $\omega$ of the type $(r,s)$ will have the form
\begin{gather*}\omega=\sum a_{\beta}^{iI}{\rm d}x^{i_1}\wedge\cdots\wedge {\rm d}x^{i_r}\wedge\theta_{I_1}^{\beta_1}\wedge\cdots\wedge\theta_{I_s}^{\beta_s}\end{gather*} for real-valued functions $a_{\beta}^{iI}$ on $J^{\infty}(E)$.
The bi-grading of forms in the de Rham complex induces a corresponding decomposition of the exterior derivative ${\rm d}\colon \Omega^n\longrightarrow\Omega^{n+1}$, given by
\begin{gather*}%\label{bigradedextder}
{\rm d}\colon \ \Omega^{r,s}\longrightarrow\Omega^{r+1,s}\oplus\Omega^{r,s+1},
\end{gather*}
 where $\omega\mapsto {\rm d}\omega={\rm d}_H\omega\oplus {\rm d}_V\omega$. Here ${\rm d}_H\omega\in\Omega^{r+1,s}$ is called the \textit{horizontal} exterior derivative and ${\rm d}_V\omega\in\Omega^{r,s+1}$ is called the \textit{vertical} exterior derivative. The operators ${\rm d}_H$ and ${\rm d}_V$ are anti-commuting differentials, i.e., ${\rm d}_H{\rm d}_V=-{\rm d}_V{\rm d}_H$. The ${\rm d}_H$ and ${\rm d}_V$ structure equations for a~func\-tion~$f$ and a type $(r,s)$ form $\omega$ are as follows
\begin{gather}\label{dHstructure}
{\rm d}_Hf=\sum_iD_{x^i}f,\qquad {\rm d}_H\big({\rm d}x^i\big)=0 \qquad \textrm{and}\qquad {\rm d}_H\theta^{\beta}_{I}=\sum_j{\rm d}x^j\wedge\theta^{\beta}_{Ij},\\
{\rm d}_Vf=\sum_I \sum_{\beta}\big(\partial_{\beta}^If\big)\theta_I^{\beta},\qquad {\rm d}_V\big({\rm d}x^i\big)=0\qquad \textrm{and}\qquad {\rm d}_V\theta^{\beta}_I=0.\nonumber
\end{gather}

Let $\pi\colon E\rightarrow M$ and $\rho\colon E'\rightarrow N$ be two fibered manifolds and let $\Phi\colon J^{\infty}(E)\rightarrow J^{\infty}(E')$ be a smooth map. We define the \textit{projected} pullback map, which will maintain a form's bi-graded type, to be the map $\Phi^{\sharp}\colon \Omega^{r,s}(J^{\infty}(E'))\rightarrow\Omega^{r,s}(J^{\infty}(E))$ given by \begin{gather*}\Phi^{\sharp}(\omega)=\pi^{r,s}[\Phi^*(\omega)],\end{gather*} where $\pi^{r,s}$ is the projection map from $\Omega^p(J^{\infty}(E))$ to $\Omega^{r,s}(J^{\infty}(E))$. Likewise, for a total vector field $X$ on $J^{\infty}(E)$, and a type $(r,s)$ form $\omega$, we define $X(\omega)\in\Omega^{r,s}(J^{\infty}(E))$ to be the \textit{projected} Lie derivative, \begin{gather*}X(\omega)=\pi^{r,s}(\mathcal{L}_X\omega).\end{gather*}
The identity $X(\omega)=X \,\lrcorner\, {\rm d}_H(\omega)+{\rm d}_H(X\, \lrcorner\, \omega)$ follows from Cartan's formula and this, along with~(\ref{dHstructure}) and the antisymmetry of ${\rm d}_H$ and ${\rm d}_V$, implies the following relations
\begin{gather}\label{Djtheta}
D_j\theta^{\beta}_{I}=\theta^{\beta}_{Ij},\qquad
X({\rm d}_H\omega)={\rm d}_HX(\omega).
\end{gather}
Finally, the proposition below gives additional properties of the projected Lie derivative that will be used in performing calculations, and are direct consequences of the definition of the projected Lie derivative and the properties of the Lie derivative.

\begin{Proposition}\label{verhorprop}Let $\omega\in\Omega^{r,s}(J^{\infty}(E))$. If $X$ and $Y$ are total vector fields on $J^{\infty}(E)$ and $Z$ is a $\pi_M^{\infty}$ vertical vector field on $J^{\infty}(E)$ then the following two equations hold
\begin{gather}%\label{liebracket}
X(Y(\omega))-Y(X(\omega)) =[X,Y](\omega),\nonumber\\
\label{verhorint} Z\,\lrcorner\,X(\omega) =[Z,X]\,\lrcorner\,\omega+X(Z\,\lrcorner\,\omega).
\end{gather}
\end{Proposition}

Using the bi-grading of the exterior derivative, we can now define the variational bi-complex for type $(r,s)$ forms, which will provide a natural framework for our investigation of conservation laws for PDEs systems of the type~(\ref{systemsbeingstudied}).

\subsection{The variational bi-complex for PDEs systems of the type (\ref{systemsbeingstudied})}\label{constrainedbisec}

We will consider a system of three semi-linear PDEs defined on an open connected subset~$U$ of~$\mathbb{R}^3$ of the following form
\begin{gather}\label{quasisystem}
F_{ij}=u_{ij}-f_{ij}\big(x^1,x^2,x^3,u,u_i,u_j\big)=0\qquad \textrm{for}\quad 1\leq i< j\leq3.
\end{gather}
The equations (\ref{quasisystem}) define a locus in $J^2(E)$ where $E$ is the trivial bundle \begin{gather*}\pi\colon \ E\simeq U\times(a,b)\subset\mathbb{R}^3\times\mathbb{R}\rightarrow M\simeq U.\end{gather*}
Let $\mathcal{R}$ denote an open contractible subset of this locus. Assume that the $f_{ij}$ are $C^{\infty}$ functions in a neighborhood of $\mathcal{R}$, and that $\pi|_{\mathcal{R}}\colon \mathcal{R}\rightarrow U$ is a subbundle of the fiber bundle $\pi\colon J^2(E)\rightarrow M$. Note that for each $ij$ pair, $\partial F_{ij}/\partial u_{ij}\neq0$ on~$\mathcal{R}$. Furthermore, in order for the system~(\ref{quasisystem}) to be locally solvable, we make the assumption that the following set of integrability conditions is satisfied,
\begin{gather*}%\label{partialscommute}
D_kf_{ij}=D_if_{kj}\qquad \textrm{for}\quad 1\leq i\neq j\neq k\leq3.
\end{gather*}
If the module of contact forms on $J^2(E)$ is pulled back to $\mathcal{R}$, then a Pfaffian system, $\mathcal{I}$, on $\mathcal{R}$ is obtained. The differential system $\mathcal{I}$ is generated by the one forms
\begin{gather*}
\theta={\rm d}u-u_1{\rm d}x^1-u_2{\rm d}x^2-u_3{\rm d}x^3, \\
\theta_1={\rm d}u_1-u_{11}{\rm d}x^1-f_{12}{\rm d}x^2-f_{13}{\rm d}x^3,\\
\theta_2={\rm d}u_2-f_{12}{\rm d}x^1-u_{22}{\rm d}x^2-f_{23}{\rm d}x^3,\\
\theta_3={\rm d}u_3-f_{13}{\rm d}x^1-f_{23}{\rm d}x^2-u_{33}{\rm d}x^3.
\end{gather*}
Solutions of (\ref{quasisystem}) are then local sections $\sigma\colon U\rightarrow\mathcal{R}$ such that $\sigma^*(\omega)=0$ for all $\omega\in\mathcal{I}$. That is, solutions are integral manifolds of $\mathcal{I}$ with the independence condition \begin{gather*}\Omega\equiv {\rm d}x^1\wedge {\rm d}x^2\wedge {\rm d}x^3\quad \mod\mathcal{I}.\end{gather*}
The $k^{\rm th}$ prolongation, written $\mathcal{R}^{(k)}$, of $\mathcal{R}$ is the locus in $J^{k+2}(E)$ defined by the equations \begin{gather*}F_{ij}=0,\qquad D_1F_{ij}=0,\qquad D_2F_{ij}=0,\qquad D_3F_{ij}=0,\qquad \ldots,\qquad D_1^lD_2^mD_3^nF_{ij}=0\end{gather*} for $1\leq i < j\leq 3$ and $l+m+n\leq k$. For example, the first prolongation of $\mathcal{R}$ is
\begin{gather*}
\mathcal{R}^{(1)}=\big\{\big(j^3s\big)(x)\colon \big(j^2s\big)(x)\in\mathcal{R}\ \textrm{and} \\
\hphantom{\mathcal{R}^{(1)}=\big\{\big(j^3s\big)(x)\colon}{} (D_1F_{ij})\big(\big(j^3s\big)(x)\big)=(D_2F_{ij})\big(\big(j^3s\big)(x)\big)=(D_3F_{ij})\big(\big(j^3s\big)(x)\big)=0\big\},
\end{gather*}
where $(j^ks)(x)$ denotes equivalence classes of $k$-jets of local sections $s$ of $E$.

Each prolongation $\mathcal{R}^{(k)}$ is a $C^{\infty}$ submanifold of $J^{k+2}(E)$. As a consequence of involu\-ti\-vi\-ty~$\mathcal{R}^{(k+1)}$ fibers over $\mathcal{R}^{(k)}$, $\pi_k^{k+1}|_{\mathcal{R}^{(k+1)}}\colon \mathcal{R}^{(k+1)}\rightarrow\mathcal{R}^{(k)}$. Thus we can define the inverse limit of the system of $k^{\rm th}$ prolongations to be $\mathcal{R}^{\infty}$, called the infinite prolongation of $\mathcal{R}$ with the projection maps $\pi_k^{\infty}\colon \mathcal{R}^{\infty}\rightarrow\mathcal{R}^k$ and $\pi_U^{\infty}\colon \mathcal{R}^{\infty}\rightarrow U$. Denote by $\mathcal{C}(\mathcal{R}^{\infty})$ the pullback of the contact ideal on $J^{\infty}(E)$ to $\mathcal{R}^{\infty}$, $\mathcal{C}(\mathcal{R}^{\infty})=\iota^*[\mathcal{C}(J^{\infty}(E))]$ where $\iota$ is the inclusion map $\iota\colon \mathcal{R}^{\infty}\rightarrow J^{\infty}(E)$. Then we can define the constrained variational bi-complex to be the pullback of the free variational bi-complex.

\begin{Definition}The constrained variational bi-complex for $\widehat{\mathcal{R}}=\{\mathcal{R}^{\infty}, \pi^{\infty}_U,\mathcal{C}(\mathcal{R}^{\infty})\}$ is the pullback of the free variational bi-complex $(\Omega^{*,*}(J^{\infty}(E)),{\rm d}_H,{\rm d}_V)$ to $\mathcal{R}^{\infty}$:
\begin{gather*}
\begin{array}{@{}ccccccccccc} & & & & \uparrow \scriptstyle{{\rm d}_V}& & \uparrow \scriptstyle{{\rm d}_V} & & \uparrow \scriptstyle{{\rm d}_V} & & \uparrow \scriptstyle{{\rm d}_V} \\ & & 0 & \xrightarrow[]{} & \Omega^{0,2}(\mathcal{R}^{\infty}) & \xrightarrow[{\rm d}_H]{} & \Omega^{1,2}(\mathcal{R}^{\infty}) & \xrightarrow[{\rm d}_H]{} & \Omega^{2,2}(\mathcal{R}^{\infty}) & \xrightarrow[{\rm d}_H]{} & \Omega^{3,2}(\mathcal{R}^{\infty}) \\ & & & & \uparrow \scriptstyle{{\rm d}_V} & & \uparrow \scriptstyle{{\rm d}_V} & & \uparrow \scriptstyle{{\rm d}_V} & & \uparrow \scriptstyle{{\rm d}_V} \\ & & 0 & \xrightarrow[]{} & \Omega^{0,1}(\mathcal{R}^{\infty}) & \xrightarrow[{\rm d}_H]{} & \Omega^{1,1}(\mathcal{R}^{\infty}) & \xrightarrow[{\rm d}_H]{} & \Omega^{2,1}(\mathcal{R}^{\infty}) & \xrightarrow[{\rm d}_H]{} & \Omega^{3,1}(\mathcal{R}^{\infty}) \\ & & & & \uparrow \scriptstyle{{\rm d}_V} & & \uparrow \scriptstyle{{\rm d}_V} & & \uparrow \scriptstyle{{\rm d}_V} & & \uparrow \scriptstyle{{\rm d}_V} \\0 & \xrightarrow[]{} & \mathbb{R} & \xrightarrow[]{} & \Omega^{0,0}(\mathcal{R}^{\infty}) & \xrightarrow[{\rm d}_H]{} & \Omega^{1,0}(\mathcal{R}^{\infty}) & \xrightarrow[{\rm d}_H]{} & \Omega^{2,0}(\mathcal{R}^{\infty}) & \xrightarrow[{\rm d}_H]{} & \Omega^{3,0}(\mathcal{R}^{\infty}).\end{array}
\end{gather*}
\end{Definition}
Notice that there are no additional columns to the right of the constrained bi-complex due to the fact that our system~(\ref{quasisystem}) involves exactly 3 independent variables.

We have the following coordinates on $\mathcal{R}^{\infty}$
\begin{gather}\label{natcoorrinfty}
\big(x^1,x^2,x^3,u,u_1,u_2,u_3,u_{11},u_{22},u_{33},\ldots,u_{1^k},u_{2^k},u_{3^k},\ldots\big)
\end{gather}
and a basis for the contact ideal on $\mathcal{R}^{\infty}$ is given by
\begin{gather}\label{natcoframerinfty}
\big\{\theta,\theta_1,\theta_2,\theta_3,\theta_{11},\theta_{22},\theta_{33},\ldots,\theta_{1^k},\theta_{2^k},\theta_{3^k},\ldots\big\},
\end{gather}
where
\begin{gather*}
\theta={\rm d}u-u_1{\rm d}x^1-u_2{\rm d}x^2-u_3{\rm d}x^3,\\
\theta_{i^k}={\rm d}u_{i^k}-u_{i^{k+1}}{\rm d}x^i-D_{i^{k-1}}(f_{ij}){\rm d}x^j-D_{i^{k-1}}(f_{il}){\rm d}x^l\qquad \textrm{for}\quad j,l\neq i.
\end{gather*}

We will call the basis $\big\{{\rm d}x^1,{\rm d}x^2,{\rm d}x^3,\theta,\theta_1,\theta_2,\theta_3,\ldots,\theta_{1^k},\theta_{2^k},\theta_{3^k},\ldots\big\}$ the coordinate coframe on~$\mathcal{R}^{\infty}$. It will be the first of several coframes introduced in our study of the systems~(\ref{quasisystem}).

In this coordinate system the total derivatives, $D_i$, are then expressed as
\begin{gather}
D_1 =\frac{\partial}{\partial x^1}+u_1\frac{\partial}{\partial u}+u_{11}\frac{\partial}{\partial u_1}+f_{12}\frac{\partial}{\partial u_2}+f_{13}\frac{\partial}{\partial u_3}\nonumber\\
\hphantom{D_1 =}{} +u_{111}\frac{\partial}{\partial u_{11}}+D_2(f_{12})\frac{\partial}{\partial u_{22}}+D_3(f_{13})\frac{\partial}{\partial u_{33}}+\cdots \label{totalderivativedef}
\end{gather}
and likewise for $D_2$ and $D_3$.

A smooth function $g\big(x^1,x^2,x^3,u,u_1,u_2,u_3,\ldots,u_{1^k},u_{2^k},u_{3^k},\ldots\big)$ on $\mathcal{R}^{\infty}$ has the structure equations
\begin{gather*}
{\rm d}_Hg =(D_1g){\rm d}x^1+(D_2g){\rm d}x^2+(D_3g){\rm d}x^3, \\
{\rm d}_Vg =\frac{\partial g}{\partial u}\theta+\frac{\partial g}{\partial u_1}\theta_1+\frac{\partial g}{\partial u_2}\theta_2+\frac{\partial g}{\partial u_3}\theta_3+\dots+\frac{\partial g}{\partial u_{1^k}}\theta_{1^k}+\frac{\partial g}{\partial u_{2^k}}\theta_{2^k}+\frac{\partial g}{\partial u_{3^k}}\theta_{3^k}+\cdots
\end{gather*}
and
\begin{gather*}
{\rm d}_H\theta_{i^k} ={\rm d}_H({\rm d}_V u_{i^k})= -{\rm d}_V\big[u_{1i^k}{\rm d}x^1+u_{2i^k}{\rm d}x^2+u_{3i^k}{\rm d}x^3\big] \\
\hphantom{{\rm d}_H\theta_{i^k}}{} =-\theta_{i^{k+1}}\wedge {\rm d}x^i-{\rm d}_V\big(D^{k-1}_i(f_{ij})\big)\wedge {\rm d}x^j-{\rm d}_V\big(D^{k-1}(f_{il})\big)\wedge {\rm d}x_l\qquad \textrm{for}\quad j,l\neq i.
\end{gather*}

\begin{Definition}Consider a form $\omega=M_1{\rm d}x^1+M_2{\rm d}x^2+M_3{\rm d}x^3$ where the $M_i$ are smooth functions on $\mathcal{R}^{\infty}$. If the horizontal derivative of $\omega$ vanishes, that is $\omega$ is ${\rm d}_H$-closed, then $\omega$ is called a classical conservation law for $\widehat{\mathcal{R}}=\{\mathcal{R}^{\infty},\pi^{\infty}_U,\mathcal{C}(\mathcal{R}^{\infty})\}$. If in addition $\omega$ is exact, meaning that there exists a function $N$ such that ${\rm d}N=\omega$, then $\omega$ is said to be a trivial conservation law.
If the~$M_i$ are themselves $(0,s)$ contact forms, for $s\geq 1$, then a $(1,s)$ form \begin{gather*}\omega=M_1\wedge {\rm d}x^1+M_2\wedge {\rm d}x^2+M_3\wedge {\rm d}x^3\end{gather*} or a $(2,s)$ form \begin{gather*}\omega=M_1\wedge {\rm d}x^1\wedge {\rm d}x^2+M_2\wedge {\rm d}x^2\wedge {\rm d}x^3+M_3\wedge {\rm d}x^1\wedge {\rm d}x^3, \end{gather*} such that ${\rm d}_H\omega=0$ is referred to as a contact form valued conservation law of~$\widehat{\mathcal{R}}$.
\end{Definition}

The following lemma follows from the fact that ${\rm d}_H$ is an anti-commuting differential.

\begin{Lemma}If $\omega$ is a type $(r,s)$ form valued conservation law and there exists a type $(r-1,s)$ form~$\gamma$ for which $\omega={\rm d}_H\gamma$, then ${\rm d}_H\omega=0$ and~$\omega$ is called a trivial conservation law.
\end{Lemma}

\begin{Remark}We will not need to investigate type $(3,s)$ conservation laws owing to the fact that the systems (\ref{quasisystem}) which are the focus of our study have three independent variables and hence any type $(3,s)$ form will be trivially ${\rm d}_H$-closed. Furthermore, due to Vinogradov's ``two line theorem'', conservation laws of type $(0,s)$ are also trivially ${\rm d}_H$-closed~\cite{vi842}. Thus the focus of what follows will be type $(1,s)$ and $(2,s)$ form valued conservation laws, henceforth simply referred to as conversation laws.
\end{Remark}

The conservation laws described above can also be viewed in terms of the horizontal cohomology of the constrained variational bi-complex. The cohomology space
\begin{gather*}
H^{1,s}(\mathcal{R}^{\infty},{\rm d}_H)=\frac{\operatorname{Ker}\big({\rm d}_H\colon \Omega^{1,s}(\mathcal{R})^{\infty}\rightarrow\Omega^{2,s}(\mathcal{R}^{\infty})\big)}
{\operatorname{Im}\big({\rm d}_H\colon \Omega^{0,s}(\mathcal{R}^{\infty})\rightarrow\Omega^{1,s}(\mathcal{R}^{\infty})\big)}
\end{gather*}
is made up of cohomology classes whose representatives are type $(1,s)$ conservation laws, and the cohomology space
\begin{gather*}
H^{2,s}(\mathcal{R}^{\infty},{\rm d}_H)=\frac{\operatorname{Ker}\big({\rm d}_H\colon \Omega^{2,s}(\mathcal{R})^{\infty}\rightarrow\Omega^{3,s}(\mathcal{R}^{\infty})\big)}
{\operatorname{Im}\big({\rm d}_H\colon \Omega^{1,s}(\mathcal{R}^{\infty})\rightarrow\Omega^{2,s}(\mathcal{R}^{\infty})\big)}
\end{gather*}
consists of cohomology classes whose representatives are type $(2,s)$ conservation laws.

\section{Conservation laws for a particular class of semi-linear PDEs}\label{conservationlawssystems}

From this point on we will be considering involutive semi-linear second order systems of PDEs specifically of the form
\begin{gather}\label{system}
F_{ij}\big(x^1,x^2,x^3,u,u_i,u_j,u_{ij}\big)=u_{ij}-f_{ij}\big(x^1,x^2,x^3,u,u_i,u_j\big)=0,
\end{gather}
where $1\leq i,j\leq3$ and $i\neq j$.\footnote{This class of PDEs falls into one of the five classes defined in Cartan's structural classification of involutive systems of three PDEs in one dependent and three independent variables, as originally published in~\cite{ca11} and outlined in Appendix~\ref{structuralclassification}.} A basis for the space of total vector fields defined on the infinite prolongation $\mathcal{R}^{\infty}$ of the equation manifold $\mathcal{R}$ defined by~(\ref{system}) is given by any set of three linearly independent vector fields
 \begin{gather*}
 X_i=m_1^iD_1+m_2^iD_2+m_3^iD_3,\qquad 1\leq i\leq3,
 \end{gather*}
 where the $D_i$ are the total vector fields on $J^{\infty}(E)$ restricted to $\mathcal{R}^{\infty}$, defined by~(\ref{totalderivativedef}).

The system (\ref{system}) is hyperbolic in the sense that its associated exterior differential system has three distinct characteristics. In particular, the characteristic equation for each $F_{ij}=0$ is just $\lambda\mu=0$ which has the roots $(\lambda,\mu)=(1,0)$ and $(\lambda,\mu)=(0,1)$ leading us to associate the total vector fields $D_i$ and $D_j$ to it . Thus we may take as our basis for the space of total vector fields on $\mathcal{R}^{\infty}$ the characteristic vector fields $D_1$, $D_2$ and $D_3$. These vector fields have the particularly convenient property of commuting with each other, in other words: \begin{gather*}[D_i,D_j]=0\end{gather*} for $1\leq i\neq j\leq3$.

Many of the results in this paper can be expected to remain true for more general involutive systems of the form
\begin{gather}\label{moregeneralsystem}
F_{ij}\big(x^1,x^2,x^3,u,u_i, u_j, u_{ii}, u_{ij}, u_{jj}\big)=0\qquad \textrm{for}\quad 1\leq i< j\leq3
\end{gather}
with the property that the universal linearization of (\ref{moregeneralsystem}), defined in Section~\ref{universallinearization-section}, is of the form
\begin{gather*}\mathcal{L}_{ij}(\theta)=X_iX_j(\theta)+A_{ij}^iX_i(\theta)+A_{ij}^jX_j(\theta)+C_{ij}\theta=0,\qquad 1\leq i< j\leq3,\end{gather*}
where $\theta$ is a contact form on $\mathcal{R}^{\infty}$ and the total vector fields $X_i$ are given by
\begin{gather}\label{generalvf}
X_i=\sum_{j=1}^3a_{ij}\big(x^1,x^2,x^3\big)D_j, \qquad 1\leq i\leq3\qquad \textrm{and}\qquad \det(a_{ij})\neq0,
\end{gather}
 and are not assumed to commute pairwise. Also note that we are \textit{not} using the Einstein summation convention anywhere in this paper.

Since the vector fields (\ref{generalvf}) form a basis for the space of total vector fields on $\mathcal{R}^{\infty}$, the commutator of $X_i$ and $X_j$ can be written as a linear combination of these, \begin{gather*}[X_i,X_j]=\sum_{k=1}^3 B_{ij}^k\big(x^1,x^2,x^3\big)X_k.\end{gather*} Choosing to consider systems of the form (\ref{system}), for which we may let $X_i=D_i$, and utilizing the fact that these characteristics commute so that the coefficients $B_{ij}^k=0$ above, will significantly simplify the expressions for the universal linearization of~(\ref{system}), as well as the structure equations and Lie bracket congruences for the Laplace adapted coframe which are given in the following sections. Therefore we will assume henceforth that we are considering systems of the form~(\ref{system}) and have chosen the characteristic vector fields for~(\ref{system}) to be $X_i=D_i$, unless explicitly stated otherwise.

\subsection{Universal linearization}\label{universallinearization-section}
The first step in analyzing systems of the form (\ref{system}) will be to linearize each equation in the system.

On the equation manifold $\mathcal{R}^{\infty}$, the contact forms $\theta$, $\theta_i$, $\theta_j$ and $\theta_{ij}$ are not independent, but rather related to each other according to the equations obtained by taking ${\rm d}_V$ of each of the three equations $F_{ij}=0$,
\begin{gather*}%\label{univlin}
{\rm d}_VF_{ij}=\theta_{ij}+F_{ij,u_i}\theta_i+F_{ij,u_j}\theta_j+F_{ij,u}\theta=0,\qquad 1\leq i< j\leq3.
\end{gather*}
According to (\ref{Djtheta}), $D_j(\theta_I)=\theta_{Ij}$, so it is straightforward to write these equations in terms of the characteristic vector fields, $X_i$:
\begin{gather}\label{universallinearization}
X_iX_j(\theta)+a_{ij}^iX_i(\theta)+a_{ij}^jX_j(\theta)+c_{ij}\theta=0,
\end{gather}
where
\begin{gather*}%\label{provisionalform}
a_{ij}^i=F_{ij,u_i},\qquad a_{ij}^j=F_{ij,u_j},\qquad \textrm{and}\qquad c_{ij}=F_{ij,u}.
\end{gather*}

We will refer to equation (\ref{universallinearization}) as the universal linearization of $F_{ij}$, and denote it by $\mathcal{L}_{ij}(F_{ij})$, or simply~$\mathcal{L}_{ij}$ when the system whose linearization we are considering is clear from the context.

Eventually we will also want to have the freedom to rescale the contact form $\theta$ in order to manipulate the form of equation~(\ref{universallinearization}). To this end, define
\begin{gather*}%\label{rescaledtheta}
\Theta=\mu\theta
\end{gather*}
for a non-vanishing function $\mu$ defined on $\mathcal{R}^{\infty}$. Then the universal linearization~(\ref{universallinearization}) can be written equivalently as
\begin{gather}\label{scaleduniversallinearization}
\mathcal{L}_{ij}(\Theta)=X_iX_j(\Theta)+A^i_{ij}X_i(\Theta)+A^j_{ij}X_j(\Theta)+C_{ij}\Theta=0,
\end{gather}
where
\begin{gather}\label{sulcoef1}
A^i_{ij}=a^i_{ij}-\frac{X_j(\mu)}{\mu},\\
\label{sulcoef2}
A^j_{ij}=a^j_{ij}-\frac{X_i(\mu)}{\mu},\\
\label{sulcoef3}
C_{ij}=c_{ij}-\frac{X_iX_j(\mu)}{\mu}-a^i_{ij}\frac{X_i(\mu)}{\mu}-a^j_{ij}\frac{X_j(\mu)}{\mu}+2\frac{X_i(\mu)X_j(\mu)}{\mu^2}.
\end{gather}

An example of a non-linear involutive system of the form (\ref{system}), can be found in Vassiliou~\cite{va10}. The linearization of this system is described in the example below.

\begin{Example}Consider the following involutive system of the form (\ref{system})
\begin{gather}\label{nonlinearexample}
u_{12}=\frac{2u+1}{u(u+1)}u_1u_2, \qquad u_{13}=\frac{u_1u_3}{u+1},\qquad u_{23}=\frac{u_2u_3}{u}.
\end{gather}

The corresponding linearized system is
\begin{gather*}
\mathcal{L}_{12}(\theta)=X_1X_2(\theta)-\frac{2u+1}{u(u+1)}u_2X_1(\theta)-\frac{2u+1}{u(u+1)}u_1X_2(\theta)+\frac{2u^2+2u+1}{u^2(u+1)^2}u_1u_2\theta,\\
\mathcal{L}_{13}(\theta)=X_1X_3(\theta)-\frac{u_3}{u+1}X_1(\theta)-\frac{u_1}{u+1}X_3(\theta)+\frac{u_1u_3}{(u+1)^2}\theta,\\
\mathcal{L}_{23}(\theta)=X_2X_3(\theta)-\frac{u_3}{u}X_2(\theta)-\frac{u_2}{u}X_3(\theta)+\frac{u_2u_3}{u^2}\theta.
\end{gather*}
\end{Example}

\subsection{Characteristic coframe}\label{characteristiccoframe}

Our next goal is to describe a set of one-forms which will form a coframe on $\mathcal{R}^{\infty}$. Begin by defining the $(1,0)$ forms $\sigma_i$ to be dual to the characteristic vector fields~$X_i$. That is, $\sigma_i(X_i)=1$ and $\sigma_i(X_j)=0$ for $i\neq j$. Then, since we will take $X_i=D_i$, the dual $(1,0)$ forms are $\sigma_i={\rm d}x^i$ and ${\rm d}_H\sigma_i=0$ for $1\leq i\leq 3$. More generally, if $\omega\in\Omega^{r,s}(\mathcal{R}^{\infty})$ is a type $(r,s)$ form, then
\begin{gather}\label{dhomega}
{\rm d}_H\omega=\sigma_1\wedge X_1(\omega)+\sigma_2\wedge X_2(\omega)+\sigma_3\wedge X_3(\omega),
\end{gather}
where the total vector fields $X_i$ act on $\omega$ by projected Lie differentiation, as defined in Section~\ref{constrainedbisec}. In particular, there will be several occasions where we wish to compute the horizontal exterior derivative of a type $(1,s)$ form. That is, if $\omega$ is given by \begin{gather*}\omega=\sigma_1\wedge M_1+\sigma_2\wedge M_2+\sigma_3\wedge M_3,\end{gather*} where each $M_i$ is a $(0,s)$ form on~$\mathcal{R}^{\infty}$, then
\begin{gather}
{\rm d}_H\omega =\sigma_1\wedge\sigma_2\wedge[X_2(M_1)-X_1(M_2)]+\sigma_2\wedge\sigma_3\wedge[X_3(M_2)-X_2(M_3)]\nonumber \\
\hphantom{{\rm d}_H\omega =}{} +\sigma_1\wedge\sigma_3\wedge[X_3(M_1)-X_1(M_3)]. \label{dh1s}
\end{gather}
Likewise, if $\omega$ is a $(2,s)$ form, $\omega=\sigma_2\wedge\sigma_3\wedge M_1+\sigma_1\wedge\sigma_3\wedge M_2+\sigma_1\wedge\sigma_2\wedge M_3$, where each $M_i$ is a $(0,s)$ form on $\mathcal{R}^{\infty}$, then
\begin{gather*}%\label{dhspecial}
{\rm d}_H\omega=\sigma_1\wedge\sigma_2\wedge\sigma_3\wedge[X_1(M_1)-X_2(M_2)+X_3(M_3)].
\end{gather*}
To complete the coframe, define the remaining contact forms inductively:
\begin{gather}\label{definecontact}
\xi_i^1=X_i(\Theta),\qquad \xi_i^2=X_iX_i(\Theta)=X_i\big(\xi_i^1\big),\qquad \ldots,\qquad \xi_i^k=X_i^k(\Theta)=X_i\big(\xi_i^{k-1}\big)
\end{gather}
for $i=1,2,3$. Before giving their structure equations and stating formally that the forms~(\ref{definecontact}), along with the previously defined forms $\{\sigma_1,\sigma_2,\sigma_3,\Theta\}$, do indeed make up a coframe on~$\mathcal{R}^{\infty}$, we need to pause to give the following definition and lemma.

\begin{Definition}A form $\omega\in\Omega^p(\mathcal{R}^{\infty})$ has \textit{adapted order} $k$ if it lies in the exterior algebra generated, over the smooth functions on $\mathcal{R}^{\infty}$, by the one-forms
\begin{gather*}
\big\{\sigma_1,\sigma_2,\sigma_3,\Theta,\xi^1_1,\xi_2^1,\xi_3^1,\ldots,\xi_1^k,\xi_2^k,\xi_3^k\big\},
\end{gather*}
 where $k$ is minimal.
\end{Definition}
The adapted order of a form $\omega$ is invariant under contact transformations on $\mathcal{R}^{\infty}$, and may differ from its order as a form on~$\mathcal{R}^{\infty}$.

The following lemma will be used to establish the structure equations for the one forms comprising the coframe given below.
\begin{Lemma}\label{adaptedorder}For $k\geq1$ and $i\neq j$, $X_i\big(\xi_j^k\big)$ restricted to $\mathcal{R}^{\infty}$ has adapted order less than or equal to $k$.
\end{Lemma}

\begin{proof}The proof is completed by induction on the adapted order, $k$. First we show the statement holds for $k=1$, that is
\begin{gather*}X_i(\xi_j^1)=X_iX_j(\Theta)=-A_{ij}^iX_i(\Theta)-A_{ij}^jX_j(\Theta)-C_{ij}\Theta,\end{gather*} which is evidently of adapted order $\leq 1$. Next assume that $X_i\big(\xi_j^l\big)$ has adapted order $\leq l$ for all $l\leq k$. Then $X_i\big(\xi_j^l\big)$ can be written as a linear combination \begin{gather*}X_i\big(\xi_j^l\big)=a\Theta+\sum_{l=1}^kb_l\xi_1^l+\sum_{l=1}^{k}c_l\xi_2^l+\sum_{l=1}^kd_l\xi_3^l\end{gather*} with $a$, $b_l$, $c_l$, $d_l$ are all functions on~$\mathcal{R}^{\infty}$. As mentioned previously in this section, we are free to choose the characteristics for systems of the form~(\ref{system}) to commute. Then \begin{gather*}X_i\big(\xi_j^{k+1}\big)=X_iX_j\big(\xi_j^k\big)=X_jX_i\big(\xi_j^k\big),\end{gather*} which has adapted order $\leq k+1$ as needed.
\end{proof}

We are now prepared to state and prove the following proposition.
\begin{Proposition}\label{charcoframe} The set of $1$-forms
\begin{gather}\label{firstcoframe}
\big\{\sigma_1,\sigma_2,\sigma_3,\Theta,\xi_1^1,\xi_2^1,\xi_3^1,\ldots,\xi_1^k,\xi_2^k,\xi_3^k,\ldots\big\}
\end{gather}
forms a coframe on the equation manifold $\mathcal{R}^{\infty}$, called the characteristic coframe. The ${\rm d}_H$ structure equations for this coframe are given by
\begin{gather}\label{structureeqnfirstcoframe1}
{\rm d}_H\sigma_i =0,\\
\label{structureeqnfirstcoframe2}
{\rm d}_H\Theta =\sigma_1\wedge\xi_1^1+\sigma_2\wedge\xi_2^1+\sigma_3\wedge\xi_3^1,\\
\label{structureeqnfirstcoframe3}
{\rm d}_H\xi_i^k =\sigma_i\wedge\xi_i^{k+1}+\sigma_j\wedge\mu_i^{k}+\sigma_l\wedge\nu_i^{k}\qquad \textrm{for}\quad j,l\neq i,
\end{gather}
where $\mu_i^k$ and $\nu_i^k$ are contact forms of adapted order $\leq k$.
\end{Proposition}

\begin{proof}Repeatedly applying the vector fields $X_i$ to the contact form $\Theta=\mu\theta$, and using equa\-tion~(\ref{Djtheta}), we see that
\begin{gather*}X_{i_1}X_{i_2}\cdots X_{i_k}(\Theta)=\mu\theta_{i_1i_2\dots i_k}+\delta,\end{gather*}
where $\delta$ consists of contact forms of order $< k$. Then the set of one forms
\begin{gather*}\big\{\Theta,X_1^i(\Theta),X_2^i(\Theta),X_3^i(\Theta),\ldots,X_1^mX_2^nX_3^l(\Theta),\ldots\big\},\end{gather*}
where at least two of the indices $m$, $n$, $l$ are $\geq1$, clearly spans the contact ideal~$\mathcal{C}(\mathcal{R}^{\infty})$. By Lemma~\ref{adaptedorder}, each $X_1^mX_2^nX_3^l(\Theta)$ can be expressed as a linear combination of the forms in~(\ref{firstcoframe}) of order $\leq m+n+l$. The structure equations~(\ref{structureeqnfirstcoframe1})--(\ref{structureeqnfirstcoframe3}) are a result of the definition of the forms $\xi_i^j$ given in~(\ref{definecontact}), formula (\ref{dhomega}), and Lemma~\ref{adaptedorder}.
\end{proof}

\subsection{Generalized Laplace method for systems}\label{genlapmetforsys}

We will now introduce the generalized Laplace method as a means of solving systems of the form~(\ref{scaleduniversallinearization}). The geometric origins of this method can be found in Chern's work~\cite{ch44} on the Laplace transformation of submanifolds admitting conjugate nets of curves. It was subsequently used in~\cite{kt96} to generalize the classical Laplace method (also described in~\cite{kt96}) to involutive overdetermined systems of linear equations in $n$ independent and 1 dependent variable of the form
\begin{gather}\label{kt96systems}
\frac{\partial^2y}{\partial x^l\partial x^k}+a_{lk}^l\frac{\partial y}{\partial x^l}+a_{lk}^k\frac{\partial y}{\partial x^k}+c_{lk}y=0, \qquad 1\leq l\neq k\leq n,
\end{gather}
where $a_{lk}^k$, $a_{lk}^l,$ and $c_{lk}$ are differentiable functions of the independent variables $x^1,\ldots,x^n$ and each set of functions is symmetric in its lower indices. The similarity between the form of the system~(\ref{kt96systems}) and that of the linearized system~(\ref{scaleduniversallinearization}) is apparent, and provides motivation for the description of the generalized Laplace method suited to studying the systems~(\ref{scaleduniversallinearization}) given below.

A linearized system of the form~(\ref{scaleduniversallinearization}) must satisfy certain integrability conditions stemming from the fact that $D_k(\theta_{ij})=D_j(\theta_{ik})$ for $i$, $j$, $k$ distinct. Applying $D_k$ to each equation $\mathcal{L}_{ij}(\Theta)=0$ and equating the coefficients on like-ordered contact forms leads to the following relations which the coefficients of any compatible system must satisfy for $1\leq i\neq j\neq k\leq3$,
\begin{gather}
D_j\big(A_{lk}^l\big)-D_k\big(A_{lj}^l\big)=0,\nonumber\\
\nonumber
D_j\big(A_{lk}^k\big)-A_{lk}^kA_{kj}^k+A_{lj}^lA_{lk}^k+A_{lj}^jA_{jk}^k=C_{lj},\\
D_j(C_{lk})-D_k(C_{lj})+A_{lj}^lC_{lk}+\big(A_{lj}^j-A_{lk}^k\big)C_{kj}-A_{lk}^lC_{lj}=0.\label{intconditions}
\end{gather}
These correspond to the compatibility conditions defined in \cite{kt96}, the only difference being that the partial derivatives found in the expressions written in~\cite{kt96} have been replaced with total derivatives in~(\ref{intconditions}).

\subsubsection{Laplace invariants and the generalized Laplace transform}

The classical Laplace method associates to a given linear hyperbolic PDE in the plane, $F(u, u_x$, $u_y, u_{xx}, u_{xy}, u_{yy})=0$, two Laplace invariants $h(F)$ and $k(F$). The vanishing of either invariant allows for the integration of the equation by quadratures, and the failure of these invariants to vanish allows one to apply the classical Laplace method in order to obtain a transformed equation whose Laplace invariants may or may not vanish. We will now develop an analogous procedure to investigate the systems (\ref{scaleduniversallinearization}). Begin by defining the following $n(n-1)^2$ expressions which are invariant under rescaling of $\Theta$ by a nonvanishing function on~$\mathcal{R}^{\infty}$, and will be referred to as the higher-dimensional Laplace invariants of the system
\begin{gather}\label{laplaceinvariants}
H_{ij}=D_i\big(A_{ij}^i\big)+A_{ij}^iA_{ij}^j-C_{ij}\qquad \textrm{for}\quad 1\leq i\neq j\leq n,\\
\label{laplaceinvariants2}
H_{ijk}=A_{kj}^k-A_{ij}^i\qquad \textrm{for}\quad 1\leq k\neq i,j\leq n.
\end{gather}
Note that when there are only two independent variables, and hence the system~(\ref{scaleduniversallinearization}) is replaced by a single equation, the invariants~(\ref{laplaceinvariants}) reduce to the classical Laplace invariants. In this case, the invariants~(\ref{laplaceinvariants2}) would not be defined.

\begin{Example}Recall the nonlinear example of an involutive system of the form (\ref{system}), given in equation (\ref{nonlinearexample}) and restated below,
\begin{gather*}%\label{nonlinearexample}
u_{12}=\frac{2u+1}{u(u+1)}u_1u_2,\qquad u_{13}=\frac{u_1u_3}{u+1},\qquad u_{23}=\frac{u_2u_3}{u}.
\end{gather*}

 The Laplace invariants for this system are as follows
\begin{gather*}
H_{ij} =0\qquad \textrm{for all}\quad i,j=1,2,3\quad \textrm{and}\quad i\neq j,\nonumber\\
H_{132}=-H_{231}=\frac{u_3}{u(u+1)},\qquad H_{213}=-H_{312}=\frac{u_1}{u},\qquad H_{123} =-H_{321}=\frac{u_2}{u+1}.
\end{gather*}
\end{Example}

We are now in a position to introduce the concept of the generalized Laplace transform. Suppose that $\Theta$ solves the system~(\ref{scaleduniversallinearization}). Then for any \textit{ordered} pair $(i,j)$, define the $(i,j)$ Laplace transform of $\Theta$ to be
\begin{gather}\label{ijtransformoftheta}
\xi_{ij}=X_j(\Theta)+A_{ij}^i\Theta,
\end{gather}
and denote this by $\xi_{ij}=\mathcal{X}_{ij,\mathcal{L}_{lk}}(\Theta)$, or simply $\mathcal{X}_{ij}(\Theta)$ if the system of differential opera\-tors,~$\mathcal{L}_{lk}, 1\leq l\neq k\leq 3$, being considered is clear from context. One should be aware that the order of the pair $(i,j)$ is significant, as we can see by comparing the $(i,j)$ Laplace transform of $\Theta$, given by~(\ref{ijtransformoftheta}), with the $(j,i)$ Laplace transform of $\Theta$, which would be
\begin{gather*}%\label{jitransformoftheta}
\xi_{ji}=\mathcal{X}_{ji}(\Theta)=X_i(\Theta)+A_{ij}^j\Theta.
\end{gather*}
 The following proposition shows that the transformed contact form $\xi_{ij}=\mathcal{X}_{ij}(\Theta)$ will solve a~system of the same form as~(\ref{scaleduniversallinearization}).

\begin{Proposition}\label{thetahatsolves}Let $\xi_{ij}$ be the $(i,j)$-Laplace transform of $\Theta$, where $\Theta$ satisfies the system of equations~\eqref{scaleduniversallinearization}. In the case that the $(i,j)$ Laplace invariants $H_{ij}$ and $H_{ijk}$ are nonzero for all $k, k\neq i, j$, $\xi_{ij}$ will satisfy a system of equations of the same form as~\eqref{scaleduniversallinearization}.
\end{Proposition}
\begin{proof}We can compute expressions for the total derivatives of $\xi_{ij}$ directly from the definition of the $(i,j)$ Laplace transform and the system~(\ref{scaleduniversallinearization}):
\begin{gather}\label{dihattheta}
X_i(\xi_{ij}) = -A_{ij}^j\Theta_j+\big(H_{ij}-A^i_{ij}A_{ij}^j\big)\Theta,\\
\label{djhattheta}
X_j(\xi_{ij})=X_j\big(A_{ij}^i\big)\Theta+A_{ij}^i\Theta_j+\Theta_{jj},\\
\label{dkhattheta}
X_k(\xi_{ij})=\big(H_{kj}-A_{kj}^jA_{ij}^i\big)\Theta-H_{kk}\Theta_k-A_{jk}^j\Theta_j
\end{gather}
for $k\neq i,j$. Plugging (\ref{ijtransformoftheta}) into (\ref{dihattheta}), we obtain
\begin{gather}\label{yintermsof}
\Theta=\frac{1}{H_{ij}}\big(X_i(\xi_{ij})+A_{ij}^j\xi_{ij}\big)
\end{gather}
and then applying $X_j$ to both sides of (\ref{yintermsof}),
\begin{gather}\label{yjintermsof}
\Theta_j=\frac{1}{H_{ij}}\big(\big(H_{ij}-A_{ij}^iA_{ij}^j\big)\xi_{ij}-A_{ij}^iX_i(\xi_{ij})\big).
\end{gather}
Finally, expressions for $\Theta_{k}$ and $\Theta_{jj}$ are gotten by solving for $\Theta_k$ and $\Theta_{jj}$ in (\ref{dkhattheta}) and (\ref{djhattheta}) respectively, and making the appropriate substitutions using~(\ref{yintermsof}) and~(\ref{yjintermsof}):
\begin{gather*}
\Theta_k = \frac{1}{H_{kk}}\left(-X_k(\xi_{ij})+\frac{H_{kj}}{H_{ij}}X_i(\xi_{ij})+\left(\frac{H_{kj}}{H_{ij}}A_{ij}^j-A_{jk}^j\right)\xi_{ij}\right),\\
\Theta_{jj} = X_j(\xi_{ij})-\frac{1}{H_{ij}}\big[\big(X_j\big(A_{ij}^i\big)-A_{ij}^iA_{ij}^i\big)X_i(\xi_{ij}) +\big(X_j\big(A_{ij}^i\big)A_{ij}^j+A_{ij}^i\big(H_{ij}-A_{ij}^iA_{ij}^j\big)\big)\xi_{ij}\big].
\end{gather*}
Now we may take the total derivatives of (\ref{dihattheta})--(\ref{dkhattheta}) to see that indeed $\xi_{ij}$ satisfies a system of the form
\begin{gather}\label{equationforhattheta}
X_lX_k(\xi_{ij})+\hat{A}_{lk}^lX_l(\xi_{ij})+\hat{A}_{lk}^kX_k(\xi_{ij})+\hat{C}_{lk}\xi_{ij}=0
\end{gather}
for $1\leq l\neq k\leq 3$. The coefficients in the equation (\ref{equationforhattheta}) can be given explicitly in terms of the coefficients of the original system as follows,
\begin{align}\label{coeftrans1}
\hat{A}_{ij}^i&=A_{ij}^i-\frac{X_j(H_{ij})}{H_{ij}},\qquad \hat{A}_{ij}^j=A_{ij}^j,\\
\label{coeftrans2}
\hat{C}_{ij}&=A_{ij}^jA_{ij}^i+H_{ij}\left(X_j\left(\frac{A_{ij}^i}{H_{ij}}\right)-1\right).
\end{align}
And for $k\neq i,j$ we have
\begin{gather}\label{coeftrans3}
\hat{A}_{ik}^i=-X_k(\log H_{ij})+A_{ik}^i, \qquad \hat{A}_{ik}^k=A_{ij}^j+\frac{H_{ij}}{H_{ijk}},\\
\label{coeftrans4}
\hat{C}_{ik}=H_{ij}\left(X_k\left(\frac{A_{ij}^j}{H_{ij}}\right)+\frac{A_{jk}^j}{H_{ijk}}\right)+A_{ij}^jA_{ik}^i,\\
\label{coeftrans5}
\hat{A}_{jk}^j =A_{jk}^j,\qquad \hat{A}_{jk}^k=A_{jk}^k-X_j(\log H_{ijk}),\\
\label{coeftrans6}
\hat{C}_{jk}=2A_{ij}^jX_k\big(A_{ij}^i\big)\frac{H_{ijk}}{H_{ij}}+A_{jk}^j\big(\hat{A}_{jk}^k-H_{ijk}\big)-H_{kj}+X_j\big(A_{jk}^j\big).
\end{gather}
This proves the proposition.
\end{proof}

According to the preceding proposition then, when a given system has non-vanishing $(i,j)$ Laplace invariants, the system of three total differential operators $\mathcal{L}_{lk}$, $1\leq l\neq k\leq3$, is transformed under the $(i,j)$ Laplace transform into another system of three total differential operators of the same form. Denote the operator that $\mathcal{L}_{lk}$ is transformed into under the $(i,j)$ Laplace transform by $\mathcal{X}_{ij}(\mathcal{L}_{lk})$. Then the $(i,j)$ Laplace transform of~$\Theta$, $\mathcal{X}_{ij}(\Theta)=\xi_{ij}$, will satisfy the system $[\mathcal{X}_{ij}(\mathcal{L}_{lk})](\xi_{ij})=0$ for all $1\leq l\neq k \leq3$. When we wish to emphasize that the $(i,j)$ Laplace transform is being performed, we will write~(\ref{equationforhattheta}) as
\begin{gather*}
X_lX_k(\xi_{ij})+\mathcal{X}_{ij}\big(A_{lk}^l\big)X_l(\xi_{ij})+\mathcal{X}_{ij}\big(A_{lk}^k\big)X_k(\xi_{ij})+\mathcal{X}_{ij}(C_{lk})\xi_{ij}=0.
\end{gather*}

Just as in the case of the classical Laplace method, the generalized Laplace transform has an inverse when the $(i,j)$ invariant is nonzero, as the next result shows.

\begin{Proposition}Let $\xi_{ij}=\mathcal{X}_{ij}(\Theta)$ be the $(i,j)$ Laplace transform of~$\Theta$, which satisfies the system~\eqref{scaleduniversallinearization}. If $H_{ij}\neq0$ then an inverse Laplace transform of the $(i,j)$ Laplace transform exists and is given by
\begin{gather*}
\Theta=\frac{1}{H_{ij}}\left(\mathcal{X}_{ji}(\xi_{ij})\right).
\end{gather*}
\end{Proposition}
\begin{proof}By the definition of the Laplace transform, $\mathcal{X}_{ji}(\xi_{ij})=X_i(\xi_{ij})+\hat{A}_{ij}^j\xi_{ij}$. Substitu\-ting~(\ref{ijtransformoftheta}) and~(\ref{dihattheta}) into this expression and using the fact that $\hat{A}_{ij}^j=A_{ij}^j$, which follows from~(\ref{intconditions}), we see that $\mathcal{X}_{ji}(\xi_{ij})=H_{ij}\Theta$. Since $H_{ij}\neq0$, this can be solved for $\Theta=\frac{1}{H_{ij}}\mathcal{X}_{ji}(\xi_{ij})$, as we wished to show.
\end{proof}

\subsection{Laplace adapted co-frame for systems of nonlinear PDEs}\label{Laplace Coframe Section}

We can now introduce another coframe on the equation manifold $\mathcal{R}^{\infty}$ which is constructed by utilizing the generalized Laplace transforms, $\mathcal{X}_{ij}$, defined in Section~\ref{genlapmetforsys}. This coframe will be denoted by
\begin{gather*}
\big\{\sigma_1, \sigma_2,\sigma_3,\Theta,\hat{\xi}_1^1,\hat{\xi}_2^1,\hat{\xi}_3^1,\ldots,\hat{\xi}_1^k,\hat{\xi}_2^k,\hat{\xi}_3^k,\ldots\big\}.
\end{gather*}
Before describing the elements of this coframe explicitly, we will pause to establish some necessary terminology and notation.

As described in Proposition~\ref{thetahatsolves}, if the $(i,j)$ Laplace invariants of a given system (\ref{scaleduniversallinearization}) are nonzero, then the system may be transformed into another system, written $\mathcal{X}_{ij}(\mathcal{L}_{lk})(\xi_{ij})=0$, of the same form. The $(i,j)$ Laplace invariants of that transformed system may likewise be computed. Denote these by $H_{ij}^1=H_{ij}(\mathcal{X}_{ij}(\mathcal{L}_{lk}))$ and $H_{ijq}^1=H_{ijq}(\mathcal{X}_{ij}(\mathcal{L}_{lk}))$. The process may be repeated so long as the $(i,j)$ Laplace invariants of the last system do not vanish. Assuming all the necessary Laplace invariants are nonzero, let the system of differential operators obtained by applying the $(i,j)$ Laplace transform $n$ times to the original system~(\ref{scaleduniversallinearization}) be denoted by~$\mathcal{X}_{ij}^n(\mathcal{L}_{lk})$. Accordingly, we will write
\begin{gather}\label{higherindexnotation}
H_{ij}^n=H_{ij}(\mathcal{X}_{ij}^n(\mathcal{L}_{lk}))\qquad \textrm{and}\qquad H_{ijq}^n=H_{ijq}(\mathcal{X}_{ij}^n(\mathcal{L}_{lk}))
\end{gather}
to refer to the $(i,j)$ Laplace invariants of the system of differential operators obtained by applying the $(i,j)$ Laplace transform to the original system (\ref{scaleduniversallinearization}) $n$ times. With this notation then, we will write $H_{ij}^0$ and $H_{ijq}^0$ to represent the Laplace invariants of (\ref{scaleduniversallinearization}) itself.

Using equations (\ref{coeftrans1})--(\ref{coeftrans6}) and the notation established in~(\ref{higherindexnotation}), we can write out explicitly the coefficients of the operator $\mathcal{X}_{ij}^{n}(\mathcal{L}_{lk})$ defining the $n^{\rm th}$ application of the $(i,j)$ Laplace transform to the system of equations $\mathcal{L}_{lk}=0$. Write
\begin{gather*}
\mathcal{X}_{ij}^n(\mathcal{L}_{lk})=X_iX_j+\mathcal{X}_{ij}^n\big(A_{lk}^l\big)X_i+\mathcal{X}_{ij}^n\big(A_{lk}^k\big)X_j+\mathcal{X}_{ij}^n(C_{lk})
\end{gather*}
and the coefficients $\mathcal{X}_{ij}^n\big(A_{lk}^l\big)$, $\mathcal{X}_{ij}^n\big(A_{lk}^k\big)$, and $\mathcal{X}_{ij}^n(C_{lk})$ can be computed in terms of the coefficients of the original system (\ref{scaleduniversallinearization}) as follows. Take the coefficient $\mathcal{X}_{ij}^n\big(A_{lk}^l\big)$ and proceed by induction, first expressing it in terms of the coefficient on~$X_i$ in the operator $\mathcal{X}_{ij}^{n-1}(\mathcal{L}_{lk})$:
\begin{gather*}
\mathcal{X}_{ij}^n\big(A_{lk}^l\big) = \mathcal{X}_{ij}^{n-1}\big(A_{lk}^l\big)-X_j\big(\log H_{ij}^{n-1}\big) \\
\hphantom{\mathcal{X}_{ij}^n\big(A_{lk}^l\big)}{} = \mathcal{X}_{ij}^{n-2}\big(A_{lk}^l\big)-X_j\big(\log H_{ij}^{n-2}\big)-X_j\big(\log H_{ij}^{n-1}\big) =\cdots\\
\hphantom{\mathcal{X}_{ij}^n\big(A_{lk}^l\big)}{} = A_{lk}^l- X_j\big(\log H_{ij}^0H_{ij}^1\cdots H_{ij}^{n-1}\big).
\end{gather*}
Likewise, it is straightforward to see that
\begin{gather*}
\mathcal{X}_{ij}^n\big(A_{lk}^k\big)=A_{lk}^k.
\end{gather*}

The following definition establishes some further notation concerning the generalized Laplace invariants.

\begin{Definition}Let $\operatorname{ind}(\mathcal{X}_{ij})=p_{ij}$ denote the number of times the $(i,j)$ Laplace transform must be applied to the system (\ref{scaleduniversallinearization}) in order to obtain vanishing $(i,j)$ Laplace invariants. That is, $H(\mathcal{X}_{ij}^{p_{ij}}(\mathcal{L}_{lk}))=0$. If the $(i,j)$ Laplace invariants never vanish despite repeated applications of the $(i,j)$ Laplace transform, then we will write $\operatorname{ind}(\mathcal{X}_{ij})=p_{ij}=\infty$.
\end{Definition}

With this notation in mind, we may now introduce the Laplace adapted coframe as the set of one-forms
\begin{gather}\label{laplacecoframedef}
\big\{\sigma_1, \sigma_2,\sigma_3,\Theta,\hat{\xi}_1^1,\hat{\xi}_2^1,\hat{\xi}_3^1,\ldots,\hat{\xi}_1^k,\hat{\xi}_2^k,\hat{\xi}_3^k,\ldots\big\},
\end{gather}
where, for $j=1,2,3$,
\begin{gather}
\hat{\xi}^1_j =X_j(\Theta)+A_{ij}^i\Theta ,\nonumber\\
\label{xinj} \hat{\xi}^n_j=X_j\big(\hat{\xi}_j^{n-1}\big)+\mathcal{X}_j^{n-1}\big(A_{ij}^i\big)\hat{\xi}_j^{n-1},\qquad \textrm{for}\quad n=2,\ldots,p_{ij}+1, \\
\hat{\xi}_j^{p_{ij}+n} =X_j\big(\hat{\xi}_j^{p_{ij}+n-1}\big),\qquad \textrm{for}\quad n\geq2 .\nonumber
\end{gather}
And for $i\neq j$,
\begin{gather}\label{xixij1}
X_i\big(\hat{\xi}_j^1\big) =H_{ij}^0\Theta-A_{ij}^j\hat{\xi}_j^1,\\
X_i\big(\hat{\xi}_j^n\big) =H_{ij}^{n-1}\hat{\xi}_j^{n-1}-\mathcal{X}_j^{n-1}\big(A_{ij}^j\big)\hat{\xi}_j^n,\qquad \textrm{for}\quad n=1,\ldots,p_{ij},\\
X_i\big(\hat{\xi}_j^{p_{ij}+n}\big) \equiv-\big(A_{ij}^j\big)^{p_{ij}}\hat{\xi}_j^{p_{ij}+n}\quad \mod\big\{\hat{\xi}_j^{p_{ij+1}},\ldots,\hat{\xi}_j^{p_{ij}+n-1}\big\}\qquad \textrm{for}\quad n>p_{ij}.\label{xixijpijplusn}
\end{gather}

The ${\rm d}_H$ structure equations for the Laplace adapted coframe (\ref{laplacecoframedef}) can be computed directly using equation~(\ref{dhomega}) and, given their complexity, are relegated to Appendix~\ref{appendixB}.

Define the vertical vector fields dual to the contact one forms $\Theta$ and $\hat{\xi}_j^n$ by
\begin{gather}\label{dualvectorfields}
\Theta(U)=1,\qquad \Theta\big(V_k^l\big)=0,\qquad \hat{\xi}_j^n(U)=0,\qquad \hat{\xi}_j^n\big(V_k^l\big)=\delta_k^j\delta_l^n,
\end{gather}
where $\delta_k^j$ and $\delta_l^n$ are both Kronecker delta functions. So, for example, $\hat{\xi}_1^n\big(V^n_1\big){=}1$ and $\hat{\xi}_2^n\big(V_3^n\big){=}0$. Then we can express the ${\rm d}_H$ structure equations above in terms of the Lie brackets of the characteristic vector fields $X_i$ and the vertical vector fields, $U$ and $V_k^l$. These congruences will be needed to prove some important results in the coming sections and have been provided in Appendix~\ref{appendixB} for the reader's reference.

\subsubsection[Adjoint of the linearized operator $\mathcal{L}_{ij}$]{Adjoint of the linearized operator $\boldsymbol{\mathcal{L}_{ij}}$}

We will now introduce the adjoint to the linearized operator $\mathcal{L}_{ij}$, which will play a pivotal role in subsequent sections. Let us begin with the following definition.
\begin{Definition}\label{defofadj}For a total differential operator $\mathcal{F}$, its formal adjoint operator, denoted by~$\mathcal{F}^*$, is the total differential operator such that for every $\rho\in\Omega^{0,s}$ and $\omega\in\Omega^{0,r}$ there exists $\gamma\in\Omega^{2,s+r}$ such that
\begin{gather*}%\label{defofadjeq}
[\rho\wedge\mathcal{F}(\omega)-\mathcal{F}^*(\rho)\wedge\omega]\wedge\sigma_1\wedge\sigma_2\wedge\sigma_3={\rm d}_H\gamma.
\end{gather*}
\end{Definition}

Using this definition, we will state the following proposition which describes the adjoint $\mathcal{L}_{ij}^*$ of the linearized operator~$\mathcal{L}_{ij}$. The proof of this result can be found in~\cite{vi841}.
\begin{Proposition}\label{adjproposition}The differential operator $\mathcal{L}_{ij}$ which defines the universal linearization~\eqref{scaleduniversallinearization} has the following adjoint operator
\begin{gather*}%\label{adjeqn}
\mathcal{L}_{ij}^*=X_iX_j+\big(A_{ij}^i\big)^*X_i+\big(A_{ij}^j\big)^*X_j+(C_{ij})^*,
\end{gather*}
where
\begin{gather*}%\label{adjpropcoef}
\big(A_{ij}^i\big)^*=-A_{ij}^i,\qquad \big(A_{ij}^j\big)^*=-A_{ij}^j,\qquad \textrm{and}\qquad (C_{ij})^*=C_{ij}-X_i\big(A_{ij}^i\big)-X_j\big(A_{ij}^j\big).
\end{gather*}
\end{Proposition}

\subsubsection{Characteristic invariant contact forms}

Characteristic invariant contact forms play an important role in the construction of conservation laws, as we will see in Sections~\ref{structuretheoremsec} and~\ref{generating1sconslaws}. First we will define invariant functions and contact forms, and then analogously, relative invariant contact forms. Then an important result, Proposition~\ref{vanishingindices}, regarding the existence of relative invariant contact forms will be stated and proved.

\begin{Definition}For a total vector field $X$ on the equation manifold $\mathcal{R}^{\infty}$, a function $f$ defined on $\mathcal{R}^{\infty}$ is said to be an $X$ invariant function if $X(f)=0$.
\end{Definition}

Likewise for contact forms, we have

\begin{Definition}A type $(0,s)$ contact form $\omega$ is said to be invariant with respect to the total vector field $X$, or equivalently called an~$X$ invariant contact form, if $X(\omega)=X\,\lrcorner \,{\rm d}_H\omega=0$. Furthermore, if $\omega$ is invariant with respect to two distinct total vector fields, $X$ and $Y$, then we will say that $\omega$ is an $X$ and $Y$ invariant contact form.
\end{Definition}
If a slightly weaker condition is satisfied, then we have a contact form which is a \textit{relative} $X$ invariant, as described in the following definition.

\begin{Definition}For $X$ a total vector field, the type $(0,s)$ contact form $\omega$ is said to be a~rela\-ti\-ve~$X$ invariant contact form if $X(\omega)=\lambda\omega$ for some function~$\lambda$ on $\mathcal{R}^{\infty}$. If $\omega$ is invariant relative to two distinct total vector fields, $X$ and $Y$, then we say that $\omega$ is a relative~$X$ and~$Y$ invariant contact form.
\end{Definition}

Let $\mathcal{M}$ be a subring of $C^{\infty}(\mathcal{R}^{\infty})$ and let $\{\omega_1,\omega_2,\ldots,\omega_i,\ldots\}$ denote a collection of one forms. Then we will write \begin{gather*}\Omega_{\mathcal{M}}^s(\omega_1,\omega_2,\ldots,\omega_i,\ldots)\end{gather*} to denote the $\mathcal{M}$ module of $s$ forms generated by $\omega_{i_1}\wedge\omega_{i_2}\wedge\cdots\wedge\omega_{i_s}$. In the case that $\mathcal{M}=C^{\infty}(\mathcal{R}^{\infty})$, we will simply write
\begin{gather*}\Omega^s(\omega_1,\omega_2,\ldots,\omega_i,\ldots).\end{gather*}
Given this notation, we may state and prove the following proposition.

\begin{Proposition}\label{vanishingindices}Let $\mathcal{R}$ be a system of three hyperbolic equations with the characteristic vector fields $X_1$, $X_2$, and $X_3$, Laplace indices $\operatorname{ind}(\mathcal{X}_{ij})=p_{ij}$ where $H_{ij}^{p_{ij}}=0$, and Laplace-adapted coframe $\big\{\sigma_1,\sigma_2,\sigma_3,\Theta,\hat{\xi}_1^1, \hat{\xi}_2^1,\hat{\xi}_3^1,\ldots,\hat{\xi}_1^j,\hat{\xi}_2^j,\hat{\xi}_3^j,\ldots\big\}$. Then for $s\geq 1$ we can make the following conclusions.
\begin{enumerate}\itemsep=0pt
 \item[$1.$] If $\omega\in\Omega^{0,s}(\mathcal{R}^{\infty})$ is a relative $X_1$ and $X_2$ invariant form, then
 \begin{gather} \label{x12invforms}
 \omega\in\Omega^s\big(\hat{\xi}_3^{p_3+1},\hat{\xi}_3^{p_3+2},\ldots\big).
 \end{gather}
where $p_3=\min\{p_{13},p_{23}\}$. Furthermore, if $\operatorname{ind}(\mathcal{X}_{13})=\operatorname{ind}(\mathcal{X}_{23})=\infty$, then there do not exist any nonzero $(0,s)$ forms which are relative $X_1$ and $X_2$ invariant.
 \item[$2.$] If $\omega\in\Omega^{0,s}(\mathcal{R}^{\infty})$ is a relative $X_2$ and relative $X_3$ invariant form, then
 \begin{gather} \label{x23invforms}
 \omega\in\Omega^s\big(\hat{\xi}_1^{p_1+1}, \hat{\xi}_1^{p_1+2},\ldots\big),
 \end{gather}
 where $p_1=\min\{p_{21},p_{31}\}$. Furthermore, if $\operatorname{ind}(\mathcal{X}_{21})=\operatorname{ind}(\mathcal{X}_{31})=\infty$, then there do not exist any nonzero $(0,s)$ forms which are relative $X_2$ and $X_3$ invariant.
 \item[$3.$] If $\omega\in\Omega^{0,s}(\mathcal{R}^{\infty})$ is a relative $X_3$ and relative $X_1$ invariant form, then
 \begin{gather} \label{x13invforms}
 \omega\in\Omega^s\big(\hat{\xi}_2^{p_2+1},\hat{\xi}_2^{p_2+2},\ldots\big),
 \end{gather}
 where $p_2=\min\{p_{12},p_{32}\}$. Furthermore, if $\operatorname{ind}(\mathcal{X}_{12})=\operatorname{ind}(\mathcal{X}_{32})=\infty$, then there do not exist any nonzero $(0,s)$ forms which are relative $X_3$ and $X_1$ invariant.
\end{enumerate}
\end{Proposition}

\begin{proof}Suppose that $\omega$ is a type $(0,s)$ contact form of adapted order $k$ which is relative~$X_1$ and~$X_2$ invariant, so that $\omega\in\Omega^s\big(\Theta,\hat{\xi}_1^1,\hat{\xi}_2^1,\hat{\xi}_3^1,\ldots,\hat{\xi}_1^k,\hat{\xi}_2^k,\hat{\xi}_3^k\big)$ and $X_1(\omega)=\lambda_1\omega$ and \smash{$X_2(\omega)=\lambda_2\omega$} for some functions $\lambda_1,\lambda_2\in C^{\infty}(\mathcal{R}^{\infty})$. We will begin by showing that $\omega$ in fact lies in $\Omega^s\big(\hat{\xi}_3^1,\hat{\xi}_3^2,$ $\ldots,\hat{\xi}_3^k\big)$. To do this, recall the vertical vector fields (\ref{dualvectorfields}), denoted by~$U$ and~$V_k^l$, which we defined to be dual to the contact forms making up the Laplace-adapted coframe. By assumption, $V_1^{k+1}\,\lrcorner\,\omega=0$. Now take the interior product of $X_1(\omega)=\lambda_1\omega$ with~$V_1^{k+1}$. By Proposition~\ref{verhorprop}, $V_1^{k+1}\,\lrcorner\,(X_1(\omega))=\big[V_1^{k+1},X_1\big]\,\lrcorner\,\omega+X_1(V_1^{k+1}\,\lrcorner\,\omega)=\big[V_1^{k+1},X_1\big]\,\lrcorner\,\omega$. At the same time $V_1^{k+1}\,\lrcorner\,(X_1(\omega))=V_1^{k+1}\,\lrcorner\,(\lambda_1\omega)=0$. So we conclude that
\begin{gather}\label{v1bracketx1}
\big[V_1^{k+1},X_1\big]\,\lrcorner\,\omega=0.
\end{gather}
Now use the Lie bracket congruences (\ref{x1withv11}), (\ref{x1withv1n}) and (\ref{x1withv1n2}) to see that
\begin{gather*}
\big[X_1,V_1^1\big] \equiv -U\quad \mod\big\{X_1,X_2,X_3,V_1^1\big\},\\
\big[X_1,V_1^k\big] \equiv -V_1^{k-1}\quad \mod\big\{X_1,X_2,X_3,V_1^k\big\}\qquad \textrm{for}\quad k\geq2.
\end{gather*}
So (\ref{v1bracketx1}) reduces to $V_1^k\,\lrcorner\,\omega=0$. Next take the interior product of $X_1(\omega)$ with $V_1^k$ and deduce similarly that $V_1^{k-1}\,\lrcorner\,\omega=0$. Continue in this manner, taking interior products of the vertical vector fields $V_1^n$ with $X_1(\omega)$. Taking the interior product with $V_1^1$ yields $U\,\lrcorner\,\omega=0$ so that finally we have $U\,\lrcorner\,\omega=V_1^1\,\lrcorner\,\omega=\cdots=V_1^k\,\lrcorner\,\omega=0$. This shows that
\begin{gather*}%\label{xi2and3}
\omega\in\Omega^s\big(\hat{\xi}_2^1,\hat{\xi}_3^1,\hat{\xi}_2^2,\hat{\xi}_3^2,\ldots,\hat{\xi}_2^k,\hat{\xi}_3^k\big).
\end{gather*}

Since $\omega$ is also relative $X_2$ invariant, we can repeat the same procedure taking the interior products of $\omega$ with $V_2^{k+1},V_2^k,\ldots,V_2^1$ to conclude
\begin{gather}\label{xi3}
\omega\in\Omega^s\big(\hat{\xi}_3^1,\hat{\xi}_3^2,\ldots,\hat{\xi}_3^k\big).
\end{gather}
Without loss of generality, assume that $p_3=p_{13}$, take the interior product of $X_1(\omega)$ with $U, V_3^1,\ldots,V_3^{p_{13}-1}$ and utilize the Lie bracket congruences (\ref{x1withu}) and (\ref{x1withv31})--(\ref{x1withv3n}). Since $U\,\lrcorner\,\omega=0$, taking $U\,\lrcorner\,X_1(\omega)$ implies $H_{13}^0V_3^1\,\lrcorner\,\omega=0$. Because $H_{13}^0$ is nonzero, we see that $V_3^1\,\lrcorner\,\omega=0.$ Next take $V_3^1\,\lrcorner\,X_1(\omega)$ to obtain $V_3^2\,\lrcorner\,\omega=0$. Continue taking interior products until we see that
\begin{gather}\label{v3zeros}
V_3^1\,\lrcorner\,\omega=V_3^2\,\lrcorner\,\omega=\cdots=V_3^{p_{13}}\,\lrcorner\,\omega=0.
\end{gather}
Nothing can be concluded by taking the interior product of $X_1(\omega)$ with $V_3^{p_3}$ since $H_{13}^{p_{13}}=0$. Then equations (\ref{xi3}) and (\ref{v3zeros}) imply (\ref{x12invforms}).

Finally, suppose $\operatorname{ind}(\mathcal{X}_{13})=\operatorname{ind}(\mathcal{X}_{23})=\infty$. Then clearly if $\omega$ is a relative~$X_1$ invariant form then the preceding argument shows $\omega$ must equal 0. The proofs of (\ref{x23invforms}) and (\ref{x13invforms}) follow by analogous arguments, permuting the roles of $X_1$, $X_2$, and $X_3$ as needed.
\end{proof}

\subsection[Generating $(2,s)$ conservation laws]{Generating $\boldsymbol{(2,s)}$ conservation laws}\label{structuretheoremsec}

In this section we will investigate the construction of $(2,s)$ conservation laws from solutions to the adjoint equation of a given linearized system. Let $\mathcal{L}_{ij}$ denote the universal linearization of $F_{ij}=u_{ij}-f_{ij}\big(x^1,x^2,x^3,u,u_i,u_j\big)=0$, for $1\leq i < j\leq 3$, expressed as
\begin{gather*}
\mathcal{L}_{ij}(\theta)=X_iX_j(\theta)+A_{ij}^iX_i(\theta)+A_{ij}^jX_j(\theta)+C_{ij}\theta=0,
\end{gather*}
with coefficients given by (\ref{sulcoef1})--(\ref{sulcoef3}). Each total differential operator $\mathcal{L}_{ij}$ has an associated adjoint operator, written as
\begin{gather*}
\mathcal{L}_{ij}^*=X_iX_j+\big(A_{ij}^i\big)^*X_i+\big(A_{ij}^j\big)^*X_j+(C_{ij})^*
\end{gather*}
 and defined in Proposition \ref{adjproposition}. Let $\big\{\sigma_1,\sigma_2,\sigma_3,\Theta,\hat{\xi}_1^1,\hat{\xi}_2^1,\hat{\xi}_3^1,\ldots,\hat{\xi}_1^j,\hat{\xi}_2^j,\hat{\xi}_3^j,\ldots\big\}$ be the Laplace adapted coframe described in Section~\ref{Laplace Coframe Section}.

 Then for $s\geq 1$, we can define a map
\begin{gather*}\Psi\colon \ \Omega^{0,s-1}(\mathcal{R}^{\infty})\rightarrow\Omega^{2,s}(\mathcal{R}^{\infty})\end{gather*} as follows
\begin{gather}\label{psidef1}
\Psi_{12}(\rho_{12})=\frac{1}{2}\sigma_1\wedge\sigma_3\wedge\big[\Theta\wedge\psi_1^{12}+\hat{\xi}_1^1\wedge\rho_{12}\big]-\frac{1}{2} \sigma_2\wedge\sigma_3\wedge\big[\Theta\wedge\psi_2^{12}-\hat{\xi}_2^1\wedge\rho_{12}\big], \\
\label{psidef2}
\Psi_{23}(\rho_{23})=\frac{1}{2}\sigma_2\wedge\sigma_1\wedge\big[\Theta\wedge\psi_2^{23}+\hat{\xi}_2^1\wedge\rho_{23}\big]-\frac{1}{2}\sigma_3\wedge\sigma_1\wedge\big[\Theta\wedge\psi_3^{23}-\hat{\xi}_3^1\wedge\rho_{23}\big], \\
\label{psidef3}
\Psi_{13}(\rho_{13})=\frac{1}{2}\sigma_1\wedge\sigma_2\wedge\big[\Theta\wedge\psi_1^{13}+\hat{\xi}_1^1\wedge\rho_{13}\big]-\frac{1}{2}\sigma_3\wedge\sigma_2\wedge\big[\Theta\wedge\psi_3^{13}-\hat{\xi}_3^1\wedge\rho_{13}\big],
\end{gather}
where the $\psi_k^{ij}$ are defined by
\begin{alignat*}{3}
& \psi_1^{12} =X_1(\rho_{12})-A_{12}^2\rho_{12},\qquad && \psi_2^{12} =-X_2(\rho_{12})+A_{12}^1\rho_{12},& \\
&\psi_2^{23}=X_2(\rho_{23})-A_{23}^3\rho_{23},\qquad && \psi_3^{23} =-X_3(\rho_{23})+A_{23}^2\rho_{23}, & \\
& \psi_1^{13}=X_1(\rho_{13})-A_{13}^3\rho_{13},\qquad && \psi_3^{13}=-X_3(\rho_{23})+A_{13}^1\rho_{13}. &
\end{alignat*}
We may now state the following theorem whose proof will employ an integration by parts type of argument to express any ${\rm d}_H$-closed $(2,s)$ form in terms of solutions to the adjoint equation \begin{gather*}\sum_{1\leq i< j\leq3}\mathcal{L}^*_{ij}(\rho_{ij})=0.\end{gather*}

\begin{Theorem}\label{structuretheorem}Let $s\geq1$ and let $\omega\in\Omega^{2,s}(\mathcal{R}^{\infty})$ be a ${\rm d}_H$-closed form. Then there exist contact forms \begin{gather*}\rho_{ij}\in\Omega^{0,s-1}(\mathcal{R}^{\infty})\qquad \textrm{and}\qquad \gamma\in\Omega^{1,s}(\mathcal{R}^{\infty})\end{gather*} for $1\leq i < j\leq 3$, such that $\omega$ is given by
\begin{gather*}
\omega=\sum_{1\leq i < j\leq 3}\Psi_{ij}(\rho_{ij})+{\rm d}_H\gamma
\end{gather*}
 and the $\rho_{ij}$ satisfy the equation
\begin{gather*}
\sum_{1\leq i < j\leq 3}\mathcal{L}_{ij}^*(\rho_{ij})=0.
\end{gather*}
\end{Theorem}

The proof of Theorem~\ref{structuretheorem} will make use of the following lemma, which concerns the forms on~$J^{\infty}(E)$ that lie in the kernel of the inclusion map $\iota\colon \mathcal{R}^{\infty}\rightarrow J^{\infty}(E)$, where $\mathcal{R}^{\infty}$ is the infinitely prolonged equation manifold defined by the system of equations $F_{ij}=0$.

\begin{Lemma}\label{pullstozero} If $\omega\in\Omega^p(J^{\infty}(E))$ satisfies $\iota^*\omega=0$ then $\omega$ can be expressed as
\begin{gather*}%\label{intbypartsii}
\omega = \sum_{1\leq i < j\leq 3}\sum_{l,k}\alpha_{lk}^{ij}D_i^lD_j^kF_{ij}+\sum_{l,k,m}\alpha_{lkm}D_1^lD_2^kD_3^mF_{12}\\
\hphantom{\omega =}{} +\sum_{1\leq i < j\leq 3}\sum_{l,k}\beta_{lk}^{ij}\wedge {\rm d}_V\big(D_i^lD_j^kF_{ij}\big)+\sum_{l,k,m}\beta_{lkm}\wedge {\rm d}_V\big(D_1^lD_2^kD_3^mF_{12}\big),
\end{gather*}
where $\alpha_{lk}^{ij}, \alpha_{lkm}\in\Omega^p(J^{\infty}(E))$ and $\beta_{lk}^{ij},\beta_{lkm}\in\Omega^{p-1}(J^{\infty}(E))$.
\end{Lemma}

\begin{proof}Utilizing the system of equations (\ref{system}), there is a set of coordinates
\begin{gather}\label{coorjinfty}
\big(x^1,x^2,x^3,u,u_1,u_2,u_3,{\ldots},u_{1^k},u_{2^k},u_{3^k},F_{12},F_{23},F_{13},{\ldots},D_1^lD_2^kD_3^mF_{12},D_i^lD_j^kF_{ij},{\ldots}\big)\!\!\!\!\!
\end{gather}
 on $J^{\infty}(E)$ and a corresponding basis of one forms on $J^{\infty}$ consisting of
\begin{gather}\label{coframei}
\big({\rm d}x^1,{\rm d}x^2,{\rm d}x^3,\theta,\theta_1,\theta_2,\theta_3,\ldots,\theta_{1^j},\theta_{2^j},\theta_{3^j},\ldots\big)
\end{gather}
along with
 \begin{gather} \label{coframeii}
 {\rm d}_V\big(D_i^lD_j^kF_{ij}\big)\qquad \textrm{and}\qquad {\rm d}_V\big(D_1^lD_2^kD_3^mF_{12}\big).
 \end{gather}
Using this coframe, any $p$-form $\omega$ on $J^{\infty}(E)$ can be written as
\begin{gather}\label{intbypartsi}
\omega=\omega_0+\sum_{l,k}\sum_{1\leq i < j\leq 3}\beta_{lk}^{ij}\wedge {\rm d}_V\big(D_i^lD_j^kF_{ij}\big)+\sum_{l,k,m}\beta_{lkm}\wedge {\rm d}_V\big(D_1^lD_2^kD_3^mF_{12}\big),
\end{gather}
where $\omega_0$ is a $p$-form generated by the forms in (\ref{coframei}) and the $\beta$ are $(p-1)$-forms on $J^{\infty}(E)$. Since all the terms involving ${\rm d}_V$ expressions will pull back to zero on the equation manifold $\mathcal{R}^{\infty}$, $\iota^*\omega=\iota^*\omega_0$. The forms in~(\ref{coframei}) are still independent when pulled back to $\mathcal{R}^{\infty}$, so $\iota^*\omega=0$ if and only if all the coefficients vanish when $\iota^*\omega_0$ is written in terms of the basis elements~(\ref{coframei}). We can conclude that each of these coefficients is a $C^{\infty}(J^{\infty}(E))$ linear combination of the functions $D_i^lD_j^kF_{ij}$ and $D_1^lD_2^kD_3^mF_{12}$. Thus~$w_0$ can be written as
\begin{gather*}
\omega_0=\sum_{l,k}\alpha_{lk}^{ij}D_i^lD_j^kF_{ij}+\sum_{l,k,m}\alpha_{lkm}D_1^lD_2^kD_3^mF_{12}.
\end{gather*}

The above expression for $\omega_0$, along with equation (\ref{intbypartsi}), yields the desired form in Lem\-ma~\ref{pullstozero}.
\end{proof}

We may now proceed with the proof of Theorem~\ref{structuretheorem}:

By extending the natural coordinates (\ref{natcoorrinfty}) and coframe (\ref{natcoframerinfty}) on $\mathcal{R}^{\infty}$ to the coordinates (\ref{coorjinfty}) and coframe (\ref{coframei})--(\ref{coframeii}) on $J^{\infty}(E)$, there exists a $(2,s)$ form $\tilde{\omega}_0$ on $J^{\infty}(E)$ such that $\iota^*\tilde{\omega}_0=\omega$ for any given $\omega\in\Omega^{2,s}(\mathcal{R}^{\infty})$. Thus let $\omega\in\Omega^{2,s}(\mathcal{R}^{\infty})$ be a ${\rm d}_H$-closed form, and $\tilde{\omega}_0$ a~$(2,s)$ form on $J^{\infty}(E)$ such that $\iota^*\tilde{\omega}_0=\omega$. Since ${\rm d}_H\omega=0$, and ${\rm d}_H$ commutes with projected pullback, $\iota^*({\rm d}_H\tilde{\omega}_0)=0$. Then using Lemma~\ref{pullstozero}, write
\begin{gather*}
{\rm d}_H\tilde{\omega}_0 ={\rm d}x^1\wedge {\rm d}x^2\wedge {\rm d}x^3\wedge \left[\sum_{1\leq i< j\leq3}\sum_{l,k}\alpha_{lk}^{ij}D_i^lD_j^kF_{ij}+\sum_{lkm}\alpha_{lkm}D_1^lD_2^kD_3^mF_{12}\right.
\nonumber\\
\left. \hphantom{{\rm d}_H\tilde{\omega}_0 =}{} +\sum_{1\leq i< j\leq3}\sum_{l,k}\beta_{lk}^{ij}\wedge {\rm d}_V\big(D_i^lD_j^kF_{ij}\big)+\sum_{l,k,m}\beta_{lkm}\wedge {\rm d}_V\big(D_1^lD_2^kD_3^mF_{12}\big)\right],
\end{gather*}
which is a $(3,s)$ form with $\alpha\in\Omega^{0,s}$ and $\beta\in\Omega^{0,s-1}$. By repeated integration by parts, i.e., bringing the highest order total derivatives into a ${\rm d}_H$ expression, ${\rm d}_H\tilde{\omega}_0$ becomes
\begin{gather}\label{intbypartsiii}
{\rm d}_H\tilde{\omega}={\rm d}x^1\wedge {\rm d}x^2\wedge {\rm d}x^3\wedge \left[\sum_{1\leq i< j\leq3}F_{ij}\tilde{\zeta}_{ij}+\sum_{1\leq i< j\leq3}{\rm d}_VF_{ij}\wedge\tilde{\rho}_{ij}\right]
\end{gather}
with $\tilde{\zeta}_{ij}\in\Omega^{0,s}(J^{\infty}(E))$ and $\tilde{\rho}_{ij}\in\Omega^{0,s-1}(J^{\infty}(E))$, where $\tilde{\omega}$ differs from $\tilde{\omega}_0$ only by terms de\-pen\-ding linearly on the $F_{ij}$, ${\rm d}_VF_{ij}$, and total derivatives of these. In other words, $\tilde{\omega}$ and $\tilde{\omega}_0$ differ by terms which vanish on the equation manifold $\mathcal{R}^{\infty}$, so that $\iota^*(\tilde{\omega})=\iota^*(\tilde{\omega}_0)=\omega$.

Thus it has been shown that for any ${\rm d}_H$ closed $(2,s)$ form, $\omega$, on $\mathcal{R}^{\infty}$ there exist forms \begin{gather*}\tilde{\omega}\in\Omega^{2,s}(J^{\infty}(E)), \qquad \tilde{\zeta}_{ij}\in\Omega^{0,s}(J^{\infty}(E)),\qquad \textrm{and}\qquad \tilde{\rho}_{ij}\in\Omega^{0,s-1}(J^{\infty}(E))\end{gather*} such that $\iota^*(\tilde{\omega})=\omega$ and ${\rm d}_H(\tilde{\omega})$ is given by equation~(\ref{intbypartsiii}).

In what follows, the notation $D_{ij}$ is used as a shorthand to mean $D_iD_j$. Defining $\widehat{\rho}_{ij}=\frac{1}{s}\iota^*(\tilde{\rho}_{ij})$, we will next show that the forms $\widehat{\rho}_{ij}$ satisfy the following relationship involving the adjoint equations $\widehat{\mathcal{L}}^*_{ij}$ in the coordinate coframe on $\mathcal{R}^{\infty}$
\begin{gather}
\sum_{1\leq i < j\leq 3}\widehat{\mathcal{L}}^*_{ij}\big(\widehat{\rho}_{ij}\big) =\sum_{1\leq i < j\leq 3}\left[\frac{\partial F_{ij}}{\partial u}\widehat{\rho}_{ij}-D_i\left(\frac{\partial F_{ij}}{\partial u_i}\widehat{\rho}_{ij}\right) -D_j\left(\frac{\partial F_{ij}}{\partial u_j}\widehat{\rho}_{ij}\right)+D_{ij}\left(\frac{\partial F_{ij}}{\partial u_{ij}}\widehat{\rho}_{ij}\right)\right]\nonumber\\
\hphantom{\sum_{1\leq i < j\leq 3}\widehat{\mathcal{L}}^*_{ij}\big(\widehat{\rho}_{ij}\big)}{} =0, \label{rhoadj}
\end{gather}
and that $\omega$ can be written as
\begin{gather}\label{omegaincoordinatecoframe}
\omega=\sum_{1\leq i < j\leq 3}\widehat{\Psi}_{ij}(\widehat{\rho}_{ij})+{\rm d}_H\gamma,
\end{gather}
where
\begin{gather}\label{psiincoordinatecoframe}
\widehat{\Psi}_{ij}\big(\widehat{\rho}_{ij}\big)=\nu_i\wedge\theta\wedge\left[\frac{\partial F_{ij}}{\partial u_i}\widehat{\rho}_{ij}-D_j\left(\frac{\partial F_{ij}}{\partial u_{ij}}\widehat{\rho}_{ij}\right)\right]+\nu_i\wedge\theta_j\wedge\left(\frac{\partial F_{ij}}{\partial u_{ij}}\widehat{\rho}_{ij}\right)
\end{gather}
and $\nu_i=\frac{\partial}{\partial x^i}\lrcorner({\rm d}x^1\wedge {\rm d}x^2\wedge {\rm d}x^3)$.

Next the interior Euler--Lagrange operator $J\colon \Omega^{3,s}(J^{\infty}(E))\rightarrow\Omega^{3,s-1}(J^{\infty}(E))$, which is defined by
\begin{gather}\label{385}
J(\alpha)=\frac{\partial}{\partial u}\,\lrcorner\,\alpha - D_i\left(\frac{\partial}{\partial u_i}\,\lrcorner\,\alpha\right)+D_{ij}\left(\frac{\partial}{\partial u_{ij}}\,\lrcorner\,\alpha\right)+\cdots,
\end{gather}
will be applied to both sides of (\ref{intbypartsiii}). Note that in equation (\ref{385}) $D_{ij}$ simply refers to the operator $D_iD_j$.
By \cite[Theorem~2.6]{an92}, we know that for any $(2,s)$ form $\tilde{\omega}$ on $J^{\infty}(E)$, $J({\rm d}_H\tilde{\omega})=0$. When the operator $J$ is applied to the right hand side of~(\ref{intbypartsiii}), we use the fact that
\begin{gather*}
\frac{\partial}{\partial u_I}\,\lrcorner\,\big[{\rm d}_VF_{ij}\wedge\tilde{\rho}_{ij}+F_{ij}\tilde{\zeta}_{ij}\big]=\frac{\partial F_{ij}}{\partial u_I}\tilde{\rho}_{ij}+Q,
\end{gather*}
where $Q$ consists of terms depending linearly on $F_{ij}$, ${\rm d}_VF_{ij}$, etc., to obtain
\begin{gather*}
J({\rm d}_H\tilde{\omega}) ={\rm d}x^1\wedge {\rm d}x^2\wedge {\rm d}x^3\wedge\left(\sum_{1\leq i< j\leq3}\left[\frac{\partial F_{ij}}{\partial u}\tilde{\rho}_{ij}-D_i\left(\frac{\partial F_{ij}}{\partial u_i}\tilde{\rho}_{ij}\right)\right.\right. \\
 \left.\left. \hphantom{J({\rm d}_H\tilde{\omega}) =}{} -D_j\left(\frac{\partial F_{ij}}{\partial u_j}\tilde{\rho}_{ij}\right)+D_{ij}\left(\frac{\partial F_{ij}}{\partial u_{ij}}\tilde{\rho}_{ij}\right)\right]+Q\right)=0.
\end{gather*}

Since, for example, the expression \begin{gather*}\frac{\partial F_{12}}{\partial u}\tilde{\rho}_{12}-D_i\left(\frac{\partial F_{12}}{\partial u_i}\tilde{\rho}_{12}\right)+D_{ij}\left(\frac{\partial F_{12}}{\partial u_{ij}} \tilde{\rho}_{12}\right)\end{gather*} restricted to $\mathcal{R}^{\infty}$ is the adjoint of $\mathcal{L}_{12}$ in the coordinate frame on $\mathcal{R}^{\infty}$, we see that $J({\rm d}_H\tilde{\omega})=0$ implies that the $\widehat{\rho}_{ij}$ satisfy the equation~(\ref{rhoadj}).

To write the expression for $\omega$ in the coordinate coframe as well, we make use of the homotopy operator $h_H^{r,s}\colon \Omega^{r,s}(J^{\infty}(E))\rightarrow\Omega^{r-1,s}(J^{\infty}(E))$ as defined in \cite{an92}:
\begin{gather*}
h_H^{r,s}(\tilde{\omega})=\frac{1}{s}\sum_{|I|=0}^{k-1}\frac{|I|+1}{n-r+|I|+1}D_I\left[\theta\wedge J^{I_j}\left(\frac{\partial}{\partial x^j}\,\lrcorner\,\tilde{\omega}\right)\right],
\end{gather*}
where
\begin{gather*}
J^I(\tilde{\alpha})=\frac{\partial}{\partial u_I}\,\lrcorner\,\tilde{\alpha}- \left(\begin{matrix}|I|+1 \\0\end{matrix}\right) D_j\left(\frac{\partial}{\partial x^j}\,\lrcorner\,\tilde{\alpha}\right)+\cdots.
\end{gather*}
We are concerned with the case where $n=r=3$, so we have the equation
\begin{gather*}
h_H^{3,s}(\tilde{\omega})=\frac{1}{s}\sum_{|I|=0}^{k-1}D_I\left[\theta\wedge J^{I_j}\left(\frac{\partial}{\partial x^j}\,\lrcorner\,\tilde{\omega}\right)\right].
\end{gather*}
If $\tilde{\omega}$ is a $(3,s)$ form of the type ${\rm d}x^1\wedge {\rm d}x^2\wedge {\rm d}x^3\wedge M$ with $M\in\Omega^{0,s}(\mathcal{R}^{\infty})$, then
\begin{gather}
h_H^{3,s}(\tilde{\omega})=\frac{1}{s}\nu_j\wedge\theta\wedge\left[\left(\frac{\partial}{\partial u_j}\,\lrcorner\,M\right)-D_i\left(\frac{\partial}{\partial u_{ij}}\,\lrcorner\,M\right)\right]\nonumber\\
\hphantom{h_H^{3,s}(\tilde{\omega})=}{} +\frac{1}{s}\nu_j\wedge\theta_i\wedge\left(\frac{\partial}{\partial u_{ij}}\,\lrcorner\,M\right)+\cdots,\label{h3s}
\end{gather}
where as before, $\nu_j=\frac{\partial}{\partial x^j}\,\lrcorner\,({\rm d}x^1\wedge {\rm d}x^2\wedge {\rm d}x^3)$, and the remaining terms depend on interior products of the form $\frac{\partial}{\partial u_I}\,\lrcorner\,M$ for $|I|\geq3$. Using~(\ref{h3s}) along with the expression~(\ref{intbypartsiii}) for~${\rm d}_H(\tilde{\omega})$, we have
\begin{gather*}
h_H^{3,s}({\rm d}_H\tilde{\omega}) = \frac{1}{s}\nu_j\wedge\theta\wedge\left[\frac{\partial F_{12}}{\partial u_j}\tilde{\rho}_{12}+\frac{\partial F_{23}}{\partial u_j}\tilde{\rho}_{23}+\frac{\partial F_{13}}{\partial u_j}\tilde{\rho}_{13}\right.\\
\left. \hphantom{h_H^{3,s}({\rm d}_H\tilde{\omega}) =}{} -D_i\left(\frac{\partial F_{12}}{\partial u_{ij}}\tilde{\rho}_{12}+\frac{\partial F_{23}}{\partial u_{ij}}\tilde{\rho}_{23}+\frac{\partial F_{13}}{\partial u_{ij}}\tilde{\rho}_{23}+\frac{\partial F_{13}}{\partial u_{ij}}\tilde{\rho}_{13}\right)\right]\\
\hphantom{h_H^{3,s}({\rm d}_H\tilde{\omega}) =}{} +\frac{1}{s}\nu_j\wedge\theta_i\wedge\left(\frac{\partial F_{12}}{\partial u_{ij}}\tilde{\rho}_{12}+\frac{\partial F_{23}}{\partial u_{ij}}\tilde{\rho}_{23}+\frac{\partial F_{13}}{\partial u_{ij}}\tilde{\rho}_{13}\right)+\cdots
\end{gather*}
with the remaining terms depending linearly on the $F_{ij}$, ${\rm d}_VF_{ij}$, and their total derivatives. As shown by Anderson in~\cite{an89}, the homotopy operator $h_H^{3,s}$ satisfies the identity
\begin{gather}\label{homotopyid}
\tilde{\omega}=h_H^{3,s}({\rm d}_H\tilde{\omega})+{\rm d}_Hh_H^{2,s}(\tilde{\omega}).
\end{gather}
Therefore the pullback of (\ref{homotopyid}) to $\mathcal{R}^{\infty}$ gives us the expression~(\ref{omegaincoordinatecoframe}) for $\omega$ in the coordinate coframe on~$\mathcal{R}^{\infty}$.

Now to write (\ref{psiincoordinatecoframe}) and (\ref{omegaincoordinatecoframe}) in terms of the Laplace adapted coframe, we let $\mathcal{L}_{ij}^0$ be the operator $\mathcal{L}_{ij}$ defined previously but with $\mu=1$. Then $\mathcal{L}_{ij}^0=\widehat{\mathcal{L}}_{ij}$ and $\mathcal{L}^{0*}_{ij}(\rho_{ij})=\widehat{\mathcal{L}}^*_{ij}(\rho_{ij})$ so we have $\sum\mathcal{L}^{0*}_{ij}(\rho_{ij})=0$. Given that $A_{ij}^i=F_{u_i}$, $A_{ij}^j=F_{u_j}$ and $\sigma_i={\rm d}x^i$, we can look at~(\ref{psiincoordinatecoframe}) and see directly that the expression for $\widehat{\Psi}_{ij}(\rho_{ij})$ in the coordinate coframe corresponds to the expression for $\Psi_{ij}(\rho_{ij})$ given by equations (\ref{psidef1})--(\ref{psidef3}) with $\mu=1$.

We can now make use of the preceding result to prove the following important result concerning the $(2,s)$ cohomology of the variational bi-complex for $s\geq3$.

\begin{Theorem}\label{sgeq3trivial}Let $\mathcal{R}$ be a second order hyperbolic system of type~\eqref{system} and suppose that $\operatorname{ind}(\mathcal{X}_{ij})=\infty$ for $1\leq i< j\leq3$. Then, for $s\geq3$, all type $(2,s)$ conservation laws are trivial. That is,
\begin{gather*}
H^{2,s}(\mathcal{R}^{\infty},{\rm d}_H)=0\qquad \textrm{for}\quad s\geq3.
\end{gather*}
\end{Theorem}

\begin{proof}According to Theorem~\ref{structuretheorem}, we only need to show that there do not exist nonzero type $(0,s-1)$ solutions $\rho_{ij}$ to the adjoint equations $\mathcal{L}^*_{ij}(\rho_{ij})=0$ as this would preclude the existence of any nontrivial conservation law. Begin by rewriting the adjoint equation
\begin{gather*}
X_iX_j(\rho_{ij})+\big(A^i_{ij}\big)^*X_i(\rho_{ij})+\big(A_{ij}^j\big)^*X_j(\rho_{ij})+C^*_{ij}\rho_{ij}=0
\end{gather*}
as a system of first order equations
\begin{gather}\label{firstordersystem1}
X_i(\rho_{ij})=A_{ij}^j\rho_{ij}+\psi^{ij}_i, \\
\label{firstordersystem2}
X_j\big(\psi_{ij}^i\big)=H_{ij}^0\rho_{ij}+A_{ij}^i\psi^{ij}_i.
\end{gather}
We will proceed by showing that if there is a nonzero solution to the system (\ref{firstordersystem1})--(\ref{firstordersystem2}), a~contradiction to Proposition~\ref{vanishingindices} results. To that end, let $\rho_{ij}$ be a nonzero solution to (\ref{firstordersystem1})--(\ref{firstordersystem2}) of adapted order $k$. Because we have taken $s\geq3$, $\rho_{ij}$ is a contact form of degree greater than or equal to 2 and the adapted order of $\rho_{ij}$ is $k\geq1$. Thus $V_l^k \,\lrcorner\,\rho_{ij}\neq0$ for some $l=1,2,3$, where we recall that $V_l^k$ is the vertical vector, defined by~(\ref{dualvectorfields}), dual to the Laplace adapted coframe. For the sake of clarity, and without loss of generality, we will take~$\rho_{ij}$ to be $\rho_{12}$ and~$V_l^k$ to be~$V_3^k$.

We will begin by demonstrating that $V_3^k\,\lrcorner\,\rho_{12}$ is an $X_1$ invariant $(0,s-1)$ contact form. Apply formula~(\ref{verhorint}) to see that
\begin{gather}\label{397}
V_3^{k+1}\,\lrcorner\,X_1(\rho_{12})=\big[V_3^{k+1},X_1\big]\,\lrcorner\,\rho_{12}+X_1\big(V_3^{k+1}\,\lrcorner\,\rho_{12}\big) =\big[V_3^{k+1},X_1\big]\,\lrcorner\,\rho_{12},
\end{gather}
where the last equality holds since $\rho_{12}$ has adapted order $k$ and hence $V_3^{k+1}\,\lrcorner\,\rho_{12}=0$. Now, using the Lie bracket congruences in Proposition~\ref{liebracketcongruences}, we see that
\begin{gather*}
\big[V_3^{k+1},X_1\big]\,\lrcorner\,\rho_{12}=-\big(A_{13}^3\big)^k\big(V_3^{k+1}\,\lrcorner\,\rho_{12}\big)+H_{13}^{k+1}\big(V_3^{k+2}\,\lrcorner\,\rho_{12}\big)=0,
\end{gather*}
again because of the adapted order of $\rho_{12}$. So we may set the right hand side of equation~(\ref{397}) equal to zero, and thereby conclude that
\begin{gather*}
V_3^{k+1}\,\lrcorner\,\psi^{12}_1=0.
\end{gather*}
Again applying formula (\ref{verhorint}) and referring to the Lie bracket congruences of Proposition \ref{liebracketcongruences}, we obtain
\begin{gather*}
V_3^{k+1}\,\lrcorner\,X_2\big(\psi^{12}_1\big) =\big[V_3^{k+1},X_2\big]\,\lrcorner\,\psi^{12}_1+X_2\big(V_3^{k+1}\,\lrcorner\,\psi^{12}_1\big)\\
\hphantom{V_3^{k+1}\,\lrcorner\,X_2\big(\psi^{12}_1\big)}{} =H_{23}^{k+1}\big(V_3^k\,\lrcorner\,\psi^{12}_1\big)-\big(A_{23}^3\big)^{k}\big(V_3^{k+1}\,\lrcorner\,\psi^{12}_1\big) =H_{23}^{k+1}\big(V_3^k\,\lrcorner\,\psi^{12}_1\big).
\end{gather*}
Since interior product of the right hand side of equation~(\ref{firstordersystem2}) with $V_3^{k+1}$ is zero, this implies that
\begin{gather*}
V_3^k\,\lrcorner\,\psi^{12}_1=0.
\end{gather*}
Finally, take the interior product of (\ref{firstordersystem1}) with $V_3^k$ to obtain
\begin{gather*}
X_1\big(V_3^k\,\lrcorner\,\rho_{12}\big) =V_3^k\,\lrcorner\,A_{12}^2\rho_{12}-\big[V_3^k,X_1\big]\,\lrcorner\,\rho_{12} =A_{12}^2V_3^k\,\lrcorner\,\rho_{12}-\big(A_{13}^3\big)^{k-1}V_3^k\,\lrcorner\,\rho_{12}
\end{gather*}
and conclude that $V_3^k\,\lrcorner\,\rho_{12}$ is a relative $X_1$ invariant contact form. Since for systems of the form (\ref{system}) we may choose the characteristic vector fields to commute, the preceding argument can be replicated with the roles of $X_1$ and $X_2$ reversed to show that $V_3^k\,\lrcorner\,\rho_{12}$ is a relative $X_2$ invariant contact form as well. However, this contradicts Proposition~\ref{vanishingindices}, which states that if all Laplace indices are infinite, no nonzero relative $X_1$ and $X_2$ invariant contact forms exist. Hence there cannot in fact exist any nonzero $(0,s-1)$ solutions $\rho_{12}$ to the adjoint equation $\mathcal{L}_{12}^*=0$, and likewise regarding the equations $\mathcal{L}_{13}^*=0$ and $\mathcal{L}_{23}^*=0$, and the proof of Theorem~\ref{sgeq3trivial} is complete.
\end{proof}

We will conclude this section by applying Theorem~\ref{structuretheorem} to an example found in~\cite{kt96}.
\begin{Example}Consider the following involutive system of the form (\ref{system}),
\begin{gather*}
u_{ij}+u_i+u_j+u+1=0\qquad \textrm{for}\quad 1\leq i< j\leq3.
\end{gather*}
The universal linearization for this system is given by
\begin{gather}\label{example1unilin}
\mathcal{L}_{ij}(\theta)=\theta_{ij}+\theta_i+\theta_j+\theta=0
\end{gather}
and the system of adjoint operators is then
\begin{gather*}
\mathcal{L}_{ij}^*=D_iD_j-D_i-D_j+1.
\end{gather*}
Taking the triple $(\rho_{12},\rho_{23},\rho_{13})=\big({\rm e}^{x^1+x^2},{\rm e}^{x^2+x^3},{\rm e}^{x^1+x^3}\big)$, which solves the adjoint equation $\sum\limits_{1\leq i< j\leq3}\mathcal{L}^*_{ij}(\rho_{ij})=0$, we can construct the following conservation law according to the structure theorem presented in Theorem~\ref{structuretheorem}:
\begin{gather*}
\omega=\sigma_1\wedge\sigma_3\wedge\big({\rm e}^{x^1+x^2}\hat{\xi}_1^1-{\rm e}^{x^2+x^3}\hat{\xi}_3^1\big) +\sigma_2\wedge\sigma_3\wedge\big({\rm e}^{x^1+x^2}\hat{\xi}_2^1-{\rm e}^{x^2+x^3}\hat{\xi}_3^1\big) \\
\hphantom{\omega=}{} +\sigma_1\wedge\sigma_2\wedge\big({\rm e}^{x^1+x^3}\hat{\xi}_1^1-{\rm e}^{x^2+x^3}\hat{\xi}_2^1\big) ={\rm I}+{\rm II}+{\rm III},
\end{gather*}
where in this example $\hat{\xi}_i^1=X_i(\theta)+\theta=\theta_i+\theta$. Then we compute ${\rm d}_H\omega={\rm d}_H({\rm I})+{\rm d}_H({\rm II})+{\rm d}_H({\rm III})$ to show that $\omega$ is ${\rm d}_H$ closed as expected. Indeed,
\begin{gather*}
{\rm d}_H({\rm I}) =\sigma_1\wedge\sigma_2\wedge\sigma_3\wedge X_2\big({\rm e}^{x^1+x^2}\hat{\xi}_1^1-{\rm e}^{x^2+x^3}\hat{\xi}_3^1\big) \\
\hphantom{{\rm d}_H({\rm I})}{} =\sigma_1\wedge\sigma_2\wedge\sigma_3\wedge \big({\rm e}^{x^1+x^2}(\theta_{12}+\theta_1+\theta_2+\theta)-{\rm e}^{x^2+x^3}(\theta_{23}+\theta_2+\theta_3+\theta)\big) =0
\end{gather*}
by (\ref{example1unilin}). So $\omega$ is in fact a $(2,1)$ conservation law as we wished to show.
\end{Example}

\begin{Example}
We refer the reader to \cite{kt01} for the following additional example of an involutive linear system of the form (\ref{system}). This system is integrable by repeated applications of the Laplace transformation
\begin{gather*}u_{ij}=\frac{x^i}{x^j(x^j-x^i)}u_i+\frac{x^j}{x^i(x^i-x^j)}u_j\qquad \textrm{for}\quad 1\leq i< j \leq3.\end{gather*}
\end{Example}

\section{Systems of Darboux integrable equations} \label{darbouxsection}

We will begin by stating the definition of Darboux integrability for a single PDE in two independent variables as described, for example, in \cite{ka02}. Given a second order scalar hyperbolic equation, $\mathcal{R}$, in two independent variables~$x$ and~$y$,
\begin{gather}\label{generalsecondorder}
F(x,y,u,u_x,u_y,u_{xx},u_{xy},u_{yy})=0.
\end{gather}
 Let $X$ and $Y$ be the characteristic total vector fields and $\mathcal{I}$ the Pfaffian system associated to $\mathcal{R}$.
\begin{Definition}\label{darbdef2}The equation $\mathcal{R}$ is said to be Darboux integrable if there exist two functionally independent $X$-invariant functions, $I$ and $\tilde{I}$, on $\mathcal{R}^{\infty}$ and two functionally independent $Y$-invariant functions, $J$ and $\tilde{J}$, on $\mathcal{R}^{\infty}$.
\end{Definition}

The method of Darboux then makes use of the invariant functions described in Definition~\ref{darbdef2}, along with arbitrary functions $f_1$ and $f_2$, in order to construct a completely integrable Pfaffian system consisting of the original PDE, $\mathcal{R}$, along with the two additional equations
\begin{gather*}%\label{twomore}
\tilde{I}=f_1(I)\qquad \textrm{and}\qquad \tilde{J}=f_2(J).
\end{gather*}
It can be shown that the integral manifolds of this system correspond to the general solution of~$\mathcal{R}$.
\begin{Theorem} Let $f_1$ and $f_2$ be a pair of monotone functions on $C^{\infty}(\mathbb{R},\mathbb{R})$ and let $\mathcal{L}_{f_1,f_2}$ denote the submanifold of $\mathcal{R}$ defined by $\tilde{I}=f_1(I)$ and $\tilde{J}=f_2(J)$. Then~$\mathcal{I}$ is completely integrable when restricted to $\mathcal{L}_{f_1,f_2}\colon {\rm d}\mathcal{I}|_{\mathcal{L}_{f_1,f_2}}\equiv0\ \mod \mathcal{I}|_{\mathcal{L}_{f_1,f_2}}$.
\end{Theorem}

The relationship between the property of Darboux integrability for equations~(\ref{generalsecondorder}) and the vanishing of generalized Laplace invariants was investigated in~\cite{ak97} and~\cite{aj97}. We will outline one of the main findings of these works here.

Let $X=m_xD_x+m_yD_y$ and $Y=n_xD_x+n_yD_y$, with $\delta=m_xn_y-m_yn_x\neq0$, be distinct characteristic vector fields associated to equation (\ref{generalsecondorder}) such that we have the factorization
\begin{gather*}
(m_x\lambda-m_y\mu)(n_x\lambda-n_y\mu)=\kappa\big(F_{u_{xx}}\lambda^2-F_{u_{xy}}\lambda\mu+F_{u_{yy}}\mu^2\big).
\end{gather*} In investigating equations of the form~(\ref{generalsecondorder}), no assumption is made in~\cite{ak97} or~\cite{aj97} regarding the existence of commuting characteristic vector fields, so we will need to consider the commutator of $X$ and $Y$, which we write as $[X,Y]=PX+QY$. Let $\mathcal{L}(\theta)$ denote the universal linearization of~(\ref{generalsecondorder}) obtained from the identity ${\rm d}_VF=0$:
\begin{gather}\label{univlinofF}
\mathcal{L}(\theta)=XY(\theta)+AX(\theta)+BY(\theta)+C\theta=0,
\end{gather}
 where the coefficients in (\ref{univlinofF}) are given by
 \begin{gather*}
 A=\frac{1}{\delta}[(\kappa F_{u_x}-X(n_x))n_y-(\kappa F_{u_y}-X(n_y))n_x],\\
 B=\frac{1}{\delta}[-(\kappa F_{u_x}-X(n_x))m_y+(\kappa F_{u_y}-X(n_y))m_x],\\
 C=\kappa F_u .
 \end{gather*}
Note that the operator in (\ref{univlinofF}) is essentially the form-valued version of what is considered in \cite{kt96} where the classical Laplace method is discussed. Likewise, in this setting two Laplace invariants are defined. The first is $H_0=X(A)+AB-C$ and depends on the coefficients in~(\ref{univlinofF}). The other Laplace invariant arises when equation~(\ref{univlinofF}) is expressed equivalently with the (non-commuting) characteristics $X$ and $Y$ written in the opposite order, leading to the equation{\samepage
\begin{gather}\label{univlinYX}
YX(\theta)+DX(\theta)+EY(\theta)+G\theta=0,
\end{gather}
where $D=A+P$, $E=B+Q$, and $G=C$.}

 Also similar to the classical case, if for example $H_0$ is nonzero, (\ref{univlinofF}) may be transformed into another equation of the same form. On the equation manifold $\mathcal{R}^{\infty}$, $\mathcal{L}(\theta)=0$ holds identically and defining $\eta_1=Y(\theta)+A\theta$, this identity can be expressed as
\begin{gather}\label{firstytransform}
X(\eta_1)+B\eta_1-H_0\theta=0.
\end{gather}
When $H_0\neq0$, we see that $\eta_1$ satisfies an equation of the same form as (\ref{univlinofF}) by rewriting (\ref{firstytransform}) as
\begin{gather}\label{eta1lin}
XY(\eta_1)+A_1X(\eta_1)+B_1Y(\eta_1)+C_1\eta_1=0,
\end{gather}
where
\begin{gather*}
A_1 =A-\frac{Y(H_0)}{H_0}-P,\qquad B_1 =B-Q,\qquad C_1 =C-X(A)-\frac{Y(H_0)}{H_0}B+Y(B).
\end{gather*}
Equation (\ref{eta1lin}) is called the $\mathcal{Y}$-Laplace transform of (\ref{univlinofF}). The Laplace invariant $H_1=X(A_1)+A_1B_1-C_1$ of (\ref{eta1lin}) can then be computed and if $H_1$ is also nonzero, the equation can be transformed again. Clearly this process can continue with $\eta_i=Y(\eta_{i-1})+A_{i-1}\eta_{i-1}$ and $H_i=X(A_i)+A_iB_i-C_i$ denoting the Laplace invariant of the $i^{\rm th}$ transformed equation, so long as the sequence of Laplace invariants does not vanish.

 Defining the form $\hat{\xi}_1=X(\theta)+E\theta$, and assuming $K_0=Y(E)+ED-G\neq0$, (\ref{univlinYX}) can be transformed analogously into what is called the $\mathcal{X}$-Laplace transform of (\ref{univlinYX}),
\begin{gather*}
YX(\hat{\xi})+D_1X(\hat{\xi})+E_1Y(\hat{\xi})+G_1\hat{\xi}=0,
\end{gather*}
where
\begin{gather*}
D_1 =D+P,\qquad E_1=E-\frac{X(K_0)}{K_0}+Q,\qquad G_1=G-Y(E)-D\frac{X(K_0)}{K_0}+X(D).
\end{gather*}
In this way a second sequence of Laplace invariants, $K_i=Y(E_i)+E_iD_i-G_i$ is generated, again noting that $K_{i+1}$ will be defined so long as $K_i\neq0$.

Let $p=\operatorname{ind}(\mathcal{Y})$ be the Laplace index of the $\mathcal{Y}$-Laplace transform and $q=\operatorname{ind}(\mathcal{X})$ be the Laplace index of the $\mathcal{X}$-Laplace transform. That is, $p$ is the number of times that the $\mathcal{Y}$ Laplace transform must be applied before the generalized Laplace invariant of the transformed equation equals zero and likewise for $q$. If one of the sequences of Laplace invariants never vanishes, we write $p=\infty$ or $q=\infty$.

\begin{Example}Consider the Liouville equation
\begin{gather}\label{liouvilleeq}
u_{xy}={\rm e}^u,
\end{gather}
 which was shown to be Darboux integrable in Section~\ref{section1}. Here the characteristic vector fields $X=D_x$ and $Y=D_y$ commute, and the linearization of (\ref{liouvilleeq}) is
 \begin{gather} \label{liouvillelin}
 XY(\theta)-{\rm e}^u\theta=0.
 \end{gather}
 Computing the Laplace invariants for this equation we see that $H_0={\rm e}^u$, so we may then write the $\mathcal{Y}$-Laplace transform of~(\ref{liouvillelin}) as $XY(\eta_1)-u_yX(\eta_1)-{\rm e}^u\eta_1=0$. From this equation we can compute $H_1=0$, so for the Liouville equation $p=1$. Likewise, we can compute $K_0={\rm e}^u$ and $K_1=0$ to conclude that $q=1$. Both Laplace indices are finite as expected from Theorem~\ref{sgeq3trivial}.
\end{Example}

\subsection{Characteristic systems and invariant functions}
In order to extend the ideas developed in the preceeding section to the case of three equations in three independent variables of the form
\begin{gather}\label{bigfij}
F_{ij}\big(x^1,x^2,x^3,u,u_i,u_j,u_{ij}\big)=0,\qquad 1\leq i< j\leq3,
\end{gather}
we begin by defining characteristic Pfaffian systems which correspond to the characteristics defined by Cartan's structural classification as is detailed in Appendix~\ref{structuralclassification}.
\begin{Definition}Define the characteristic Pfaffian system of order $k$, for a given characteristic vector field, $X_i,$ as
\begin{gather*}
C_k(X_i)=\Omega^1\big(\sigma_j,\sigma_l,\theta,\hat{\xi}_1^1,\hat{\xi}_2^1,\hat{\xi}_3^1,\ldots,\hat{\xi}_1^k,\hat{\xi}_2^k,\hat{\xi}_3^k\big)\qquad \textrm{for}\quad j,l\neq i,
\end{gather*}
where $\big\{\theta,\hat{\xi}_1^1,\hat{\xi}_2^1,\hat{\xi}_3^1,\ldots,\hat{\xi}_1^k,\hat{\xi}_2^k,\hat{\xi}_3^k\big\}$ is either the Laplace-adapted or characteristic coframe on~$\mathcal{R}^{\infty}$, the infinite prolongation of the equation manifold defined by~(\ref{bigfij}). We can likewise define a~characteristic Pfaffian system of order $k$ with respect to a pair of characteristic vector fields, $X_i$~and~$X_j$, as
\begin{gather*}
C_k(X_i,X_j)=\Omega^1\big(\sigma_l,\theta,\hat{\xi}_1^1,\hat{\xi}_2^1,\hat{\xi}_3^1,\ldots,\hat{\xi}_1^k,\hat{\xi}_2^k,\hat{\xi}_3^k\big)\qquad \textrm{for}\quad l\neq i,j.
\end{gather*}
\end{Definition}
Since ${\rm d}I=X_1(I)\sigma_1+X_2(I)\sigma_2+X_3(I)\sigma_3+{\rm d}_VI$, it is immediately apparent that $I$ is an $X_i$ invariant function of order $k$ if and only if ${\rm d}I\in C_k(X_i)$, and that $I$ is invariant with respect to both $X_i$ and $X_j$ if and only if ${\rm d}I\in C_k(X_i,X_j)$.

Associated to any Pfaffian system $\mathcal{I}$, there is a flag of Pfaffian subsystems called the derived flag,
\begin{gather*}\cdots\subset\mathcal{I}^{(i)}\subset\mathcal{I}^{(i-1)}\subset\cdots\subset\mathcal{I}^{(2)}\subset\mathcal{I}^{(1)}\subset\mathcal{I}.\end{gather*}
The $i^{\rm th}$ derived system is defined inductively by a short exact sequence.

The derived flag stabilizes at the maximal completely integrable subsystem of $\mathcal{I}$, which we denote by $\mathcal{I}^{(\infty)}$. Thus, the rank of $(C_k(X_i))^{(\infty)}$, written as $C_k^{(\infty)}(X_i)$ from here on, as a~modu\-le over the ring of $C^{\infty}$ functions equals the number of functionally independent $X_i$ invariant functions of order $\leq k$.

\begin{Lemma}\label{charpfaflemma}Let $\mathcal{R}$ be a second order system of the form~\eqref{bigfij} with characteristic vector fields~$X_i$, and $\big\{\sigma_1,\sigma_2,\sigma_3,\theta,\hat{\xi}_1^1\hat{\xi}_2^1,\hat{\xi}_3^1,\hat{\xi}_1^2,\hat{\xi}_2^2,\hat{\xi}_3^2,\ldots\big\}$ either the Laplace-adapted or characteristic coframe on~$\mathcal{R}^{\infty}$. Then for any $k\geq 1$,
\begin{gather}\label{maxintsubxi}
C_k^{(\infty)}(X_i)\subset\Omega^1\big(\sigma_j,\sigma_l,\theta,\hat{\xi}_1^1,\hat{\xi}_2^1,\hat{\xi}_3^1,\hat{\xi}_j^2,\hat{\xi}_l^2,\ldots,\hat{\xi}_j^k,\hat{\xi}_l^k\big)\qquad \text{for} \quad j,l\neq i
\end{gather}
and
\begin{gather}\label{maxintsubxixj}
C_k^{(\infty)}(X_i,X_j)\subset\Omega^1\big(\sigma_l,\theta,\hat{\xi}_1^1,\hat{\xi}_2^1,\hat{\xi}_3^1,\hat{\xi}_l^2,\ldots,\hat{\xi}_l^k\big)\qquad \text{for}\quad l\neq i,j.
\end{gather}
\end{Lemma}

\begin{proof}To establish (\ref{maxintsubxi}), refer to the structure equations given in Appendix~\ref{appendixB} for the Laplace adapted coframe to see that
\begin{gather*}{\rm d}_H\hat{\xi}_i^k\equiv\sigma_i\wedge\hat{\xi}_i^{k+1}\quad \mod \big\{\sigma_j,\sigma_l,\theta,\hat{\xi}_1^1,\hat{\xi}_2^1,\hat{\xi}_3^1,\ldots,\hat{\xi}_1^k,\hat{\xi}_2^k,\hat{\xi}_3^k\big\}.\end{gather*}
This implies that $\hat{\xi}_i^k\notin C_k^{(1)}(X_i)$ since ${\rm d}\hat{\xi}_i^k\not\equiv0 \mod\big\{\sigma_j,\sigma_l,\theta,\hat{\xi}_1^1,\hat{\xi}_2^1,\hat{\xi}_3^1,\ldots,\hat{\xi}_1^k,\hat{\xi}_2^k,\hat{\xi}_3^k\big\}$. Conti\-nuing to argue in the same manner allows one to conclude that \begin{gather*}\hat{\xi}_i^{k-1}\notin C_k^{(2)}(X_i) , \qquad \hat{\xi}_i^{k-2}\notin C_k^{(3)}(X_i),\end{gather*} and so forth. The contact form $\hat{\xi}_i^1$ can not be eliminated however since ${\rm d}\sigma_2$ and/or ${\rm d}\sigma_3$ may contain the term $\sigma_i\wedge\hat{\xi}_i^2$. Likewise, for (\ref{maxintsubxixj}), we note that ${\rm d}_H\hat{\xi}_m^k\equiv \sigma_m\wedge\hat{\xi}_m^{k+1} \mod C_k(X_i,X_j)$ for $m=i,j$ so that $\hat{\xi}_m^k\not\in C_k^{(1)}(X_i,X_j)$ and the rest of the argument continues in the same fashion.
\end{proof}

The next result follows immediately from Lemma~\ref{charpfaflemma}.
\begin{Corollary}\label{orderkequivto}If $J_k$ is an $X_i$-invariant function of order $k$, then
\begin{gather*}%\label{orderkinvariant}
{\rm d}_VJ_k\equiv a_k\hat{\xi}_l^k+b_k\hat{\xi}_j^k\quad \mod \big\{\theta,\hat{\xi}_1^1,\hat{\xi}_2^1,\hat{\xi}_3^1,\hat{\xi}_j^2,\hat{\xi}_l^2,\ldots,\hat{\xi}_j^{k-1},\hat{\xi}_l^{k-1}\big\}
\end{gather*}
for $l,j\neq i$ and where $a_kb_k\neq0$. If $I_k$ is invariant with respect to $X_i$ and $X_j$, then
\begin{gather}
\label{orderkinvariant2}
{\rm d}_VI_k\equiv a_k\hat{\xi}_l^k\quad \mod \big\{\theta,\hat{\xi}_1^1,\hat{\xi}_2^1,\hat{\xi}_3^1,\hat{\xi}_l^2,\ldots,\hat{\xi}_l^{k-1}\big\}
\end{gather}
for $l\neq i,j$ and $a_k\neq0$.
\end{Corollary}

\subsection{Darboux integrability for a system of three equations}\label{darbforsystemofthree}

We will first discuss the definition of Darboux integrability given by Anderson et al.\ in~\cite{afv09} where the authors investigate a class of differential systems for which superposition formulas can be constructed. For a system of the form
\begin{gather}\label{systems4}
F_{ij}\big(x^1,x^2,x^3,u,u_i,u_j,u_{ij}\big)=u_{ij}-f_{ij}\big(x^1,x^2,x^3,u,u_i,u_j\big)=0
\end{gather}
we have the associated Pfaffian system $\mathcal{I}$ generated by $\big\{\omega^0,\omega^1,\omega^2,\omega^3\big\}$ with structure equations
\begin{gather}\label{streqn1}
{\rm d}\omega^0 \equiv 0,\\
\label{streqn2}
{\rm d}\omega^1 \equiv \sigma_1\wedge\pi_1\quad \mod\mathcal{I},\\
\label{streqn3}
{\rm d}\omega^2 \equiv \sigma_2\wedge\pi_2,\\
\label{streqn4}
{\rm d}\omega^3 \equiv \sigma_3\wedge\pi_3
\end{gather}
and characteristic vector fields $X_i=D_i$. To apply the definition of Darboux integrability established in~\cite{afv09}, there must exist a coframe, $\{\theta_0,\theta_1,\theta_2,\theta_3,\sigma_1,\pi_1,\sigma_2,\pi_2,\sigma_3,\pi_3\}$ such that $\mathcal{I}$ is generated algebraically by the 1-forms and 2-forms
\begin{gather}\label{assosingsys}
\mathcal{I}=\big\{\theta_0,\theta_1,\theta_2,\theta_3,\hat{\Omega}^1,\hat{\Omega}^2,\check{\Omega}^1\big\},
\end{gather}
 where $\hat{\Omega}^i\in\Omega^2(\sigma_1,\pi_1,\sigma_2,\pi_2)$, for $i=1,2$, and $\check{\Omega}^1\in\Omega^2(\sigma_3,\pi_3)$. One then defines two Pfaffian systems $\hat{\mathcal{V}}=\{\theta_0,\theta_1,\theta_2,\theta_3,\sigma_1,\pi_1,\sigma_2,\pi_2\}$ and $\check{\mathcal{V}}=\{\theta_0,\theta_1,\theta_2,\theta_3,\sigma_3,\pi_3\}$, which will be referred to as the singular differential systems for $\mathcal{I}$ with respect to the decomposition (\ref{assosingsys}). For simplicity of notation, we will continue using this particular choice of decomposition in all that follows. Note that we could just as well have grouped ${\rm d}\omega^1$ and ${\rm d}\omega^3$, or ${\rm d}\omega^2$ and ${\rm d}\omega^3$, together to make an analogous decomposition. Let $\mathcal{V}^{(\infty)}$ denote the largest integrable sub-bundle of a given Pfaffian system, $\mathcal{V}$, and recall that the rank of $\mathcal{V}^{(\infty)}$ will equal the number of functionally independent first integrals of $\mathcal{V}$. Then according to \cite{afv09}, a differential system $\mathcal{I}$ is Darboux integrable if the associated singular systems $\hat{\mathcal{V}}$ and $\check{\mathcal{V}}$ are Pfaffian and define a \textit{Darboux pair}, meaning that the following three properties are satisfied:
\begin{enumerate}\itemsep=0pt
\item[1)] $\hat{\mathcal{V}}+\check{\mathcal{V}}^{(\infty)}=T^*M$ and $\check{\mathcal{V}}+\hat{\mathcal{V}}^{(\infty)}=T^*M$,
\item[2)] $\hat{\mathcal{V}}^{(\infty)}\cap\check{\mathcal{V}}^{(\infty)}=\{0\}$,
\item[3)] ${\rm d}\omega\in\Omega^2(\hat{\mathcal{V}})+\Omega^2(\check{\mathcal{V}})$ for all $\omega\in\Omega^1(\hat{\mathcal{V}}\cap\check{\mathcal{V}}),$
\end{enumerate}
where $M$, thought of as an open subset of $\mathbb{R}^{10}$, is the manifold on which the exterior differential system $\mathcal{I}$ is defined.

We can now show that systems of equations of the form (\ref{systems4}) which possess certain characteristic invariant functions, described in the lemma below, will define a Darboux pair and thus be Darboux integrable in the sense of~\cite{afv09}.
\begin{Lemma}\label{darbdef} Let $u_{ij}=f_{ij}\big(x^1,x^2,x^3,u,u_i,u_j\big)$, $1\leq i< j<3$ be a system of three hyperbolic equations in one dependent and three independent variables with characteristic vector fields $X_1$, $X_2$, $X_3$. If there exist smooth, real-valued functions, $I$, $\tilde{I}$, $J$, $\tilde{J},$ and~$K$,~$\tilde{K}$, such that the following two conditions hold, then the system defines a Darboux pair and thus satisfies the notion of Darboux integrability defined in~{\rm \cite{afv09}}.
 \begin{enumerate}\itemsep=0pt
 \item[$1.$] $I$ and $\tilde{I}$ are invariant with respect to two of the characteristic vector fields, say $X_i$ and $X_j$, and $I$ and $\tilde{I}$ are functionally independent: ${\rm d}I\wedge {\rm d}\tilde{I}\neq0$.
 \item[$2.$] $J$, $\tilde{J}$, $K$, and $\tilde{K}$ are all invariant with respect to $X_l$, $l\neq i,j$, and are all functionally independent from each other: ${\rm d}J\wedge {\rm d}\tilde{J}\wedge {\rm d}K\wedge {\rm d}\tilde{K}\neq0$.
\end{enumerate}
\end{Lemma}

From structure equations (\ref{streqn1})--(\ref{streqn4}) it is clear that the third property defining a Darboux pair is satisfied. Property 2 asserts that $\hat{\mathcal{V}}$ and $\check{\mathcal{V}}$ have no integrals in common, and if necessary this condition can be satisfied by restricting $\hat{\mathcal{V}}$ and $\check{\mathcal{V}}$ to a level set of their common integrals. Thus the essential property to check when constructing a Darboux pair is that
\begin{gather}\label{darbouxpair}
\hat{\mathcal{V}}+\check{\mathcal{V}}^{(\infty)}=T^*M \qquad \textrm{and}\qquad \check{\mathcal{V}}+\hat{\mathcal{V}}^{(\infty)}=T^*M.
\end{gather}
We will now show that this condition is indeed satisfied by the systems described in Lemma~\ref{darbdef}. Letting $\hat{H}$ be the annihilator of $\hat{\mathcal{V}}$ and $\check{H}$ be the annihilator of $\check{\mathcal{V}}$, we have that
\begin{gather*}\hat{H}=\operatorname{span} \{\partial_{\sigma_3},\partial_{\pi_3}\}\qquad \textrm{and}\qquad \check{H}=\operatorname{span} \{\partial_{\sigma_1},\partial_{\pi_1},\partial_{\sigma_2},\partial_{\pi_2}\}, \end{gather*} where $\partial_{\sigma_i}$ is the vector field dual to~$\sigma_i$. Then the following lemma established in~\cite{afv09} allows us to see that $\hat{\mathcal{V}}$ and $\check{\mathcal{V}}$ satisfy~(\ref{darbouxpair}) and thus form a Darboux pair.

\begin{Lemma}Let $f$ be a real-valued function on $M$. If $X(f)=0$ for all vector fields $X\in\hat{H}$, then ${\rm d}f\in\hat{\mathcal{V}}^{(\infty)}$. Likewise, $X(f)=0$ for all $X\in\check{H}$ implies ${\rm d}f\in\check{\mathcal{V}}^{(\infty)}$.
\end{Lemma}

Take $i=1$, $j=2$, and $l=3$ in Lemma~\ref{darbdef}. Then to show that $\hat{\mathcal{V}}+\check{\mathcal{V}}^{(\infty)}=T^*M$, we observe that $X(I)=X(\tilde{I})=0$ for all $X\in\check{H}$ so that $\{{\rm d}I,{\rm d}\tilde{I}\}\subset\check{\mathcal{V}}^{(\infty)}$. Since dim $\hat{\mathcal{V}}=8$ and dim $\check{\mathcal{V}}^{(\infty)}\geq2$, we see that we have enough one forms to span $T^*M=\mathbb{R}^{10}$. Likewise, $Z(J)=Z(\tilde{J})=Z(K)=Z(\tilde{K})=0$ for all $Z\in\hat{H}$, so $\{{\rm d}J, {\rm d}\tilde{J},{\rm d}K,{\rm d}\tilde{K}\}\subset\hat{\mathcal{V}}^{(\infty)}$. Finally, $\dim\check{\mathcal{V}}=6$ and $\dim \hat{\mathcal{V}}^{(\infty)}\geq4$ establishes that again, there are sufficient one forms to span~$T^*M$.

As in the preceding sections, for systems of the form (\ref{systems4}), we take the characteristic vector fields to be $X_i=D_i$ which commute. However one would expect to be able to reproduce the analysis carried out in this paper for more general systems of the form \begin{gather*}F_{ij}\big(x^1,x^2,x^3,u,u_i, u_j, u_{ii}, u_{ij}, u_{jj}\big)=0\qquad \textrm{for}\quad 1\leq i< j\leq3,\end{gather*} which are still contained in the same class of equations described in Cartan's structural classification (see Theorem~\ref{classification}) as the systems~(\ref{systems4}), but for which the characteristics~$X_i$ do not commute. We will now describe how for systems with non-commuting characteristic vector fields, the vector fields~$X_i$ may be rescaled to make them commute pairwise, given the existence of sufficiently many characteristic invariant functions.

\begin{Proposition}\label{commutingcvfs} Let $I$ and $K$ be functions on $\mathcal{R}^{\infty}$ such that $I$ is invariant with respect to both~$X_i$ and~$X_j$, and $K$ is invariant with respect to $X_l$. Assume that $X_i(K)$, $X_l(I)$ and $X_j(K)$ are all nonzero and define the following rescaled characteristic vector fields
\begin{gather*}
\tilde{X}_i=\frac{X_i}{X_i(K)},\qquad \tilde{X}_l=\frac{X_l}{X_l(I)},\qquad \text{and}\qquad \tilde{X}_j=\frac{X_j}{X_j(K)}.
\end{gather*}
Then the vector field $\tilde{X}_l$ commutes with both $\tilde{X}_i$ and $\tilde{X}_j$:
\begin{gather}\label{xlcommutes}
\big[\tilde{X}_i,\tilde{X}_l\big]=0\qquad \textrm{and}\qquad \big[\tilde{X}_j,\tilde{X}_l\big]=0.
\end{gather}
\end{Proposition}

\begin{proof}Write the commutator of $X_i$ and $X_j$ as follows
\begin{gather*}\label{liebracketxixj}
[X_i,X_j]=\sum_{k=1}^3B_{ij}^kX_k\qquad \textrm{for}\quad 1\leq i< j\leq3.
\end{gather*}
Taking into consideration the conditions of invariance imposed on $I$ and $K$, we may write
\begin{gather}\label{commutingcvf1}
[X_i,X_l](I) =X_iX_l(I)=B_{il}^lX_l(I)\qquad \textrm{and}\\
\label{commutingcvf2}
[X_i,X_l](K) =-X_lX_i(K)=B_{il}^iX_i(K)+B_{il}^jX_j(K)\qquad \textrm{for}\quad i,j\neq l.
\end{gather}
Now use (\ref{commutingcvf1}) and (\ref{commutingcvf2}) in the following computation to confirm (\ref{xlcommutes})
\begin{gather*}
\big[\tilde{X}_i,\tilde{X}_l\big] =\frac{1}{X_i(K)}X_i\left(\frac{1}{X_l(I)}X_l\right)-\frac{1}{X_l(I)}X_l\left(\frac{1}{X_i(K)}X_i\right) \\
\hphantom{\big[\tilde{X}_i,\tilde{X}_l\big]}{} =\frac{1}{X_i(K)}\left(\frac{1}{X_l(I)}X_iX_l-\frac{X_iX_l(I)}{X_l(I)^2}X_l\right)-\frac{1}{X_l(I)}\left(\frac{1}{X_i(K)}X_lX_i-\frac{X_lX_i(K)}{X_i(K)^2}X_i\right) \\
\hphantom{\big[\tilde{X}_i,\tilde{X}_l\big]}{} =\frac{1}{X_i(K)X_l(I)}\left([X_i,X_l]-B_{il}^lX_l-B_{il}^iX_i-B_{il}^jX_j\right) =0.
\end{gather*}
An analogous calculation using $B_{jl}^jX_j(K)+B_{jl}^iX_i(K)=-X_lX_j(K)$ and $B_{jl}^lX_l(I)=X_jX_l(I)$ shows that $\big[\tilde{X}_j,\tilde{X}_l\big]=0$.
\end{proof}

The formulation of the Darboux-adapted coframe, as well as the construction of conservation laws for Darboux integrable systems, will rely on our ability to obtain characteristic invariant contact forms from the invariant functions whose existence is guaranteed in Lemma~\ref{darbdef}. For example, we can write down a contact form which is invariant with respect to the characteristic vector fields~$X_1$ and~$X_2$ as follows,

\begin{Lemma}\label{justx3inv}Let $I$, $J$, and $K$ be functions on $\mathcal{R}^{\infty}$ that are invariant with respect to the characteristic vector field $X_3$ such that $X_1(I)=1$, $X_2(I)=0$, $X_1(J)=0$, and $X_2(J)=1$. Also define $K'=X_1(K)$ and $K''=X_2(K)$. Then \begin{gather*}\omega={\rm d}_VK-K'{\rm d}_VI-K''{\rm d}_VJ\end{gather*} is an $X_3$-invariant contact form.
\end{Lemma}

\begin{proof} Note that $K'$ and $K''$ are also $X_3$ invariant since the characteristic vector fields $X_i$ all commute. Computing ${\rm d}_H\omega$, we obtain
\begin{gather*}
{\rm d}_H\omega =-{\rm d}_V{\rm d}_HK-{\rm d}_HK'\wedge {\rm d}_VI+K'{\rm d}_V{\rm d}_HI-{\rm d}_HK''\wedge {\rm d}_VJ+K''{\rm d}_V{\rm d}_HJ \\
\hphantom{{\rm d}_H\omega}{} =-{\rm d}_V(K'\sigma_1+K''\sigma_2)-(X_1(K')\sigma_1+X_2(K')\sigma_2)\wedge {\rm d}_VI \\
\hphantom{{\rm d}_H\omega=}{}+K'{\rm d}_V(\sigma_1)-(X_1(K'')\sigma_1+X_2(K'')\sigma_2)\wedge {\rm d}_VJ+K''{\rm d}_V(\sigma_2) \\
\hphantom{{\rm d}_H\omega}{}=\sigma_1\wedge {\rm d}_VK'+\sigma_2\wedge {\rm d}_VK''-X_1(K')\sigma_1\wedge {\rm d}_VI-X_2(K')\sigma_2\wedge {\rm d}_VI \\
\hphantom{{\rm d}_H\omega=}{} -X_1(K'')\sigma_1\wedge {\rm d}_VJ-X_2(K'')\sigma_2\wedge {\rm d}_VJ \\
\hphantom{{\rm d}_H\omega}{}=\sigma_1\wedge({\rm d}_VK'-X_1(K'){\rm d}_VI-X_1(K''){\rm d}_VJ)\\
\hphantom{{\rm d}_H\omega=}{} +\sigma_2\wedge({\rm d}_VK''-X_2(K'){\rm d}_VI-X_2(K''){\rm d}_VJ).
\end{gather*}
From this calculation we see that $X_3(\omega)=X_3\,\lrcorner\,{\rm d}_H(\omega)=0$ so that $\omega$ is $X_3$-invariant.
\end{proof}

Likewise,we can use $X_1$ and $X_2$ invariant functions to obtain $X_1$ and $X_2$ invariant contact forms in the following way.

\begin{Lemma}\label{x1andx2inv}Let $I$ and $J$ be functions on $\mathcal{R}^{\infty}$ which are invariant with respect to both cha\-racteristic vector fields $X_1$ and $X_2$, such that $X_3(I)=I'$ and $X_3(J)=1$. Then \begin{gather*}\omega={\rm d}_VI-I'{\rm d}_VJ\end{gather*} is an $X_1$ and $X_2$ invariant contact form.
\end{Lemma}

\begin{proof}In the same spirit as the last proof, we compute ${\rm d}_H\omega$ and by grouping terms arrive at \begin{gather*}{\rm d}_H\omega=\sigma_3\wedge({\rm d}_VI'-X_3(I'){\rm d}_VJ).\end{gather*} So $X_1\,\lrcorner\,{\rm d}_H(\omega)=X_2\,\lrcorner\,{\rm d}_H(\omega)=0$ and $\omega$ is both $X_1$ and $X_2$ invariant.
\end{proof}

Given the preceding lemma, the following corollary follows immediately from Propo\-si\-tion~\ref{vanishingindices}.

\begin{Corollary}\label{ifdarbthenindex}Let $\mathcal{R}$ be a system of equations of the form \eqref{systems4}. If $\mathcal{R}$ satisfies the conditions of Lemma~{\rm \ref{darbdef}}, then at least one of the Laplace indices $p_{ij}=\operatorname{ind}(\mathcal{X}_{ij})$ must be finite.
\end{Corollary}

The construction of the Darboux adapted coframe will proceed according to the orders of the various characteristic invariant functions in Lemma~\ref{darbdef}. Here we do not have a complete analysis of the possible orders of invariants, as was given by Goursat~\cite{go96} in the case of a single hyperbolic second order equation in the plane. Thus we will carry out the construction under certain assumptions on the orders of the invariants to serve as a demonstration of how a coframe may be constructed in a particular case.

Given two first order $X_1$ and $X_2$ invariant functions, $I_1$ and $\tilde{I}_1$, we may obtain a sequence of functions invariant with respect to $X_1$ and $X_2$ by repeatedly applying $X_3$:
\begin{gather*}\tilde{I}_1,I_1,I_2=X_3(I_1),I_3=X_3(I_2),\ldots,\qquad \textrm{where}\quad X_3\big(\tilde{I}_1\big)=1.
\end{gather*}
Then by Lemma~\ref{x1andx2inv}, the following are a sequence of $X_1$ and $X_2$ invariant contact forms,
\begin{gather*}\alpha_1={\rm d}_VI_1-I_2{\rm d}_V\tilde{I}_1,\quad \alpha_2={\rm d}_VI_2-I_3{\rm d}_V\tilde{I}_1,\quad \ldots,\quad \alpha_i={\rm d}_VI_i-I_{i+1}{\rm d}_V\tilde{I}_1 ,\quad \dots.\end{gather*}
For $i\geq2$, equation~(\ref{orderkinvariant2}) implies that \begin{gather*}\alpha_i\equiv a_i\hat{\xi}_3^i\mod\big\{\theta,\hat{\xi}_1^1,\ldots,\hat{\xi}_1^{i-1},\xi_2^1,\ldots,\hat{\xi}_2^{i-1}\big\},\end{gather*} where $a_i\neq0$. For $i=1$, we have $\alpha_1\equiv a_1\hat{\xi}_3^1\mod\{\theta\}$ with $a_1\neq0$, as we will now show. Since ${\rm d}_H\theta\equiv \sigma_1\wedge\hat{\xi}_1^1+\sigma_2\wedge\hat{\xi}_2^1+\sigma_3\wedge\hat{\xi}_3^1\mod\{\theta\}$, if $a_1=0$ then
\begin{gather}\label{dhalpha1}
{\rm d}_H\alpha_1\equiv0 \quad \mod\big\{\theta,\hat{\xi}_1^1,\hat{\xi}_2^1,\hat{\xi}_3^1\big\}.
\end{gather}
 But since ${\rm d}_H\alpha_1=\sigma_3\wedge\alpha_2\equiv a_2\sigma_3\wedge\hat{\xi}_3^2\mod\big\{\theta,\hat{\xi}_1^1,\hat{\xi}_2^1,\hat{\xi}_3^1\big\}$, equation~(\ref{dhalpha1}) contradicts the fact that~$\alpha_2$ is of order~2, so indeed $a_1\neq0$.

The set of contact forms $\{\alpha_1,\alpha_2,\ldots\}$ can thus replace the branch of the Laplace coframe given by $\big\{\hat{\xi}_3^1,\hat{\xi}_3^2,\ldots\big\}$ previously. The structure equations for the $\alpha_i$ are given by
\begin{gather}\label{alphaistreqn}
{\rm d}\alpha_i={\rm d}\tilde{I}_1\wedge\alpha_{i+1},
\end{gather}
where (\ref{alphaistreqn}) is calculated in the following way
\begin{gather*}
{\rm d}\alpha_i ={\rm d}({\rm d}_VI_i)-{\rm d}(I_{i+1})\wedge {\rm d}_V\tilde{I}_1-I_{i+1}{\rm d}({\rm d}_V\tilde{I}_1)
 ={\rm d}_H{\rm d}_VI_i-{\rm d}(I_{i+1})\wedge {\rm d}_V\tilde{I}_1-I_{i+1}{\rm d}_H{\rm d}_V\tilde{I}_1 \\
\hphantom{{\rm d}\alpha_i}{} =-{\rm d}_V(I_{i+1}\sigma_3)-(I_{i+2}\sigma_3+{\rm d}_VI_{i+1})\wedge {\rm d}_V\tilde{I}_1+I_{i+1}{\rm d}_V(\sigma_3) \\
\hphantom{{\rm d}\alpha_i}{} =-{\rm d}_VI_{i+1}\wedge\sigma_3-(I_{i+2}\sigma_3+{\rm d}_VI_{i+1})\wedge {\rm d}_V\tilde{I}_1 \\
\hphantom{{\rm d}\alpha_i}{} =(\sigma_3+{\rm d}_V\tilde{I}_1)\wedge({\rm d}_VI_{i+1}-I_{i+2}{\rm d}_V\tilde{I}_1) ={\rm d}\tilde{I}_1\wedge\alpha_{i+1}.
\end{gather*}
To complete the coframe, we need to construct contact forms which are equivalent to $\big\{\hat{\xi}_1^1,\hat{\xi}_1^2,\ldots\big\}$ and $\big\{\hat{\xi}_2^1,\hat{\xi}_2^2,\ldots\big\}$. We accomplish this by first generating $X_3$ invariant functions, repeatedly applying the total vector fields $X_1$ and~$X_2$ to the $X_3$ invariant functions $J$ and $K$ from Lemma~\ref{darbdef} to obtain the following two sequences of~$X_3$ invariant functions:
\begin{gather*}
\tilde{J},J=J_1,J_2=X_1(J_1),J_3=X_1(J_2),\ldots,\\
\tilde{K},K=K_1,K_2=X_2(K_1),K_3=X_2(K_2),\ldots.
\end{gather*}

 Then using Lemma~\ref{justx3inv}, we can write two corresponding sequences of $X_3$ invariant contact forms:
\begin{gather*}
\beta_1 ={\rm d}_VJ_1-X_1(J_1){\rm d}_V\tilde{K}-X_2(J_1){\rm d}_V\tilde{J},\\
 \beta_2 ={\rm d}_VJ_2-X_1(J_2){\rm d}_V\tilde{K}-X_2(J_2){\rm d}_V\tilde{J},\\
 \cdots \cdots\cdots\cdots\cdots\cdots\cdots\cdots\cdots\cdots\cdots\cdots\cdots\\
 \beta_i ={\rm d}_VJ_i-X_1(J_i){\rm d}_V\tilde{K}-X_2(J_i){\rm d}_V\tilde{J}
\end{gather*}
and
\begin{gather*}
\gamma_1={\rm d}_VK_1-X_1(K_1){\rm d}_V\tilde{K}-X_2(K_1){\rm d}_V\tilde{J},\\
\gamma_2={\rm d}_VK_2-X_1(K_2){\rm d}_V\tilde{K}-X_2(K_2){\rm d}_V\tilde{J},\\
 \cdots \cdots\cdots\cdots\cdots\cdots\cdots\cdots\cdots\cdots\cdots\cdots\cdots\\
\gamma_i={\rm d}_VK_i-X_1(K_i){\rm d}_V\tilde{K}-X_2(K_i){\rm d}_V\tilde{J}.
\end{gather*}
Then a similar argument to that given above for the sequence of $\alpha_i$ shows that
\begin{gather*}
\beta_i\equiv b_k\hat{\xi}_1^i+c_k\hat{\xi}_2^i\quad \mod\big\{\theta,\hat{\xi}_3^1,\hat{\xi}_3^2,\ldots,\hat{\xi}_3^{i-1}\big\}
\end{gather*}
and likewise,
\begin{gather*}
\gamma_i\equiv e_k\hat{\xi}_1^i+f_k\hat{\xi}_2^i\quad \mod\big\{\theta,\hat{\xi}_3^1,\hat{\xi}_3^2,\ldots,\hat{\xi}_3^{i-1}\big\},
\end{gather*}
where at least one of $b_k$ and $c_k$ is nonzero, and at least one of $e_k$ and $f_k$ is nonzero. While Lemma~\ref{darbdef} requires that ${\rm d}J\wedge {\rm d}\tilde{J}\wedge {\rm d}K\wedge {\rm d}\tilde{K}\neq0$, this condition is not enough to guarantee that ${\rm d}_VJ_i$ and ${\rm d}_VK_i$ are functionally independent. As a result, there is some ambiguity about the construction of the remaining two branches of the coframe, as it will depend on whether the coefficients on $\hat{\xi}_1^i$ and $\hat{\xi}_2^i$ in~$\beta_i$ and~$\gamma_i$ satisfy $b_kf_k-c_ke_k\neq0$. In other words on whether~$\beta_i$ and~$\gamma_i$ are independent ${\rm mod}\big\{\theta,\hat{\xi}_3^1,\hat{\xi}_3^2,\ldots,\hat{\xi}_3^{i-1}\big\}$. If any pair $\beta_i$ and $\gamma_i$ does fail to be independent, then we will complete the coframe by taking the necessary contact forms from the Laplace adapted coframe.

This coframe would likely be useful in the pursuit of an eventual classification of Darboux-integrable systems, similar to that presented in \cite{aj97} for the case of a scalar second order hyperbolic PDE in the plane.

\subsection[Generating infinitely many $(1,s)$ and $(2,s)$ conservation laws]{Generating infinitely many $\boldsymbol{(1,s)}$ and $\boldsymbol{(2,s)}$ conservation laws}\label{generating1sconslaws}

We may assume that for a Darboux integrable system, $[X_i,X_l]=P_{il}^iX_i+P_{il}^lX_l=0$ and $[X_j,X_l]=P_{jl}^jX_j+P_{jl}^lX_l=0$. Recalling the structure equation
\begin{gather*}
{\rm d}_H\sigma_l=P_{il}^l\sigma_i\wedge\sigma_l+P_{jl}^l\sigma_j\wedge\sigma_l
\end{gather*}
we see that if $X_l$ commutes with $X_i$ and $X_j$, then ${\rm d}_H\sigma_l=0$. With this in mind, we may now demonstrate a way to construct infinitely many $(1,s)$ conservation laws for Darboux integrable systems.
\begin{Theorem}
Let $\mathcal{R}$ be a Darboux integrable system of equations. Then there exist infinitely many non-trivial type $(1,s)$ conservation laws for all $s\geq0$.
\end{Theorem}
\begin{proof}Fix distinct $i$, $j$, and $l$ as in Lemma~\ref{darbdef}. For $s=0$, we may construct a conservation law by taking $\omega=I\sigma_l$ where $I$ is an $X_i$ and $X_j$ invariant function. Suppose that $\omega$ is exact, that is $\omega=I\sigma_l={\rm d}_H\tilde{I}$ for some function $\tilde{I}$. Then \begin{gather*}{\rm d}_H\tilde{I}=X_1(\tilde{I})\sigma_1+X_2(\tilde{I})\sigma_2+X_3(\tilde{I})\sigma_3=I\sigma_l,\end{gather*} so $\tilde{I}$ is $X_i$ and $X_j$ invariant and $X_l(\tilde{I})=I$. Thus $\omega$ will be a nontrivial conservation law so long as $I\neq X_l(\tilde{I})$ for $\tilde{I}$ another $X_i$ and $X_j$ invariant function. Additional $(1,0)$ conservation laws may be constructed by generating additional $X_i$ and $X_j$ invariant functions by applying~$X_l$ to~$I$ repeatedly. Note that $X_l(I)$ is still $X_i$ and $X_j$ invariant since $X_l$ commutes with both~$X_i$ and~$X_j$.

For $s\geq1$, we first note that if $\omega$ is a $d$ closed $(s+1)$ form with vanishing $(2,s-1)$ component, then the $(1,s)$ component of $\omega$ will be ${\rm d}_H$ closed. To see this, write ${\rm d}\omega$ as a sum of forms of different bidegrees, according to the direct sum decomposition of the space of $(s+2)$ forms,
\begin{gather}\label{decompdw}
\Omega^{s+2}(\mathcal{R}^{\infty})=\bigoplus_{0\leq i\leq3}\Omega^{(i,s+2-i)}(\mathcal{R}^{\infty}).
\end{gather}
 Let ${\rm d}\omega=\nu_0+\nu_1+\nu_2+\nu_3$ where $\nu_i\in\Omega^{(i,s+2-i)}$. Since each $\nu_i$ belongs to a different component of (\ref{decompdw}), ${\rm d}\omega=0$ implies that $\nu_i=0$ for all $i$. In particular, $\nu_2=0$. The form $\omega$ can itself be decomposed according to bidegree as well. Write $\omega=\mu_0+\mu_1+\mu_2+\mu_3$ where $\mu_i\in\Omega^{(i,s+1-i)}$. Now decompose ${\rm d}\omega$ into its horizontal and vertical components, ${\rm d}\omega=({\rm d}_H+{\rm d}_V)(\omega)$, to see that ${\rm d}_H\mu_1+{\rm d}_V\mu_2=\nu_2=0$. If the $(2,s-1)$ component, $\mu_2$, of $\omega$ vanishes, then this last equation becomes ${\rm d}_H\mu_1=0$ and we can conclude that the $(1,s)$ component of $\omega$ is ${\rm d}_H$ closed.

Now, to construct a ${\rm d}_H$ closed $(1,s)$ form, we begin by defining $\alpha={\rm d}I_1\wedge {\rm d}I_2\wedge\cdots\wedge {\rm d}I_{s+1}$ where each of the functions $I_n$ is an $X_i$ and $X_j$ invariant function. These functions can be obtained as before by taking $I=I_1$ and repeatedly applying the total vector field $X_l$ to $I$ to generate the remaining functions. The $(s+1)$ form $\alpha$ will have vanishing $(2,s-1)$ component: since the functions $I_n$ are all $X_i$ and $X_j$ invariant, the only horizontal form that will appear in any of the ${\rm d}I_n$ is $\sigma_l$. Thus, by the preceding discussion,
\begin{gather*}
\pi^{1,s}({\rm d}I_1\wedge {\rm d}I_2\wedge\cdots\wedge {\rm d}I_{s+1})
\end{gather*}
will be ${\rm d}_H$ closed as required.

Furthermore, we can construct a ${\rm d}_H$ closed $(1,s)$ form that is not exact. In other words, we can write down a nontrivial $(1,s)$ conservation law. Let $p_l=\min\{\operatorname{ind}(\mathcal{X}_{il}),\operatorname{ind}(\mathcal{X}_{jl})\}$, and take $k\geq p_l+s$. Then for a nonzero form $\eta\in\Omega_{\mathcal{I}}^{s-2}(\alpha_{p_l+1},\alpha_{p_l+2},\ldots,\alpha_{k-1})$ where $\mathcal{I}$ is the ring of functions which are both $X_i$ and $X_j$ invariant, the $(1,s)$ form
\begin{gather}\label{omegaexp}
\omega=\sigma_l\wedge\alpha_{k+1}\wedge\alpha_k\wedge\eta
\end{gather}
 is ${\rm d}_H$ closed, but not ${\rm d}_H$ exact. It is straightforward to see that $\omega$ is indeed ${\rm d}_H$ closed since~$\sigma_l$ is itself ${\rm d}_H$ closed and $\alpha_{k+1}\wedge\alpha_k\wedge\eta$ is both $X_i$ and $X_j$ invariant. Now suppose that $\omega$ were ${\rm d}_H$ exact, that is, that there exists a $(0,s)$ form, $\gamma$ such that ${\rm d}_H\gamma=\omega$. Since ${\rm d}_H\gamma=\sigma_1\wedge X_1(\gamma)+\sigma_2\wedge X_2(\gamma)+\sigma_3\wedge X_3(\gamma)$, comparing this with the above expression,~(\ref{omegaexp}), for~$\omega$ we conclude that $\gamma$ must also be $X_i$ and $X_j$ invariant. Moreover, $X_l(\gamma)=\alpha_{k+1}\wedge\alpha_k\wedge\eta$ which implies that $\gamma\in\Omega_{\mathcal{I}}^s(\alpha_{p_l+1},\alpha_{p_l+2},\ldots)$. The adapted order of $\omega$ being $k+1$ implies that the adapted order of~$\gamma$ would be~$k$. Then $\gamma$ can be written as $\gamma=\alpha_k\wedge\beta+\delta$ where $\beta$ and $\delta$ each have adapted order $\leq k-1$, and are both~$X_i$ and $X_j$ invariant forms. But then $X_l(\gamma)=X_l(\alpha_k\wedge\beta+\delta)=\alpha_{k+1}\wedge\alpha_k\wedge\eta$ which indicates that $\beta=\alpha_k\wedge\eta$, showing that $\beta$ has adapted order~$k$. This contradicts the original premise that $\beta$ be of adapted order $\leq k-1$, from which we conclude that $\omega$ is not exact.
\end{proof}

It is possible to construct infinitely many nontrivial $(2,s)$ conservation laws in a similar way. By an argument analogous to that given above in the case of $(1,s)$ conservation laws, if~$\omega$ is a~${\rm d}$~closed~$(s+2)$ form with vanishing $(3,s-1)$ component, then the $(2,s)$ component of~$\omega$ is~${\rm d}_H$ closed. Consider the form ${\rm d}J_1\wedge {\rm d}J_2\wedge\cdots\wedge {\rm d} J_{s+2}$ where the functions $J_n=X_i(J_{n-1})$ are all $X_l$ invariant. This $(s+2)$ form has vanishing $(3,s-1)$ component since each term ${\rm d} J_n$ can only contain the horizontal forms $\sigma_i$ and $\sigma_j$. So
\begin{gather*}
\nu=\pi^{2,s}({\rm d}J_1\wedge {\rm d}J_2\wedge\cdots\wedge {\rm d}J_{s+2})
\end{gather*}
is a ${\rm d}_H$ closed $(2,s)$ form. Construct a second ${\rm d}_H$ closed $(2,s)$ form by taking the $(2,s)$ component of the $(s+2)$ form ${\rm d}K_1\wedge {\rm d}K_2\wedge\cdots\wedge {\rm d}K_{s+2}$ where $K_n=X_j(K_{n-1})$ and each $K_n$ is $X_l$ invariant. Let
\begin{gather*}
\mu=\pi^{2,s}({\rm d}K_1\wedge {\rm d}K_2\wedge\cdots\wedge {\rm d}K_{s+2}).
\end{gather*}

Take $\mathcal{I}$ to be the ring of $X_l$ invariant functions. Then the following $(2,s)$ form will be a~nontrivial conservation law
\begin{gather}\label{2somega}
\omega=\sigma_i\wedge\sigma_j\wedge[\nu_{k+1}\wedge\nu_{k}\wedge\eta_1+\mu_{k+1}\wedge\mu_k\wedge\eta_2],
\end{gather}
where $\eta_1\in\Omega_{\mathcal{I}}^{s-2}(\nu_{p_i+1},\nu_{p_i+2},{\ldots},\nu_{k-1})$ and $\eta_2\in\Omega_{\mathcal{I}}^{s-2}(\mu_{p_j+1},\mu_{p_j+2},{\ldots},\mu_{k-1})$, and $k\geq\max\{p_j+s, p_i+s\}$. Clearly ${\rm d}_H\omega=0$ since $\sigma_i\wedge\sigma_j$ is ${\rm d}_H$ closed and $\nu_{k+1}\wedge\nu_{k}\wedge\eta_1$ and $\mu_{k+1}\wedge\mu_{k}\wedge\eta_2$ are $X_l$ invariant. To see that $\omega$ is not exact, we will assume otherwise and produce a contradiction. Suppose that there exists a $(1,s)$ form, $\gamma$, such that ${\rm d}_H\gamma=\omega$. First note that by writing ${\rm d}_H\gamma=\sigma_1\wedge X_1(\gamma)+\sigma_2\wedge X_2(\gamma)+\sigma_3 \wedge X_3(\gamma)$ and looking at equation~(\ref{2somega}), it may be concluded that $\gamma$ is $X_l$ invariant, $l\neq i,j$. Furthermore, any $(1,s)$ form can be written as \begin{gather*}\gamma=\sigma_1\wedge M_1+\sigma_2\wedge M_2+\sigma_3\wedge M_3,\end{gather*} where each $M_i$ is a $(0,s)$ form. Comparing equation (\ref{dh1s}) for ${\rm d}_H\gamma$ with the expression (\ref{2somega}) for $\omega$, we see that ${\rm d}_H\gamma=\omega$ implies that
\begin{gather*}
{\rm d}_H\gamma=\sigma_i\wedge\sigma_j\wedge[X_j(M_i)-X_i(M_j)],
\end{gather*}
where
\begin{gather}\label{expformimj}
X_j(M_i)-X_i(M_j)=\nu_{k+1}\wedge\nu_{k}\wedge\eta_1+\mu_{k+1}\wedge\mu_k\wedge\eta_2
\end{gather}
and so $M_i, M_j\in\Omega_{\mathcal{I}}^s(\nu_{p_l+1},\nu_{p_l+2},\ldots)$. Since the adapted order of $\omega$ is $(k+1)$, the adapted order of $\gamma$ must be $k$ and hence $M_i$ and $M_j$ may be written as
\begin{gather*}%\label{mimiforms}
M_i=\mu_{k}\wedge\alpha_{1}+\delta_{1}, \qquad M_j=\nu_{k}\wedge\alpha_{2}+\delta_{2},
\end{gather*}
where $\alpha_{1}$, $\alpha_{2}$, $\delta_{1}$, and $\delta_{2}$ all have adapted order $\leq k-1$. Plugging these expressions for~$M_i$ and~$M_j$ into (\ref{expformimj}) yields
\begin{gather} \label{ordercontra}
X_j(\mu_{k}\wedge\alpha_{1}+\delta_{1})-X_i(\nu_{k}\wedge\alpha_{2}+\delta_{2}) =\nu_{k+1}\wedge\nu_{k}\wedge\eta_1+\mu_{k+1}\wedge\mu_k\wedge\eta_2.
\end{gather}
From equation (\ref{ordercontra}), we deduce that $\alpha_1=\mu_k\wedge\eta_2$ and $\alpha_2=\nu_k\wedge\eta_1$. Since this contradicts the fact that $\alpha_1$ and $\alpha_2$ were taken to have order $\leq k-1$, we can conclude that $\omega$ is indeed not exact and thus represents a nontrivial conservation law.

\section{Conclusion and further research}\label{section5}

In this paper we have proven a number of results concerning the existence and structure of conservation laws for involutive systems of partial differential equations of the form
\begin{gather}\label{conclusionsystem}
u_{ij}-f_{ij}\big(x^1,x^2,x^3,u,u_i,u_j\big)=0\qquad \textrm{for}\quad 1\leq i< j\leq3.
\end{gather}
After defining the constrained variational bicomplex, which provides the framework in which we investigate form-valued conservation laws, and developing other technical tools such as the generalized Laplace transform, and the structure equations and Lie bracket congruences for the Laplace adapted coframe, the first main result is proved in Section~\ref{structuretheoremsec}. This result, Theorem~\ref{structuretheorem}, describes the structure of $(2,s)$ conservation laws for $s\geq1$ in terms of solutions to the adjoint equations of the associated linearized system. As a consequence of this structure theorem, we obtain another key result which states that if all of the Laplace indices of the linearized system are infinite, then all $(2,s)$ conservation laws for $s\geq3$ are trivial. That is, in this case the horizontal cohomology of the constrained variational bi-complex associated to the system is trivial:
\begin{gather*}H^{2,s}(\mathcal{R}^{\infty},{\rm d}_H)=0\qquad \textrm{for}\quad s\geq3.\end{gather*}

Conditions are given in Lemma~\ref{darbdef} which will guarantee that a system of the form (\ref{systemsbeingstudied}) will satisfy the definition of Darboux integrability introduced in~\cite{afv09}. It is then show that these same conditions will imply that at least one of the Laplace indices for the linearization of such a~system must be finite. A procedure for using characteristic invariant contact forms to construct non-trivial $(1,s)$ and $(2,s)$ conservation laws for $s\geq0$ is then described in Section~\ref{generating1sconslaws}.

There are several natural directions in which one may try to extend the results presented here. First, we should emphasize that in this paper we have considered the very specific class of semi-linear systems of the form (\ref{conclusionsystem}). In particular, the characteristic vector fields, $X_i$, for such systems commute. It should be fairly easy to investigate which of the results presented here could be replicated for more general classes of nonlinear systems belonging to the dif\-fe\-rent cases of Cartan's structural classification listed in Theorem~\ref{classification}, specifically systems with non-commuting characteristics. Other directions in which to extend this work would be to consider overdetermined systems in one dependent and $n$ independent variables, instead of only 3 independent variables, or to allow $u$ to be a vector-valued function (see, for example, \cite{sz03}). Furthermore, it remains to find further examples and applications which fall under the form of systems (\ref{conclusionsystem}) which we consider here. Constructing involutive overdetermined systems is a~project in its own right. Another interesting avenue of research would consist in classifying Darboux-integrable systems of the form~(\ref{conclusionsystem}) or a~generalization of these. A potentially fruitful approach to this problem would utilize the definitions and classification scheme presented in~\cite{afv09}, which is discussed in Section~\ref{darbouxsection}.

Beyond these points, there is the overarching question of how the horizontal cohomology of the constrained variational bi-complex $H^{r,s}$ can be interpreted for $s\geq 1$. In other words, what does the horizontal cohomology mean in terms of the analytic properties of the solution space of the system $\mathcal{R}$? Here we refer the reader to \cite{an92} for some examples of the ways in which the cohomology of the variational bi-complex has been used to study, for example, the inverse problem of the calculus of variations for autonomous systems of ODE, Riemannian structures, and Gelfand--Fuks cohomology. In~\cite{ts82}, Tsujishita also describes several applications of the higher-degree cohomology of the variational bi-complex, including its role in determining deformation classes of the solution space of~$\mathcal{R}$. An interpretation of $H^{n-1,2}(\mathcal{R}^{\infty})$ cohomology, where~$n$ is the number of independent variables in the associated PDEs system, is also given in~\cite{zu87}. It is shown there that the existence of nontrivial cohomology classes $[\omega]\in H^{n-1,2}$ indicates the possible existence of variational principles for~$\mathcal{R}$. Specifically, if~$\mathcal{R}$ is the Euler--Lagrange equations for some Lagrangian~$\lambda$, then it can be shown that ${\rm d}_V\lambda={\rm d}_H\eta$ on $\mathcal{R}^{\infty}$ for $\eta$ a form of bi-degree $(n-1,1)$. Therefore $\omega={\rm d}_V(\eta)$, which Zuckerman calls the universal conserved current for $\lambda$, is a ${\rm d}_H$ closed $(n-1,2)$ form. Assuming that $\lambda$ is nontrivial, Vinogradov's two-line theorem \cite{vi842} implies that~$\omega$ is not~${\rm d}_H$ exact, so that if $H^{n-1,2}(\mathcal{R}^{\infty})=0$, then $\mathcal{R}$ does not admit a variational principle.

\appendix
\section{Structural classification}\label{structuralclassification}

In \cite{ca11}, Cartan studies involutive Pfaffian systems defined on a suitable jet space, which arise when considering involutive systems of PDEs from the perspective of exterior differential systems. In the course of his analysis, the Pfaffian systems' structure equations are computed and simplified. Systems with the same reduced structure equations are then regarded as being structurally equivalent. Some of the main results of this work are summarized in~\cite{st00}, and it is this exposition which we will follow as we proceed to place systems of the form~(\ref{systemsbeingstudied}) in the context of Cartan's structural classification.

Let $\Delta^a(x^i,u,u_i,u_{ij})=0$ with $i,j=1,2,3$ be an involutive second order system of PDEs considered as a subset of the bundle of 2-jets, $J^2\big(\mathbb{R}^3,\mathbb{R}\big)$. As defined in Section~\ref{jetbundles}, consider the contact system on $J^2\big(\mathbb{R}^3,\mathbb{R}\big)$ generated by the 1-forms
\begin{gather*}
\theta^0={\rm d}u-\sum_{j=1}^3u_j{\rm d}x^j,\qquad \theta^i={\rm d}u_i-\sum_{j=1}^3u_{ij}{\rm d}x^j, \qquad i=1,2,3.
\end{gather*}
Let $\omega^i={\rm d}x^i$, $\pi_j^i={\rm d}u_{ij}=\pi_i^j$ so that a local coframe for $J^2\big(\mathbb{R}^3,\mathbb{R}\big)$ is given by $\big\{\omega^i,\theta^0,\theta^i,\pi_j^i\big\}$ for $i,j=1,2,3$. Then we have the following structure equations
\begin{gather}\label{dstructure}
{\rm d}\theta^0=\sum_{j=1}^3\omega^j\wedge\theta^j,\qquad {\rm d}\theta^i=\sum_{j=1}^3\omega^j\wedge\pi_j^i.
\end{gather}
The set of forms $\big\{\theta^0,\theta^i\big\}$ comprises a basis for the contact system on $J^2\big(\mathbb{R}^3,\mathbb{R}\big)$. Keeping in mind that our goal is to obtain a simplified set of structure equations, another basis may be found by transforming the forms $\big\{\theta^0,\theta^i\big\}$ in the following way
\begin{gather*}
\left(\begin{matrix}\bar{\theta}^0 \\\bar{\theta}^1 \\\bar{\theta}^2 \\\bar{\theta}^3\end{matrix}\right)=\left(\begin{matrix}1 & 0 & 0 & 0 \\0 & a_1^1 & a_2^1 & a_3^1 \\0 & a_1^2 & a_2^2 & a_3^2 \\0 & a_1^3 & a_2^3 & a_3^3\end{matrix}\right)\left(\begin{matrix}\theta^0 \\\theta^1 \\\theta^2 \\\theta^3\end{matrix}\right)=\left(\begin{matrix}1 & 0 \\0 & A\end{matrix}\right)\left(\begin{matrix}\theta^0 \\\theta^1 \\\theta^2 \\\theta^3\end{matrix}\right),
\end{gather*}
where the entries $a_j^i$ are functions on the jet space $J^2\big(\mathbb{R}^3,\mathbb{R}\big)$ and $\det A\neq0$. Observe that this transformation preserves the contact form $\theta^0$, as well as the Pfaffian system $\big\{\theta^1,\theta^2,\theta^3\big\}$. Denote by $\mathcal{A}$ the group consisting of such matrices\footnote{In~\cite{ca11}, Cartan allows for contact transformations on the full set of contact forms, including shifts in the $\theta^0$ direction. Since the terms $\theta^0$ are not relevant to the symbol analysis as carried out in~\cite{st00}, we will likewise fix these terms in the group of transformations we consider.},~$A$. Then we will show presently that there is a~prolongation $\mathcal{A}'$ of $\mathcal{A}$ that acts on the coframe $\big\{\omega^i,\theta^0,\theta^i,\pi_j^i\big\}$, $i=1,2,3,$ such that the form of the structure equations for the corresponding contact system is preserved.

Let $A=(a_j^i)$ and $A^{-1}=(\alpha_j^i)$ so that $\bar{\theta^i}=A\theta^i$ and $\theta^i=A^{-1}\bar{\theta^i}$. A straightforward calculation using the definitions provided shows that \begin{gather*}{\rm d}\bar{\theta}^0=\sum_{j=1}^3\bar{\omega}^j\wedge\bar{\theta}^j,\end{gather*} where $\bar{\omega}^j:=\sum\limits_{i=1}^3\alpha_j^i\omega^i$. Likewise, utilizing the transformations $\bar{\omega}=\big(A^{-1}\big)^{\rm T}\omega=\big(A^{\rm T}\big)^{-1}\omega$ and $\omega=A^{\rm T}\bar{\omega}$, we can compute \begin{gather*}{\rm d} \bar{\theta}^i\equiv\sum_{j=1}^3\bar{\omega}^j\wedge\bar{\pi}^i_j\quad \mod \big\{\bar{\theta}^i\big\},\end{gather*} where $\bar{\pi}_j^i:=\sum\limits_{k,l=1}^3a_k^ia_l^j\pi_l^k.$ Hence, if~$\mathcal{A}'$ acts on~$\omega^i$ and~$\pi_j^i$ as follows
\begin{gather*}
\omega^i\longmapsto\sum_{j=1}^3\alpha_i^j\omega^j\qquad \textrm{and}\qquad \pi_j^i\longmapsto\sum_{k,l=1}^3a_k^ia_l^j\pi_l^k,
\end{gather*}
then the structure equations ${\rm d}\theta^i=\sum\limits_{j=1}^3\omega^j\wedge\pi_j^i$ become the congruences
\begin{gather*}{\rm d}\bar{\theta}^i\equiv\sum_{j=1}^3\bar{\omega}^j\wedge\bar{\pi}_j^i\quad \mod \big\{\theta^i\big\}\end{gather*} and the following lemma may be stated.

\begin{Lemma}$\mathcal{A}'$ prolongs $\mathcal{A}$, maps linear combinations of the $\omega^i$ to linear combinations of $\omega^i$, and transforms the structure equations~\eqref{dstructure} to
\begin{gather*}
{\rm d}\bar{\theta}^0=\sum_{j=1}^3\bar{\omega}^j\wedge\bar{\theta}^j, \qquad {\rm d}\bar{\theta}^i\equiv\sum_{j=1}^3\bar{\omega}^j\wedge\bar{\pi}^i_j\quad \mod \bar{\theta}^i, \qquad i=1,2,3.
\end{gather*}
Moreover, the $\pi_j^i$ are transformed as the coefficients of a quadratic form.
\end{Lemma}
The group $\mathcal{A}'$ will later be utilized to simplify the possible structure equations which must be considered in the structural classification.

We will now consider more specifically systems of three second order equations, $\Delta^a(x^i,u,u_i$, $u_{ij})=0$, $1\leq a\leq3$, which satisfy the maximal rank condition that the $3\times13$ Jacobian matrix \begin{gather*}J_{\Delta}(x,u,u_i,u_{ij})=\left(\frac{\partial\Delta^a}{\partial x^i},\frac{\partial\Delta^a}{\partial u_I}\right)\end{gather*} of $\Delta^a$ with respect to $x^i$, $u$, $u_i$, and~$u_{ij}$ is of rank~3 whenever $\Delta^a=0$. To further ensure that the PDEs system is non-degenerate, we require the $(3\times 6)$ submatrix \begin{gather*}\left(\frac{\partial \Delta^a}{\partial u_{ij}}\right)\end{gather*} to have rank~3. Let $\mathcal{R}$ be the codimension 3 submanifold of $J^2\big(\mathbb{R}^3,\mathbb{R}\big)$ thus defined by $\Delta^a=0$, and $\mathcal{C}^2(\mathcal{R})$ the corresponding contact system restricted to $\mathcal{R}$. When the one forms $\big\{\omega^i,\theta^0,\theta^i,\pi_j^i\big\}$ are restricted to~$\mathcal{R}$, we obtain 3 linear relations between them. By replacing $\pi_j^i$ with $\pi_j^i$ plus an appropriate linear combination of the $\theta^i$, including $\theta^0$, these linear relations may be simplified to the following
\begin{gather}\label{relations}
\sum_{i,j=1}^3A_i^{jk}\pi_j^i+\sum_{l=1}^3A_l^k\omega^l=0\qquad \textrm{for}\quad k=1,2,3.
\end{gather}
We will henceforth assume that the Pfaffian system $\mathcal{C}^2(\mathcal{R})$ is an involution on which \begin{gather*}\omega^1\wedge\omega^2\wedge\omega^3\neq0,\end{gather*} in other words, there exists at least one K\"{a}hler-regular integral element of the form \begin{gather*}\theta^i=0, \qquad \pi_j^i-\sum_{l=1}^3\alpha_{jl}^i\omega^l=0,\end{gather*} where the $\alpha_{jl}^i$ are symmetric with respect to all indices. This allows the relations (\ref{relations}) to be further reduced to
\begin{gather}\label{reducedrelations}
\sum_{i,j=1}^3A_i^{jk}\pi_j^i=0\qquad \textrm{for}\quad k=1,2,3.
\end{gather}
Finally, when considering the general 3-dimensional involution defined by \begin{gather*}\theta^i=0,\qquad \pi_j^i-\sum_{l=1}^3a_{jl}^i\omega^l=0,\end{gather*} the relations (\ref{reducedrelations}) produce 9 equations
\begin{gather}\label{9equations}
\sum_{i,j=1}^3A_i^{jk}a_{jl}^i=0\qquad \textrm{for}\quad l,k=1,2,3.
\end{gather}
We would like to know how many of the coefficients $a_{jl}^i$ are left arbitrary by these 9 equations since Cartan's involutivity test tells us that $\mathcal{C}^2(\mathcal{R})$ is involutive with $\omega^1\wedge\omega^2\wedge\omega^3\neq0$ if and only if the number of arbitrary coefficients $a_{jl}^i$ above equals $3\cdot3-2\sigma_1-\sigma_2$ where $\sigma_1$ and $\sigma_2$ are the first two reduced Cartan characters of $\mathcal{C}^2(\mathcal{R})$. Given that we are concerned with the case of systems of 3 equations in 3 independent variables, we can see that the Cartan characters $\sigma_1$ and $\sigma_2$ satisfy $\sigma_1\leq3$ and $\sigma_1+\sigma_2\leq3$. Thus the number of arbitrary coefficients $a_{jl}^i$ will be $9-2\sigma_1-\sigma_2\geq9-6=3$. Since the $a_{jl}^i$ are symmetric with respect to all indices, there are 10 terms, which we now see are related by at most $10-3=7$ independent linear equations. Hence, at most 7 of the 9 equations (\ref{9equations}) are linearly independent. Making the identifications $\pi_j^i\longleftrightarrow\hat{\xi}^i\hat{\xi}^j$ and $a_{jl}^i\longleftrightarrow\hat{\xi}^i\hat{\xi}^j\hat{\xi}^l$, we can state the following result:
\begin{Lemma}\label{linearrelations}Defining the quadratic forms
\begin{gather*}M^k=\sum_{i,j=1}^3A_i^{jk}\hat{\xi}^i\hat{\xi}^j,\qquad k=1,2,3,
\end{gather*} the $9$ cubic forms $\hat{\xi}^iM^k$ are related by at least $9-7=2$ linear equations.
\end{Lemma}
The lemma states that the 9 cubic forms $\hat{\xi}^iM^k$ are related by at least two linear relations, thus there are 6 linear forms, $l^i$ and $m^i$, $i=1,2,3$, in the variables $\hat{\xi}^i$, such that these two linear relations can be expressed by the two equations,
\begin{gather*}
l^1M^1+l^2M^2+l^3M^3=0,\qquad m^1M^1+m^2M^2+m^3M^3=0.
\end{gather*}
Given that these relations among the $M^i$ exist, one may take a linear combination, \begin{gather*}\sum_{i=1}^3\big(z_1l^i+z_2m^i\big)M^i=0,\end{gather*} and find $z_1$ and $z_2$ such that one of the coefficients $z_1l^i+z_2m^i$ vanishes. In other words, we can choose $z_1$ and $z_2$ such that the three linear forms $z_1l^i+z_2m^i$ are linearly dependent. Equivalently, consider $z_1$ and $z_2$ as homogeneous coordinates on the Riemann sphere, $\mathbb{P}_1(\mathbb{C})$, and define $z:=z_2/z_1$. Then we have three linear forms $l^i+zm^i$, each of which is a linear combination of the three variables $\hat{\xi}^i$. Taking the coefficient of each of the three variables $\hat{\xi}^i$ in each of the three expressions $l^i+zm^i$ yields a $3\times3$ matrix. Setting the determinant of this matrix equal to zero will determine $z$ in such a way that the three linear forms $l^i+zm^i$ are linearly dependent, and produce a cubic equation in $z$ with three roots in $\mathbb{P}_1(\mathbb{C})$.

To make this more explicit, and simplified, the group $\mathcal{A}'$ may be used to induce linear transformations of the variables $\hat{\xi}^i$ which then induce corresponding linear transformations of~$z_1$ and~$z_2$. This allows us to assume, for example, that one root of the equation mentioned above is $z=0$. Then $l^1$, $l^2,$ and $l^3$ are linearly dependent and can be taken to be \begin{gather*}l^1=\hat{\xi}^1,\qquad l^2=\hat{\xi}^2,\qquad l^3=0.\end{gather*} Furthermore, if we specify the $m^i$ by \begin{gather*}m^1=a^1\hat{\xi}^1+b^1\hat{\xi}^2+c^1\hat{\xi}^3, \qquad m^2=a^2\hat{\xi}^1+b^2\hat{\xi}^2+c^2\hat{\xi}^3,\qquad m^3=a^3\hat{\xi}^1+b^3\hat{\xi}^2+c^3\hat{\xi}^3, \end{gather*} then the equation which $z$ must satisfy is the cubic polynomial given by,
\begin{gather}\label{cubicpoly}
\det\left(\begin{matrix}
1+a^1z & b^1z & c^1z \\
a^2z & 1+b^2z & c^2z \\
a^3z & b^3z & c^3z
\end{matrix}
\right)= z\det \left(\begin{matrix}
1+a^1z & b^1z & c^1z \\
a^2z & 1+b^2z & c^2z \\
a^3 & b^3 & c^3
\end{matrix}
\right)=0.
\end{gather}
The structural classification then breaks into 5 possible cases, according to whether the equation~(\ref{cubicpoly}) has $(i)$ 3 simple roots, $(ii)$~one double and one simple root, $(iii)$~one triple root, or $(iv)$~vanishes identically. Case~$(iv)$ splits further into two sub-cases according to which of the following two forms the matrix in equation~(\ref{cubicpoly}) can be reduced,
\begin{gather*}
\left(\begin{matrix}
1 & 0 & 0 \\
0 & 1 & z \\
1 & 0 & 0
\end{matrix}
\right)
\qquad \textrm{or}\qquad
\left(\begin{matrix}
1 & z & 0 \\
0 & 1 & 0 \\
1 & 0 & 0
\end{matrix}
\right).
\end{gather*}
 Hence there are 5 possible contact invariant structures in total, as described in the following theorem.
\begin{Theorem}\label{classification}There are $5$ distinct classes of involutive systems of $3$ PDEs in $3$ independent and $1$ dependent variable. These $5$ classes are described according to which of the $\pi_j^i$ vanish:
\begin{alignat*}{3}
&(i) && \pi_2^1=\pi_3^1=\pi_3^2=0,& \\
&(ii)&& \pi_1^1=\pi_3^1=\pi_3^2=0,&\\
&(iii) \ && \pi_1^1=\pi_2^1=\pi_3^1-\pi_2^2=0, & \\
&(iv)&& \pi_1^1=\pi_2^1=\pi_2^2=0, & \\
&(v)&& \pi_1^1=\pi_2^1=\pi_3^1=0.&
\end{alignat*}
\end{Theorem}

Finally, we will demonstrate to which of the five classes given in Theorem~\ref{classification} systems of the type (\ref{systemsbeingstudied}) belong. Recall that the system of three equations (\ref{systemsbeingstudied}) is given by
\begin{gather*}F_{ij}\big(x^1,x^2,x^3,u,u_i,u_j,u_{ij}\big)=u_{ij}-f_{ij}\big(x^1,x^2,x^3,u,u_i,u_j\big)=0\qquad \textrm{for}\quad 1\leq i < j\leq 3.\end{gather*}
Then by Lemma~\ref{linearrelations} there exist at least 2 linear relations between the 3 quadratic forms \begin{gather*}M^1=\hat{\xi}^1\hat{\xi}^2,\qquad M^2=\hat{\xi}^2\hat{\xi}^3,\qquad M^3=\hat{\xi}^1\hat{\xi}^3.\end{gather*} Write these equations as
\begin{gather*}
l^1\hat{\xi}^1\hat{\xi}^2+l^2\hat{\xi}^2\hat{\xi}^3+l^3\hat{\xi}^1\hat{\xi}^3=0,\\
m^1\hat{\xi}^1\hat{\xi}^2+m^2\hat{\xi}^2\hat{\xi}^3+m^3\hat{\xi}^1\hat{\xi}^3 =0.
\end{gather*}
Then, for example, we can take $\big(l^1,l^2,l^3\big)=\big(\hat{\xi}^3,-\hat{\xi}^1,0\big)$ and $\big(m^1,m^2,m^3\big)=\big(\hat{\xi}^3,\hat{\xi}^1,-2\hat{\xi}^2\big)$ and set the following determinant equal to zero
\begin{gather*}
\det\left(\begin{matrix}0 & 0 & 1+z \\-1+z & 0 & 0 \\0 & -2z & 0\end{matrix}\right)=0.
\end{gather*}
In other words, the equation $(1+z)(1-z)(2z)=0$ must hold. Here we have 3 simple roots, so the structure of this system fits into case $(i)$ of Theorem~\ref{classification}. The corresponding structure equations are then
\begin{gather*}
{\rm d}\theta^0 \equiv0,\\ \nonumber
{\rm d}\theta^1 \equiv\omega^1\wedge\pi_1^1\quad \mod \big(\theta^0,\theta^1,\theta^2,\theta^3\big),\\
{\rm d}\theta^2 \equiv\omega^2\wedge\pi_2^2,\\
{\rm d}\theta^3 \equiv\omega^3\wedge\pi_3^3.
\end{gather*}

Cartan's test then establishes that systems of the form (\ref{systemsbeingstudied}) which satisfy the integrability conditions $D_ku_{ij}=D_iu_{kj},\;\textrm{for}\;1\leq i\neq j\neq k \leq3$, are involutive and their solutions are parametrized by three arbitrary functions of one variable.

\section{Structure equations and Lie bracket congruences}\label{appendixB}
\begin{Proposition}\label{dhstructureequations}The ${\rm d}_H$ structure equations for the Laplace adapted coframe \eqref{laplacecoframedef} for a system of the form~\eqref{systemsbeingstudied} of three hyperbolic equations with commuting characteristics are given by the following equations. Recall that $p_1=\min \{p_{21},p_{31}\}$, $p_2=\min \{p_{12},p_{32}\}$, and $p_3=\min \{p_{13},p_{23}\}$. Furthermore, $p_{k1}$, $p_{k2}$, and $p_{k3}$ will be used to denote which Laplace index achieves the minimum in the definitions of $p_1$, $p_2$, and $p_3$ respectively; whereas $p_{l1}$, $p_{l2}$, and~$p_{l3}$ will denote the Laplace index which does \textit{not} achieve the minimum in the definitions of~$p_1$,~$p_2$, and $p_3$. So, for example, if $p_1=p_{31}$ then $p_{k1}=p_{31}$ and $p_{l1}=p_{21}$,
\begin{gather}
{\rm d}_H\sigma_i =0,\\
\label{dhtheta}
{\rm d}_H\Theta =\sigma_1\wedge\big(\hat{\xi}_1^1-A_{i1}^i\Theta\big)+\sigma_2\wedge\big(\hat{\xi}_2^1-A_{i2}^i\Theta\big)+\sigma_3\wedge\big(\hat{\xi}_3^1-A_{i3}^i\Theta\big),\\
\label{dhxi11}
{\rm d}_H\hat{\xi}_1^1 =\sigma_1\wedge\big(\hat{\xi}_1^2-\big(A_{i1}^i\big)^1\hat{\xi}_1^1\big)+\sigma_2\wedge\big(H_{21}^0\Theta-A_{21}^1\hat{\xi}_1^1\big)+\sigma_3\wedge\big(H_{31}^0\Theta-A_{31}^1\hat{\xi}_1^1\big),\\
\nonumber
\cdots\cdots\cdots\cdots\cdots\cdots\cdots\cdots\cdots\cdots\cdots\cdots\cdots\cdots\cdots\cdots\cdots\cdots\cdots\cdots\cdots\cdots\cdots\\
{\rm d}_H\hat{\xi}_1^n =\sigma_1\wedge\big(\hat{\xi}_1^{n+1}-\mathcal{X}_1^n\big(A_{i1}^i\big)\hat{\xi}_1^n\big)+\sigma_2\wedge\big(H_{21}^{n-1}\hat{\xi}_1^{n-1}-\mathcal{X}_1^{n-1}\big(A_{21}^1\big)\hat{\xi}_1^n\big)\nonumber\\
\hphantom{{\rm d}_H\hat{\xi}_1^n =}{} +\sigma_3\wedge\big(H_{31}^{n-1}\hat{\xi}_1^{n-1}-\mathcal{X}_1^{n-1}\big(A_{31}^1\big)\hat{\xi}_1^n\big)\qquad \textrm{for all}\quad n,\quad 2\leq n\leq p_{k1},\label{dhxi1nfornbetween2andminpi1}\\
{\rm d}_H\hat{\xi}_1^n =\sigma_1\wedge\big(\hat{\xi}_1^{n+1}-\big(A_{i1}^i\big)^n\hat{\xi}_1^n\big)+\sigma_k\wedge\big({-}\big(A_{k1}^1\big)^{n-1}\hat{\xi}_1^n\big)\nonumber\\
\hphantom{{\rm d}_H\hat{\xi}_1^n =}{} +\sigma_l\wedge\big(H_{l1}^{n-1}\hat{\xi}_1^{n-1}-\big(A_{l1}^1\big)^{n-1}\hat{\xi}_1^n\big)\qquad \textrm{for}\quad n=p_{k1}+1,\\
{\rm d}_H\hat{\xi}_1^n \equiv\sigma_1\wedge\big(\hat{\xi}_1^{n+1}-\big(A_{i1}^i\big)^n\hat{\xi}_1^n\big)+\sigma_k\wedge\big({-}\big(A_{k1}^1\big)^{p_{k1}}\hat{\xi}_1^n\big)\nonumber\\
\hphantom{{\rm d}_H\hat{\xi}_1^n \equiv}{}
 +\sigma_l\wedge\big(H_{l1}^{n-1}\hat{\xi}_1^{n-1}-\big(A_{l1}^1\big)^{n-1}\hat{\xi}_1^n\big)\quad \mod\big\{\hat{\xi}_1^{p_{k1}+1},\ldots,\hat{\xi}_1^{n-1}\big\}\nonumber\\
\hphantom{{\rm d}_H\hat{\xi}_1^n \equiv}{} \textrm{for}\quad p_{k1}+2\leq n\leq p_{l1},\\
{\rm d}_H\hat{\xi}_1^n \equiv\sigma_1\wedge\hat{\xi}_1^{n+1}+\sigma_k\wedge\big({-}\big(A_{k1}^1\big)^{p_{k1}}\hat{\xi}_1^n\big)+\sigma_l\wedge\big({-}\big(A_{l1}^1\big)^{n-1}\hat{\xi}_1^n\big)\nonumber\\
\hphantom{{\rm d}_H\hat{\xi}_1^n \equiv}{} \textrm{for} \quad n=p_{l1}+1=p_1+1\quad \mod\big\{\hat{\xi}_1^{p_{k1}+1},\ldots,\hat{\xi}_1^{n-1}\big\},\label{dhxi1nfornequaltop1plus1}\\
{\rm d}_H\hat{\xi}_1^n \equiv\sigma_1\wedge\hat{\xi}_1^{n+1}+\sigma_k\wedge\big({-}\big(A_{k1}^1\big)^{p_{k1}}\hat{\xi}_1^n)+\sigma_l\wedge\big({-}\big(A_{l1}^1\big)^{p_{l1}}\hat{\xi}_1^n\big)\nonumber\\
\hphantom{{\rm d}_H\hat{\xi}_1^n \equiv}{} \textrm{for} \quad n\geq p_1+2\quad \mod\big\{\hat{\xi}_1^{p_{k1}+2},\ldots,\hat{\xi}_1^{n-1}\big\},\label{dhxi1nforngreaterthanp1plus2}
\end{gather}
where $p_{ij}$ is the smallest integer such that $H_{ij}^{p_{ij}}=0$. Similar formulas hold for the forms~$\hat{\xi}_2^i$ and~$\hat{\xi}_3^i$:
\begin{gather*}
{\rm d}_H\hat{\xi}_2^1 =\sigma_1\wedge\big(H_{12}^0\Theta-A_{12}^2\hat{\xi}_2^1\big)+\sigma_2\wedge\big(\hat{\xi}_2^2-\big(A_{i2}^i\big)^1\hat{\xi}_2^1\big)+\sigma_3\wedge\big(H_{32}^0\Theta-A_{32}^2\hat{\xi}_2^1\big), \\
\cdots\cdots\cdots\cdots\cdots\cdots\cdots\cdots\cdots\cdots\cdots\cdots\cdots\cdots\cdots\cdots\cdots\cdots\cdots\cdots\cdots\cdots\cdots \\
{\rm d}_H\hat{\xi}_2^n =\sigma_1\wedge\big(H_{12}^{n-1}\hat{\xi}_2^{n-1}-\big(A_{12}^2\big)^{n-1}\hat{\xi}_2^n\big)+\sigma_2\wedge\big(\hat{\xi}_2^{n+1}-\big(A_{i2}^i\big)^n\hat{\xi}_2^n\big) \\
 \hphantom{{\rm d}_H\hat{\xi}_2^n =}{} +\sigma_3\wedge\big(H_{32}^{n-1}\hat{\xi}_2^{n-1}-\big(A_{32}^2\big)^{n-1}\hat{\xi}_2^n\big)\qquad \textrm{for} \quad 2\leq n\leq p_{k2},\\
{\rm d}_H\hat{\xi}_2^n =\sigma_k\wedge\big({-}\big(A_{k2}^2\big)^{n-1}\hat{\xi}_2^n\big)+\sigma_2\wedge\big(\hat{\xi}_2^{n+1}-\big(A_{i2}^i\big)^n\hat{\xi}_2^n\big)+\sigma_l\wedge\big(H_{l2}^{n-1}\hat{\xi}_2^{n-1}-\big(A_{l2}^2\big)^{n-1}\hat{\xi}_2^n\big) \\
 \hphantom{{\rm d}_H\hat{\xi}_2^n =}{} \textrm{for} \quad n=p_{k2}+1, \\
{\rm d}_H\hat{\xi}_2^n \equiv\sigma_k\wedge\big({-}\big(A_{k2}^2\big)^{p_{k2}}\hat{\xi}_2^n\big)+\sigma_2\wedge\big(\hat{\xi}_2^{n+1}-\big(A_{i2}^i\big)^n\hat{\xi}_2^n\big)
+\sigma_l\wedge\big(H_{l2}^{n-1}\hat{\xi}_2^{n-1}-\big(A_{l2}^2\big)^{n-1}\hat{\xi}_2^n \big)\\
\hphantom{{\rm d}_H\hat{\xi}_2^n \equiv}{} \mod\big\{\hat{\xi}_2^{p_{k2}+1},\ldots,\hat{\xi}_2^{n-1}\big\}\qquad \textrm{for}\quad p_{k2}+2\leq n\leq p_{l2}, \\
{\rm d}_H\hat{\xi}_2^n \equiv\sigma_k\wedge\big({-}\big(A_{k2}^2\big)^{p_{k2}}\hat{\xi}_2^n\big)+\sigma_2\wedge\hat{\xi}_2^{n+1}+\sigma_l\wedge\big({-}\big(A_{l2}^2\big)^{n-1}\hat{\xi}_2^n\big) \\
\hphantom{{\rm d}_H\hat{\xi}_2^n \equiv}{} \mod\big\{\hat{\xi}_2^{p_{k2}+1},\ldots,\hat{\xi}_2^{n-1}\big\}\qquad \textrm{for} \quad n\geq p_{l2}+1=p_2+1
\end{gather*}
and
\begin{gather*}
{\rm d}_H\hat{\xi}_3^1 =\sigma_1\wedge\big(H_{13}^0\Theta-A_{13}^3\hat{\xi}_3^1\big)+\sigma_2\wedge\big(H_{23}^0\Theta-A_{23}^3\hat{\xi}_3^1\big)+\sigma_3\wedge\big(\hat{\xi}_3^2-\big(A_{i3}^i\big)^1\hat{\xi}_3^1\big), \\
\cdots\cdots\cdots\cdots\cdots\cdots\cdots\cdots\cdots\cdots\cdots\cdots\cdots\cdots\cdots\cdots\cdots\cdots\cdots\cdots\cdots\cdots\cdots \\
{\rm d}_H\hat{\xi}_3^n =\sigma_1\wedge\big(H_{13}^{n-1}\hat{\xi}_3^{n-1}-\big(A_{13}^3\big)^{n-1}\hat{\xi}_3^n\big)+\sigma_2\wedge\big(H_{23}^{n-1}\hat{\xi}_3^{n-1}-\big(A_{23}^3\big)^{n-1}\hat{\xi}_3^n\big) \\
\hphantom{{\rm d}_H\hat{\xi}_3^n =}{} +\sigma_3\wedge\big(\hat{\xi}_3^{n+1}-\big(A_{i3}^i\big)^n\hat{\xi}_3^n\big)\qquad \textrm{for} \quad 2\leq n\leq p_{k3}, \\
{\rm d}_H\hat{\xi}_3^n =\sigma_k\wedge\big({-}\big(A_{k3}^3\big)^{n-1}\hat{\xi}_3^n\big)+\sigma_l\wedge\big(H_{l3}^{n-1}-\big(A_{l3}^3\big)^{n-1}\hat{\xi}_3^n\big)+\sigma_3\wedge\big(\hat{\xi}_3^{n+1}-\big(A_{i3}^i\big)^n\hat{\xi}_3^n\big) \\
 \hphantom{{\rm d}_H\hat{\xi}_3^n =}{} \textrm{for}\quad n=p_{k3}+1, \\
{\rm d}_H\hat{\xi}_3^n \equiv\sigma_k\wedge\big({-}\big(A_{k3}^3\big)^{p_{k3}}\hat{\xi}_3^n\big)+\sigma_l\wedge\big(H_{l3}^{n-1}\hat{\xi}_3^{n-1}-\big(A_{l3}^3\big)^{n-1}\hat{\xi}_3^n\big)+\sigma_3\wedge \big(\hat{\xi}_3^{n+1}-\big(A_{i3}^i\big)^n\hat{\xi}_3^n\big) \\
\hphantom{{\rm d}_H\hat{\xi}_3^n \equiv}{} \mod\big\{\hat{\xi}_3^{p_{k3}+1},\ldots,\hat{\xi}_3^{n-1}\big\}\qquad \textrm{for} \quad p_{k3}+2\leq n\leq p_{l3}, \\
{\rm d}_H\hat{\xi}_3^n \equiv\sigma_k\wedge\big({-}\big(A_{k3}^3\big)^{p_{k3}}\hat{\xi}_3^n\big)+\sigma_l\wedge\big({-}\big(A_{l3}^3\big)^{n-1}\hat{\xi}_3^n\big)+\sigma_3\wedge\hat{\xi}_3^{n+1} \\
\hphantom{{\rm d}_H\hat{\xi}_3^n \equiv}{} \mod\big\{\hat{\xi}_3^{p_{k3}+1},\ldots,\hat{\xi}_3^{n-1}\big\}\qquad \textrm{for} \quad n\geq p_{l3}+1=p_3+1.
\end{gather*}
\end{Proposition}

\begin{proof}The structure equations ${\rm d}_H\sigma_i=0$ follow from the assumption that the characteristic vector fields $X_i$ all commute, along with Definition~\ref{totalverticaldef} which implies that $({\rm d}_V\sigma_i)(X_i,X_j)=0$. Thus
\begin{gather*}({\rm d}_H\sigma_i)(X_i,X_j)=({\rm d}\sigma_i)(X_i,X_j)=-\sigma_i([X_i,X_j])=0\end{gather*}
for any pair of $X_i$ and $X_j$. The other structure equations are computed by using the formula
\begin{gather}\label{dhomegaformula}
{\rm d}_H\omega=\sigma_1\wedge X_1(\omega)+\sigma_2\wedge X_2(\omega)+\sigma_3\wedge X_3(\omega).
\end{gather}
The formula for ${\rm d}_H\Theta$ follows immediately from the definitions of $\hat{\xi}_1^1$, $\hat{\xi}_1^2$, and $\hat{\xi}_1^3$. The remaining formulas are obtained by using equations (\ref{xinj})--(\ref{xixijpijplusn}) in formula~(\ref{dhomegaformula}). For example, for ${\rm d}_H\hat{\xi}_1^1$ we use equation~(\ref{xinj}) with $n=2$ and $j=1$ to see that \begin{gather*}\sigma_1\wedge X_1\big(\hat{\xi}_1^1\big)=\sigma_1\wedge\big(\hat{\xi}_1^2-\big(A_{i1}^i\big)^1\hat{\xi}_1^1\big).\end{gather*} Equation~(\ref{xixij1}) with $i=2$ and $j=1$ shows that \begin{gather*}\sigma_2\wedge X_2\big(\hat{\xi}_1^1\big)=\sigma_2\wedge\big(H_{21}^0\Theta-\big(A_{21}^1\big)\hat{\xi}_1^1\big)\end{gather*} and likewise, (\ref{xixij1}) with $i=3$ and $j=1$ gives $\sigma_3\wedge X_3\big(\hat{\xi}_1^1\big)=\sigma_3\wedge\big(H_{31}^0\Theta-\big(A_{31}^1\big)\hat{\xi}_1^1\big)$. Similar substitutions yields the~${\rm d}_H$ structure formulas (\ref{dhxi1nfornbetween2andminpi1})--(\ref{dhxi1nfornequaltop1plus1}) above. Finally, formula~(\ref{dhxi1nforngreaterthanp1plus2}) is proved by induction on~$n$, beginning with
\begin{gather*}
X_k\big(\hat{\xi}_1^{n+1}\big) =X_k\big(X_1\big(\hat{\xi}_1^n\big)\big) =X_1\big(X_k\big(\hat{\xi}_1^n\big)\big)\nonumber\\
\hphantom{X_k\big(\hat{\xi}_1^{n+1}\big)}{} \equiv X_1\big({-}\big(A_{k1}^1\big)^{p_{k1}}\hat{\xi}_1^n\big)\quad \mod\big\{\hat{\xi}_1^{p_{k1}+2},\ldots,\hat{\xi}_1^{n-1}\big\},\qquad \textrm{for}\quad k=2,3,
\end{gather*}
where the above equalities are obtained by making use of equation~(\ref{xixijpijplusn}). The sets of~${\rm d}_H$ structure formulas for $\hat{\xi}_2^n$ and $\hat{\xi}_3^n$ are derived via an analogous procedure, again using the appropriate versions of equations (\ref{xinj})--(\ref{xixijpijplusn}) in formula~(\ref{dhomegaformula}).
\end{proof}

The Lie bracket congruences for the characteristic vector fields~$X_i$ with the vertical vector fields~$U$ and $V_k^l$ defined by \begin{gather*}\Theta(U)=1,\qquad \Theta\big(V_k^l\big)=0,\qquad \hat{\xi}_j^n(U)=0,\qquad \hat{\xi}_j^n\big(V_k^l\big)=\delta_k^j\delta_l^n\end{gather*} are given in the following proposition.

\begin{Proposition}\label{liebracketcongruences}The following Lie bracket congruences hold modulo $\{X_1,X_2,X_3\}$:
\begin{gather}\label{x1withu}
[X_1,U] \equiv A_{i1}^iU-H_{12}^0V_2^1-H_{13}^0V_3^1\quad \mod\{X_1,X_2,X_3\},\\
\label{x1withv11}
\big[X_1,V_1^1\big] \equiv -U+\big(A_{i1}^i\big)^1V_1^1\quad \mod\{X_1,X_2,X_3\},\\
\label{x1withv1n}
\big[X_1,V_1^n\big] \equiv -V_1^{n-1}+\big(A_{i1}^i\big)^nV_1^n\quad \mod\{X_1,X_2,X_3\}\qquad \textrm{for}\quad 2\leq n\leq p_{l1},\\
\label{x1withv1n2}
\big[X_1,V_1^n\big] \equiv -V_1^{n-1}\quad \mod\{X_1,X_2,X_3\}\qquad \textrm{for}\quad n\geq p_{l1}+1,\nonumber\\
\big[X_1,V_2^1\big] \equiv A_{12}^2V_2^1-H^1_{12}V_2^2\quad \mod\{X_1,X_2,X_3\},\nonumber\\
\big[X_1,V_2^n\big] \equiv \big(A_{12}^2\big)^{n-1}V_2^n-H_{12}^nV_2^{n+1}\quad \mod\{X_1,X_2,X_3\}\qquad \textrm{for}\quad 2\leq n<p_{12},\nonumber\\
\big[X_1,V_2^n\big] \equiv \big(A_{12}^2\big)^{p_{12}}V_2^n\quad \mod\big\{X_1,X_2,X_3,V_2^{n+1},\ldots\big\}\qquad \textrm{for}\quad n\geq p_{12},\\
\label{x1withv31}
\big[X_1,V_3^1\big] \equiv A_{13}^3V_3^1-H^1_{13}V_3^2\quad \mod\{X_1,X_2,X_3\},\\
\label{x1withv3n}
\big[X_1,V_3^n\big] \equiv \big(A_{13}^3\big)^{n-1}V_3^n-H_{13}^nV_3^{n+1}\quad \mod\{X_1,X_2,X_3\}\qquad \textrm{for}\quad 2\leq n<p_{13},\\
\big[X_1,V_3^n\big] \equiv \big(A_{13}^3\big)^{p_{13}}V_3^n\quad \mod\big\{X_1,X_2,X_3,V_3^{n+1},\ldots\big\}\qquad \textrm{for}\quad n\geq p_{13},\nonumber\\
[X_2,U] \equiv A_{i2}^iU-H_{21}^0V_1^1-H_{23}^0V_3^1\quad \mod\{X_1,X_2,X_3\},\nonumber\\
\big[X_2,V_1^1\big] \equiv A_{21}^1V_1^1-H_{21}^1V_1^2\quad \mod\{X_1,X_2,X_3\},\nonumber\\
\big[X_2,V_1^n\big] \equiv \big(A_{21}^1\big)^{n-1}V_1^n-H_{21}^nV_1^{n+1}\quad \mod\{X_1,X_2,X_3\}\qquad \textrm{for}\quad 2\leq n<p_{21},\nonumber\\
\big[X_2,V_1^n\big] \equiv \big(A_{21}^1\big)^{p_{21}}V_1^n\quad \mod\big\{X_1,X_2,X_3,V_1^{n+1},\ldots\big\}\qquad \textrm{for}\quad n\geq p_{21},\nonumber\\
\big[X_2,V_2^1\big] \equiv -U+\big(A_{i2}^i\big)^1V_2^1\quad \mod\{X_1,X_2,X_3\},\nonumber\\
\big[X_2,V_2^n\big] \equiv -V_2^{n-1}+\big(A_{i2}^i\big)^nV_2^n\quad \mod\{X_1,X_2,X_3\}\qquad \textrm{for}\quad 2\leq n\leq p_{l2},\nonumber\\
\big[X_2,V_2^n\big] \equiv -V_2^{n-1}\quad \mod\{X_1,X_2,X_3\}\qquad \textrm{for}\quad n\geq p_{l2}+1,\nonumber\\
\big[X_2,V_3^1\big] \equiv A_{23}^3V_3^1-H_{23}^1V_3^2\quad \mod\{X_1,X_2,X_3\},\nonumber\\
\big[X_2,V_3^n\big] \equiv \big(A_{23}^3\big)^{n-1}V_3^n-H_{23}^nV_3^{n+1}\quad \mod\{X_1,X_2,X_3\}\qquad \textrm{for}\quad 2\leq n<p_{23},\nonumber\\
\big[X_2,V_3^n\big] \equiv \big(A_{23}^3\big)^{p_{23}}V_3^n\quad \mod\big\{X_1,X_2,X_3,V_3^{n+1},\ldots\big\}\qquad \textrm{for}\quad n\geq p_{23},\nonumber\\
[X_3,U] \equiv A_{i3}^iU-H_{31}^0V_1^1-H_{32}^0V_2^1\quad \mod\{X_1,X_2,X_3\},\nonumber\\
\big[X_3,V_1^1\big] \equiv A_{31}^1V_1^1-H_{31}^1V_1^2\quad \mod\{X_1,X_2,X_3\},\nonumber\\
\big[X_3,V_1^n\big] \equiv \big(A_{31}^1\big)^{n-1}V_1^n-H_{31}^nV_1^{n+1}\quad\mod\{X_1,X_2,X_3\}\qquad \textrm{for}\quad 2\leq n<p_{31},\nonumber\\
\big[X_3,V_1^n\big] \equiv \big(A_{31}^1\big)^{p_{31}}V_1^n\quad\mod\big\{X_1,X_2,X_3,V_1^{n+1},\ldots\big\}\qquad \textrm{for}\quad n\geq p_{31},\nonumber\\
\big[X_3,V_2^1\big] \equiv A_{32}^2V_2^1-H^1_{32}V_2^2\quad\mod\{X_1,X_2,X_3\},\nonumber\\
\big[X_3,V_2^n\big] \equiv \big(A_{32}^2\big)^{n-1}V_2^n-H_{32}^nV_2^{n+1}\quad\mod\{X_1,X_2,X_3\}\qquad \textrm{for}\quad 2\leq n<p_{32},\nonumber\\
\big[X_3,V_2^n\big] \equiv \big(A_{32}^2\big)^{p_{32}}V_2^n\quad\mod\big\{X_1,X_2,X_3,V_2^{n+1},\ldots\big\}\qquad \textrm{for}\quad n\geq p_{32},\nonumber\\
\big[X_3,V_3^1\big] \equiv -U+\big(A_{i3}^i\big)^1V_3^1\quad\mod\{X_1,X_2,X_3\},\nonumber\\
\big[X_3,V_3^n\big] \equiv -V_3^{n-1}+\big(A_{i3}^i\big)^nV_3^n\quad\mod\{X_1,X_2,X_3\}\qquad \textrm{for}\quad 2\leq n\leq p_{l3},\nonumber\\
\big[X_3,V_3^n\big] \equiv -V_3^{n-1}\quad\mod\{X_1,X_2,X_3\}\qquad \textrm{for}\quad n\geq p_{l3}+1.\nonumber
\end{gather}
\end{Proposition}

\begin{proof} If $X$ is a total vector field and $V$ is a $\pi_M^{\infty}$ vertical vector field on $\mathcal{R}^{\infty}$ and $\omega$ is a contact 1-form, then the following equation holds
\begin{gather*}
({\rm d}_H\omega)(X,V)=({\rm d}\omega)(X,V) =X(\omega(V))-V(\omega(X))-w([X,V]) =X(\omega(V))-\omega([X,V]),
\end{gather*}
since $\omega(X)=0$ by definition of $X$ being a total vector field, and ${\rm d}_V\omega(X,V)=0$ again by definition because ${\rm d}_V\omega$ is a~type $(0,2)$ contact form. Moreover, if $\omega(V)$ is a constant, then
\begin{gather*}
({\rm d}_H\omega)(X,V)=-\omega([X,V]).
\end{gather*}
This last equation, along with the ${\rm d}_H$ structure equations from Theorem~\ref{dhstructureequations}, allow us to obtain the above Lie bracket congruences. As an example, let's compute $[X_1,U]$. Taking the structure equation (\ref{dhtheta}), we see that $\Theta([X_1,U])= -({\rm d}_H\Theta)(X_1,U)=A_{i1}^i$. This tells us that the bracket $[X_1,U]$ has a $U$ component with coefficient~$A_{i1}^i$. Likewise, taking equation~(\ref{dhxi11}) shows that $\hat{\xi}_1^1([X_1,U])=-\big({\rm d}_H\hat{\xi}_1^1\big)(X_1,U)=0$, which implies that~$[X_1,U]$ has no $V_1^1$ component. Continue on in this way using each of the remaining ${\rm d}_H\hat{\xi}_i^j$ structure equations in turn in order to identify all the nonzero components of $[X_1,U]$. Repeat the same procedure with $[X_2,U]$, $[X_3,U]$, and $\big[X_i,V_k^l\big]$ to obtain all of the above Lie bracket congruences.
\end{proof}

\subsection*{Acknowledgements}

I would like to express my heartfelt gratitude to my Ph.D.~advisor, Professor Niky Kamran for his constant encouragement and guidance throughout the preparation of this paper. I would also like to thank Professor Mark Fels for the lectures he presented during the spring of 2012 at McGill University, which contributed greatly to my understanding of Cartan's structural classification of involutive overdetermined systems of PDEs. And for bringing to my attention the example presented in Section~\ref{universallinearization-section}, I~thank Professor Peter Vassiliou. Finally, I~thank the referees of this paper for their many thoughtful and helpful comments.

\pdfbookmark[1]{References}{ref}
\LastPageEnding

\end{document}